\documentclass[a4paper,dvipsnames]{article}

\usepackage{amsthm,amsmath,amssymb,graphicx,hyperref,enumitem,authblk}

\usepackage{algorithm}
\usepackage{algpseudocode}
\usepackage{cleveref}
\usepackage{tikz}
\usepackage{pgfplots}
\usepackage{standalone}
\usepackage[title]{appendix}
\usepackage{listings}

\theoremstyle{theorem}
\newtheorem{theorem}{Theorem}[section]
\newtheorem{proposition}[theorem]{Proposition}
\newtheorem{corollary}[theorem]{Corollary}
\newtheorem{lemma}[theorem]{Lemma}
\newtheorem{conjecture}[theorem]{Conjecture}
\newtheorem{question}[theorem]{Question}

\theoremstyle{definition}
\newtheorem{definition}[theorem]{Definition}
\newtheorem{remark}[theorem]{Remark}

\counterwithin{figure}{section}
\counterwithin{table}{section}

\crefname{conjecture}{conjecture}{conjectures}
\crefname{question}{question}{questions}

\usetikzlibrary{positioning}
\usetikzlibrary{decorations.markings}
\usetikzlibrary{arrows}

\tikzset{
  >=latex
}

\colorlet{MyRed}{red!100!}
\colorlet{MyGreen}{green!100!}
\colorlet{MyBlue}{blue!100!}
\colorlet{MyPurple}{Plum!100!}

\newcommand\ca{MyRed}
\newcommand\cb{MyGreen}
\newcommand\cc{MyBlue}
\newcommand\cd{MyPurple}

\newcommand{\tri}{\mathcal{T}}
\newcommand{\manifold}{\mathcal{M}}

\newcommand{\ZZ}{\mathbb{Z}}
\newcommand{\RR}{\mathbb{R}}

\newcommand{\Sym}{\operatorname{Sym}}
\DeclareMathOperator{\crit}{crit}

\DeclareMathOperator{\branch}{branch}
\DeclareMathOperator{\rank}{rank}

\newcommand{\id}{\textrm{id}}

\pgfdeclareplotmark{v}{%
  \pgfpathmoveto{\pgfpoint{-0.7\pgfplotmarksize}{\pgfplotmarksize}}
  \pgfpathlineto{\pgfpoint{0pt}{0pt}}
  \pgfpathlineto{\pgfpoint{0.7\pgfplotmarksize}{\pgfplotmarksize}}
  \pgfusepathqstroke
}
\pgfdeclareplotmark{^}{%
  \pgfpathmoveto{\pgfpoint{-0.7\pgfplotmarksize}{-\pgfplotmarksize}}
  \pgfpathlineto{\pgfpoint{0pt}{0pt}}
  \pgfpathlineto{\pgfpoint{0.7\pgfplotmarksize}{-\pgfplotmarksize}}
  \pgfusepathqstroke
}

\begin{document}

\title{Vertex Bounds in Triangulated $d$-Manifolds and an Application to 4-Manifold Complexity}

\author[1]{Jonathan Spreer}
\author[1]{Lucy Tobin}

\affil[1]{{\small School of Mathematics and Statistics, University of Sydney, Australia, \texttt{jonathan.spreer@sydney.edu.au}, \texttt{lucy.tobin@sydney.edu.au}}}

\date{}

\maketitle

\begin{abstract}
  We investigate face numbers of generalised triangulations of manifolds in arbitrary dimensions. This is motivated by the study of connections between the combinatorics of triangulations and topological properties of their underlying manifolds.   
  For an $n$-facet triangulation of an odd-dimensional $d$-manifold with $n \geq d$, we prove that the number of vertices satisfies $v \leq n + \frac{d - 1}{2}$. Moreover, we show that this bound is tight for all odd $d$ and all $n \geq d$. For even dimensions, we conjecture the bound $v \leq \frac{n}{2} + d$. We prove that, if true, the bound is tight. 
    Although we cannot prove the conjecture for arbitrary generalised triangulations, we prove it in the case of balanced triangulations, and provide a sufficient condition on the dual graph of the triangulation for it to hold. Furthermore, we construct families of $d$-dimensional triangulations with singularities that exceed the bound $v > \frac{n}{2} + d$, thereby demonstrating that the manifold condition is necessary for the validity of the conjecture.
    Our study is further motivated by an application to triangulations of simply connected $4$-manifolds: For $d=4$, our conjecture implies that a triangulation $\tri$ of a simply connected 4-manifold $\manifold$ with $n$ pentachora satisfies $2\beta_2(\manifold) \leq n$ as a lower bound on its Matveev complexity -- a bound that is known to be best possible up to a very small additive constant.
\end{abstract}



\section{Introduction}
\label{sec:intro}

Piecewise linear (PL) manifolds are often studied by decomposing them into simplices. These {\em triangulations} enable the use of combinatorial and algorithmic methods to study topological properties of the underlying space. Several distinct notions of triangulations exist. In combinatorial topology, a {\em combinatorial manifold} \cite{Lutz03TrigMnfFewVertVertTrans} is defined as a simplicial complex representing a PL manifold. This concept generalises the boundary complex of a simplicial $(d+1)$-polytope. In contrast, low-dimensional topology often employs \emph{generalised triangulations} -- which are the focus of this work. Generalised triangulations are constructed by gluing simplices along their faces in pairs, allowing even faces of the same simplex to be identified, see \Cref{sec:TRIG} for a precise definition. Every combinatorial manifold is a generalised triangulation of the same PL manifold, and the second derived subdivision of a generalised triangulation is always a combinatorial manifold. Since we work exclusively within the PL category, we will henceforth omit the prefix {\em PL} when referring to manifolds.

A central combinatorial invariant of a $d$-dimensional triangulation $\tri$ is its sequence of face numbers $f_i$, $0 \leq i \leq d$, collectively known as the \emph{$f$-vector} $f(\tri) = (f_0, \ldots, f_d)$. A natural question to ask is: Which $f$-vectors can be observed for $\tri$ under the condition that $\tri$ triangulates a given $d$-manifold $\manifold$? This question originated in the context of simplicial polytopes with McMullen's {\em $g$-conjecture} \cite{McMullen71NumbersFacesSimplPoly}, later proved by Billera and Lee \cite{Billera81ProofSuffMcMullenCondFVec}, and Stanley \cite{Stanley80NumFacesSimplConvPoly}. 
Since then, the study of $f$-vectors has been extended from simplicial polytopes to combinatorial manifolds representing the $d$-sphere (and even more general combinatorial objects). A major breakthrough in this direction was achieved by Adiprasito, and subsequently by Papadakis and Petrotou \cite{Adiprasito18GTheorem,Papdakis20GTheorem,Adiprasito21GTheorem}, who established the $g$-conjecture in this broader setting.

The $f$-vectors of combinatorial manifolds other than spheres have also been extensively studied; see, for example, \cite{Kuehnel1995,Lutz08FVec3Mnf,Murai14,Novik05OnFNumMnfSymm,Spreer14CyclicCombMflds}. However, in this more general context, our understanding remains limited, and many fundamental questions remain open. One such question is:

\medskip

{\em \centerline{Given a piecewise linear manifold $\manifold$, what is the smallest number of faces}}
{\em \centerline{required to triangulate $\manifold$?}}

\medskip

For combinatorial manifolds, this question is typically asked about vertices, which in turn provides an upper bound on its number of faces in any dimension; see, for instance, \cite[Section 4]{Kuehnel1995}. In contrast, for generalised triangulations -- which may only contain a single vertex -- the focus is usually on the number of top-dimensional faces. This minimum is known as the \emph{(Matveev) complexity} of the manifold $\manifold$ \cite{jaco09-minimal-lens,matveev90-complexity}. In both frameworks, only finitely many manifolds can have complexity or vertex number bounded above by a given constant $c > 0$. This finiteness follows from the elementary observation that a finite number of simplices can be glued together in only finitely many distinct ways.

Vertex numbers and the complexity of $2$-dimensional manifolds are known for both combinatorial manifolds and generalised triangulations. In dimension $3$, the situation is more intricate, but a variety of techniques and a substantial body of literature exist for studying the complexity of $3$-manifolds in both the generalised \cite{Cha-complexity-2016,ishikawa_construction_2016,jaco_2thurston_2020,jaco2022complexity,jaco09-minimal-lens,jaco_mathbb_2013,Lackenby19Fibred,matveev90-complexity} and the simplicial \cite{Kuehnel1995,Lutz08FVec3Mnf} settings.

In dimension $4$, much of the focus has been on triangulations of simply connected $4$-manifolds. A lower bound on the number of vertices in a combinatorial triangulation, expressed in terms of the second Betti number, was established by K\"uhnel in \cite[Section 4B]{Kuehnel1995}. This bound is known to be sharp in only a few cases \cite{Casella01TrigK3MinNumVert,Kuehnel83The9VertComplProjPlane,Kuehnel83Uniq3Nb4MnfFewVert,Spreer09CombPropsOfK3}. In contrast, for \emph{crystallisations} -- a subclass of {\em balanced triangulations} or {\em graph encoded manifolds} that lies somewhat between combinatorial and generalised triangulations (see \Cref{sub:balanced}) -- provably minimal triangulations are known for every connected sum of $\mathbb{S}^2 \times \mathbb{S}^2$, $\mathbb{C}P^2$, and the K3 surface~\cite{Basak14Crystallizations}. 

However, the truly minimal number of top-dimensional simplices required to triangulate a simply connected $4$-manifold is captured by its complexity in the generalised setting. Determining this quantity appears to be closely linked to understanding the face numbers of generalised triangulations of $4$-manifolds, specifically, to the question: What is the maximal number of vertices that a triangulation of a given $4$-manifold with $n$ top-dimensional simplices can have? This question is of independent interest and is completely resolved for combinatorial manifolds via the celebrated Lower Bound Theorem, due to Barnette and Kalai. For a combinatorial $d$-manifold with $n$ top-dimensional faces and $v$ vertices, the theorem asserts, among other things, that
\[
n \geq d \cdot v - (d+2)(d-1),
\]
as shown in \cite{Barnette73ProofLBCConvPoly,Kalai87RigidityLBT}.

Addressing this question in the context of generalised triangulations serves as the principal motivation for this work. Our answers uncover an intriguing mix between straightforward results, essentially based on Euler characteristic computations and straightforward constructions, and what appears to be more difficult combinatorial questions with intricate links to the topology of the underlying spaces. 

\paragraph*{Outline of paper and contributions.}

Let $\tri$ be a generalised triangulation of a $d$-manifold $\manifold$ with $f$-vector $f(\tri) = (f_0,f_1,\ldots ,f_d)$. In \Cref{sec:higherDims}, we show for $\tri$ a generalised triangulation of an arbitrary odd-dimensional $d$-manifold that
\[f_0 \le f_d + \frac{d-1}{2} \]
holds for $f_d \geq d$ and is sharp, as established in \Cref{prop:construction-odd-large,prop:odd}. A summary of our upper and lower bounds for small numbers of facets in dimension $7$ is illustrated in \Cref{fig:bound-plot}.

In \Cref{sec:sufficient-conditions}, we show that the situation for triangulations $\tri$ in even dimensions $d$ is markedly different. We conjecture that
\[f_0 \le \frac{f_d}{2}+d\]
(\Cref{conj:vertex-bound-even}). The remainder of this section, as well as \Cref{sec:pseudomanifolds} explores this conjecture. In particular, we prove \Cref{conj:vertex-bound-even} for balanced triangulations, and provide a full classification of the cases of equality in this setting, see \Cref{thm:vertex-bound-gem}. We furthermore prove that \Cref{conj:vertex-bound-even} holds when the dual graph $\Gamma(\tri)$ of $\tri$ has small {\em branching number} -- that is, when it admits a map onto a tree with few leaves (see \Cref{thm:branch-bound}).

Finally, in \Cref{sec:pseudomanifolds}, we show that the structure of the dual graph $\Gamma(\tri)$ alone is insufficient to prove \Cref{conj:vertex-bound-even}. We construct generalised $d$-dimensional triangulations with singularities that violate the conjectured bound. Among other infinite families in arbitrary even dimensions, we exhibit triangulations of 4-dimensional pseudomanifolds with singularities in their edge links, in which the number of vertices asymptotically approaches $\frac{3}{4} \cdot f_4$.

All this work is motivated by a case study in \Cref{sec:facenumbers}, where we show that, for $d=4$, and $\manifold$ simply connected, any bound $f_0 \leq a f_4 + b$, $a,b \in \mathbb{R}$, $a>0$, results in a lower bound
\[ f_4 \geq \frac{1}{4a+1} \left( 6 \beta_2 (\manifold) + 12-4b \right) \]
on the Matveev complexity of $\manifold$, see \Cref{cor:general}. Moreover, we show that, if our main conjecture holds for $d=4$, then 
\[ 2 \beta_2(\manifold) \leq f_4, \]
for $\beta_2(\manifold)$ the second Betti number of $\manifold$, see \Cref{cor:actualBound}. This would be an exciting result since {\em (a)} bounds on Matveev complexity are typically hard to establish, and {\em (b)} this bound would be tight for all Betti numbers and up to a small additive constant of $2$ due to recent work by the authors in \cite{SpreerTobin2025-SmallTriangulations} and, independently, Burke~\cite{Burke23}.

\paragraph*{Acknowledgements.}

The work of Spreer is partially supported by the Australian Research Council under the Discovery Project scheme (grant number DP220102588). Part of this work was done while both authors were visiting Technische Universit\"at Berlin. Our special thanks go to Michael Joswig and his research group for their support and hospitality. We also want to thank the Discrete Geometry group at Freie Universt\"at Berlin for additional support.

\section{Background}

\subsection{Manifolds}

Throughout this article, we assume a general familiarity with manifolds in the topological and piecewise linear (PL) categories. In dimension four, we focus on closed (orientable) piecewise linear manifolds with trivial fundamental group. These manifolds are famously classified as having topological (not necessarily PL-topological) types of connected sums of $\mathbb{C}P^2$, $\mathbb{S}^2 \times \mathbb{S}^2$ and the $K3$ surface (modulo the $11/8$-conjecture \cite{Matsumoto82ElevenEightConj}). 

For more background reading on manifolds we refer the reader to \cite{Gompf,Saveliev2011-3Manifolds}.

\subsection{Graphs}

In this paper, we shall use the term \textit{graph} to mean an undirected graph allowing loops and multiple edges. That is, a graph $\Gamma$ consists of a set of \textit{nodes} $V(\Gamma)$ and a multiset of \textit{arcs} $E(\Gamma)$ taking the form of an unordered pair $e=\langle v,v' \rangle$ with \textit{endpoints} $v,v' \in V(\Gamma)$. An arc is called a \textit{loop} if both of its endpoints are the same node. We consistently use the terms \textit{nodes} and \textit{arcs} over the more familiar \textit{vertices} and \textit{edges} to avoid confusion with the vertices and edges of a triangulation. We also use a number of standard graph theory terms such as {\em neighbours}, \textit{paths}, {\em connectedness}, {\em connected component} and \textit{$k$-regular graphs} -- for a standard reference see \cite{BondyMurty2008}. Note that, when determining vertex degrees and regularity, a loop arc counts as two arcs for its node. For example, a $5$-regular multigraph may include a node $v$ which is the endpoint of one loop arc and three other non-loop arcs.

\subsection{Triangulations}
\label{sec:TRIG}

The triangulations we work with are known as \textit{generalised triangulations}.

\begin{definition}
  A \textit{(generalised) triangulation} $\tri$ of dimension $d$ is a set of abstract $d$-simplices together with a set of \textit{gluings}: pairs of $(d-1)$-faces which are identified affinely, determined by a pairing of their vertices. A gluing cannot pair a $(d-1)$-face with itself (but it can pair two faces of the same simplex) and each $(d-1)$-face can only be used in one gluing.
  
  Such a triangulation determines an underlying topological space, the identification space obtained by performing these gluings, which we denote by $|\tri|$. We refer to $\tri$ as a \textit{triangulation of} this space and understand the triangulation both as a purely combinatorial object, and as a representation of this space.

  We call a triangulation \textit{closed} if every $(d-1)$-face is used in a gluing, and \textit{connected} if $|\tri|$ is connected.

  A \textit{face} of a triangulation $\tri$ is a face after the identifications, and we call the number of original faces of the abstract $d$-simplices which were glued together to form it, its \textit{degree}. The $d$-faces are called \textit{facets}, and the $(d-1)$-faces \textit{ridges}. The \textit{face-vector} or \textit{$f$-vector} of $\tri$ is $f(\tri) = (f_0,f_1,\dots,f_d)$, where $f_i$ is the number of faces of dimension $i$. The set of $i$-dimensional faces of $\tri$ is denoted by $\tri^{(i)}$.
\end{definition}

Generalised triangulations allow more freedom than, for instance, combinatorial manifolds, described by simplicial complexes (which do not allow two ridges of the same simplex to be identified, or two simplices to be joined across multiple ridges). They are also useful for computations: they often allow a manifold to be triangulated using significantly fewer $d$-simplices than would be required for a simplicial complex, and can still be encoded conveniently, as each gluing can be represented by a single permutation. However, we do need to be careful when applying results originally stated for combinatorial manifolds to generalised triangulations.

For convenience of encoding, we typically assume that before gluing the vertices of each facet are labelled by $\{0,1,\dots,d\}$. We can then encode a gluing as a tuple $(\Delta_1,\Delta_2,i,\sigma)$, where $\Delta_1$ and $\Delta_2$ are the two facets, $i$ is the vertex of $\Delta_1$ not involved in the gluing, and $\sigma \in \Sym_{d+1}$ is a permutation such that $\sigma(i)$ is the vertex of $\Delta_2$ not involved in the gluing and $\sigma|_{\{0,1,\dots,d-1\}\setminus\{i\}}$ specifies the pairing from vertices of $\Delta_1$ to $\Delta_2$.

When working with triangulations, we generally consider them up to \textit{combinatorial equivalence}. That is, a bijection between their faces which respects the gluing structure (a \textit{relabelling} of the faces).

In this work, we are primarily interested in triangulations with underlying space a PL manifold. We can use the following result to enforce this.

\begin{theorem}[see \cite{Burton2012-Regina}]
   Let $\tri$ be a $d$-dimensional triangulation with none of its faces (of any dimension) being identified with themselves under a non-identity relation (for example, and edge identified with itself in reverse). Then $|\tri|$ is:
\begin{enumerate}[label=(\alph*)] 
  \item a closed PL $d$-manifold if $\tri$ is closed, and the link of every vertex is a PL $(d-1)$-sphere;
  \item a PL $d$-manifold with boundary if some of its ridges are not used in any gluing (boundary ridges), and the link of every vertex is a PL $(d-1)$-ball (for boundary vertices) or $(d-1)$-sphere (for interior vertices).
\end{enumerate}
\label{thm:TRIG:vertex-links} 
\end{theorem}

Triangulations that do not meet the preconditions of \Cref{thm:TRIG:vertex-links} are sometimes referred to as {\em invalid}. See also \Cref{def:pseudo} for a hierarchy of types of triangulations. 

Note here that the link of a face in a triangulation is intuitively the same as for simplicial complexes, but it must be defined slightly differently. The {\em link} of an $i$-dimensional face $\sigma$ of $\tri$ is a triangulation of a $(d-i-1)$-dimensional space ``around'' $\sigma$ that is constructed as follows: For every facet $\Delta$ containing $\sigma$, we take a copy of the $(d-i-1)$-face $\tau \subset \Delta$ opposite to $\sigma$, moved to sufficiently close proximity of $\sigma$ in $\Delta$. These copies are glued together according to the face gluings containing $\sigma$ and the $(d-i-2)$-faces of $\tau$. If $\sigma$ is a vertex, its link is the boundary of a small neighbourhood of $\sigma$ in $\tri$, if $\sigma$ is a ridge, its link is a collection of points, one for each facet containing $\sigma$. If $\sigma$ is a triangle in a triangulation of a $4$-manifold, its link is a cycle of length the degree of $\sigma$.

To help visualise the structure of a triangulation $\tri$ (particularly in dimensions $d \ge 3$) it is useful to consider its \textit{dual graph} $\Gamma(\tri)$ -- the graph with a node for each $d$-simplex and an arc for each gluing, linking two $d$-simplices together. We can immediately observe that $\tri$ is closed if and only if $\Gamma(\tri)$ is $(d+1)$-regular, and connected if and only if $\Gamma(\tri)$ is connected.

Finally, we need to make use of {\em elementary moves} on triangulations -- local modifications which change the triangulation, but preserve the underlying PL type. The most familiar of these (both for simplicial complexes and generalised triangulations) are the so-called {\em Pachner} or {\em bistellar} moves, see, for instance, \cite{Pachner1991}. We, however, need two slightly less familiar moves, the {\em 0-2 vertex move} and its inverse the {\em 2-0 vertex move}. The 0-2 vertex move takes a ridge of the triangulation and ``inflates'' it, replacing it with two facets glued together along all but one of their ridges by identity permutations. The two remaining ridges take the place of the original ridge in the two facets which contained it (or, if we started with a boundary ridge, one remains unglued to take its place on the boundary). The 2-0 vertex move, conversely, takes two facets glued by the identity along all but one of their faces and ``flattens'' them to a single ridge.

\begin{figure}
\centering
\resizebox{0.7\linewidth}{!}{
\begin{tikzpicture}[every node/.style={circle, inner sep=0, outer sep=0, minimum width=0.1cm, fill=black}, every path/.style={thick}, label distance=2mm, on grid]



\node[fill=none] (c2) at (0,0) {} ;
\node (l2) [fill=none, left=of c2] {} ;
\node (r2) [fill=none, right=of c2] {} ;

\node (t2) [above=1.5 of c2] {} ;
\node (b2) [below=1.5 of c2] {} ;

\path (l2) arc (180:360:1 and 0.4) node[pos=0,auto,anchor=center] (sl2) {} node[pos=1,auto,anchor=center] (sr2) {} ;
\path (l2) arc (180:0:1 and 0.5) node[pos=0.6,auto,anchor=center] (sm2) {} ;

\draw (l2) -- (r2) ;
\draw (l2) -- (sm2) ;
\draw (sm2) -- (r2) ;

\draw (l2) to[bend left=20] (t2) ;
\draw[dashed] (t2) to[bend left=10] (sm2) ;
\draw (r2) to[bend right=20] (t2) ;

\draw (l2) to[bend right=20] (b2) ;
\draw[dashed] (sm2) to[bend left=5] (b2) ;
\draw (r2) to[bend left=20] (b2) ;


\node[fill=none,outer sep=0.5] (altop) at (1.5,0.05) {} ;
\node[fill=none,outer sep=0.5] (artop) at (2.5,0.05) {} ;

\node[fill=none,outer sep=0.5] (albot) at (1.5,-0.05) {} ;
\node[fill=none,outer sep=0.5] (arbot) at (2.5,-0.05) {} ;

\draw[->,black] (altop) -- (artop) {} ;
\draw[->,black] (arbot) -- (albot) {} ;

\node[fill=none] at (2,0.3) {\small{0-2 (vertex)}} ;
\node[fill=none] at (2,-0.3) {\small{2-0 (vertex)}} ;


\node[fill=none] (c1) at (4,0) {} ;
\node (l1) [fill=none, left=of c1] {} ;
\node (r1) [fill=none, right=of c1] {} ;

\node (t1) [above=1.5 of c1] {} ;
\node (b1) [below=1.5 of c1] {} ;

\path (l1) arc (180:360:1 and 0.4) node[pos=0,auto,anchor=center] (sl1) {} node[pos=1,auto,anchor=center] (sr1) {} ;
\path (l1) arc (180:0:1 and 0.5) node[pos=0.6,auto,anchor=center] (sm1) {} ;

\draw[gray] (l1) arc (180:0:1 and 0.7) (r1) {} ;
\draw[gray] (l1) arc (180:0:1 and -0.5) (r1) {} ;

\draw (l1) -- (r1) ;
\draw (l1) -- (sm1) ;
\draw (sm1) -- (r1) ;

\path (c1) -- node[pos=0.33,auto,anchor=center,red] (v) {} (sm1);

\draw (v) -- (l1) ;
\draw (v) -- (r1) ;
\draw (v) -- (sm1);

\draw (l1) to[bend left=20] (t1) ;
\draw[dashed] (t1) to[bend left=10] (sm1) ;
\draw (r1) to[bend right=20] (t1) ;

\draw (l1) to[bend right=20] (b1) ;
\draw[dashed] (sm1) to[bend left=5] (b1) ;
\draw (r1) to[bend left=20] (b1) ;



\node[fill=none] (c3) at (7,0) {} ;

\node (t3) [above=of c3] {} ;
\node (b3) [below=of c3] {} ;

\node (tc3) [above=0.5 of t3] {} ;
\node (tl3) [above left=0.5 of t3] {} ;
\node (tr3) [above right=0.5 of t3] {} ;

\node (bc3) [below=0.5 of b3] {} ;
\node (bl3) [below left=0.5 of b3] {} ;
\node (br3) [below right=0.5 of b3] {} ;

\node[fill=none] (lspace3) [left=of c3] {} ;
\node[fill=none] (rspace3) [left=of c3] {} ;

\draw (t3) -- (b3) ;

\draw (t3) -- (tc3) ;
\draw (t3) -- (tl3) ;
\draw (t3) -- (tr3) ;

\draw (b3) -- (bc3) ;
\draw (b3) -- (bl3) ;
\draw (b3) -- (br3) ;


\node[fill=none,outer sep=0.5] (altop) at (8.75,0.05) {} ;
\node[fill=none,outer sep=0.5] (artop) at (9.75,0.05) {} ;

\node[fill=none,outer sep=0.5] (albot) at (8.75,-0.05) {} ;
\node[fill=none,outer sep=0.5] (arbot) at (9.75,-0.05) {} ;

\draw[->,black] (altop) -- (artop) {} ;
\draw[->,black] (arbot) -- (albot) {} ;

\node[fill=none] at (9.25,0.3) {\small{0-2 (vertex)}} ;
\node[fill=none] at (9.25,-0.3) {\small{2-0 (vertex)}} ;


\node[fill=none] (c4) at (11.5,0) {} ;

\node (t4) [above=of c4] {} ;
\node (b4) [below=of c4] {} ;

\node (tc4) [above=0.5 of t4] {} ;
\node (tl4) [above left=0.5 of t4] {} ;
\node (tr4) [above right=0.5 of t4] {} ;

\node (bc4) [below=0.5 of b4] {} ;
\node (bl4) [below left=0.5 of b4] {} ;
\node (br4) [below right=0.5 of b4] {} ;

\node[fill=none] (lspace4) [left=of c4] {} ;
\node[fill=none] (rspace4) [left=of c4] {} ;

\node (ti4) [above=0.5 of c4] {} ;
\node (bi4) [below=0.5 of c4] {} ;

\draw (t4) -- (ti4) ;
\draw (b4) -- (bi4) ;

\draw (ti4) -- (bi4) ;
\draw (ti4) to[bend left=50] (bi4) ;
\draw (ti4) to[bend right=50] (bi4) ;

\draw (t4) -- (tc4) ;
\draw (t4) -- (tl4) ;
\draw (t4) -- (tr4) ;

\draw (b4) -- (bc4) ;
\draw (b4) -- (bl4) ;
\draw (b4) -- (br4) ;

\end{tikzpicture}
}
\caption{0-2 and 2-0 vertex moves for a 3-dimensional triangulation (left), and their effect on the dual graph (right).}
\label{fig:0-2-move}
\end{figure}

\section{Face numbers of generalised triangulations of $4$-manifolds}
\label{sec:facenumbers}

In this section, we fully concentrate on dimension $d=4$, focusing on relationships between the number of vertices, number of pentachora and topology in generalised triangulations of $4$-manifolds. The observations we make in this section will serve as motivation for our study of face numbers of triangulations in higher dimensions in the subsequent sections.

\subsection{Dehn-Sommerville type equations}

Given a (generalised) triangulation $\tri$ of a compact closed (PL) $4$-manifold $\manifold$ with face vector $f(\tri) = (f_0,f_1,f_2,f_3,f_4)$, we first review some well-known identities on its face numbers. In the context of polytope theory and combinatorial topology, these identities are referred to as Dehn-Sommerville equations \cite{Sommerville27RelConnAngleSumVolPolyNDim,Ziegler95LectPolytopes}. Here we call them equations of {\em Dehn-Sommerville type} to highlight that they describe the same phenomenon in a different type of combinatorial object. For precisely this reason we also re-prove these equations for generalised triangulations of $4$-dimensional manifolds.

First note that the Euler characteristic of $\manifold$ equals the alternating sum of the face numbers of $\tri$. In particular, we have
\begin{equation*}
  f_0 - f_1 + f_2 - f_3 + f_4 = \chi(\manifold) = \beta_0 (\manifold) - \beta_1(\manifold) + \beta_2(\manifold) - \beta_3(\manifold) + \beta_4(\manifold),
\end{equation*}
where $\beta_i(\manifold):=\rank(H_i(\manifold))$ is the $i$th {\em Betti number} of $\manifold$. In particular, for $\manifold$ simply connected and closed, this reduces to $ f_0 - f_1 + f_2 - f_3 + f_4 = 2 + \beta_2 (\manifold).$ Next, we note that each pentachoron of $\tri$ contains exactly five tetrahedra in its boundary, and each of these is contained in exactly two pentachora. This implies the following identity
\begin{equation*}
  2f_3 - 5f_4 = 0 
\end{equation*}
which implies that $f_4$ is an even number.
Finally, we obtain a third linear equation on the face numbers of $\tri$ by the following construction:
Recall that edge links in $\tri$ are $2$-dimensional spheres. Hence, for $e \in \tri^{(1)}$, and $f(\operatorname{lk}_e (\tri))=(f^e_{0},f^e_{1},f^e_{2})$, we have $\chi (\operatorname{lk}_e (\tri))  = f^e_{0}-f^e_{1}+f^e_{2} =2$. Hence, we have for the sum of all edge links
\begin{align}
  2 f_1 &= \sum \limits_{e \in \tri^{(1)}} \chi (\operatorname{lk}_e (\tri))) \nonumber \\
   &=  \sum \limits_{e \in \tri^{(1)}} \left ( f^e_{0}-f^e_{1}+f^e_{2} \right ) \nonumber
\end{align} 

The vertices of edge links are in one-to-one correspondence with the corners of triangles, and hence $\sum_{e \in \tri^{(1)}} f^e_{0} = 3f_2$. Similarly, the edges of edge links are in one-to-one correspondence with edges of tetrahedra implying $\sum_{e \in \tri^{(1)}} f^e_{1} = 6f_3$, and the triangles of edge links are in one-to-one correspondence with triangles of pentachora leading to $\sum_{e \in \tri^{(1)}} f^e_{2} = 10f_4$. Altogether we have 
\begin{align}
  2 f_1 &= \sum \limits_{e \in \tri^{(1)}} \left ( f^e_{0}-f^e_{1}+f^e_{2} \right ) \nonumber \\
  &= 3f_2 - 6f_3 + 10f_4 \nonumber \\
  &= 3f_2 - 4f_3 + 5f_4 \nonumber 
\end{align} 

In summary, the $f$-vector of $\tri$, a triangulation of a closed $4$-manifold $\manifold$, must satisfy
\begin{align}
  \label{eq:DehnSommerville1}
  f_0 - f_1 + f_2 - f_3 + f_4 &= \chi (\manifold) \\
  \label{eq:DehnSommerville2}
    2f_1 - 3f_2 + 4f_3 - 5f_4 &= 0 \\
    \label{eq:DehnSommerville3}
                  2f_3 - 5f_4 &= 0
\end{align}
and, moreover, when $\manifold$ is simply connected, $\chi(\manifold)=2+\beta_2(\manifold)$.

\begin{remark}
  \label{rmk:dehn-sommerville-vs-nonsingular}
  These equations are closely related to the hierarchy of 4-dimensional \textit{pseudomanifolds} discussed in \Cref{sec:pseudomanifolds}. In the terminology of that section, a $4$-pseudomanifold satisfies \Cref{eq:DehnSommerville1,eq:DehnSommerville2,eq:DehnSommerville3} if and only if it is \textit{$1$-nonsingular}.

  To see this, let $\tri$ be a $4$-dimensional pseudomanifold triangulation as defined in \Cref{def:pseudo}, and $\manifold = |\tri|$. Since $\tri$ is, in particular, a closed triangulation (equivalently, all links of tetrahedra are $0$-spheres) \Cref{eq:DehnSommerville3} holds by the same argument as above. \Cref{eq:DehnSommerville1} holds by definition.

  We claim that \Cref{eq:DehnSommerville2} holds if and only if the edge links of $\tri$ are $2$-spheres. The ``if'' part of this statement was shown above. For the ``only if'' part, note that the edge links must be closed and connected surfaces (this follows from \Cref{prop:link-is-connected} and \Cref{cor:always-d-2}) and hence have Euler characteristic at most $2$. So in general we have
  \begin{align*}
    3f_2 - 4f_3 + 5f_4 = \sum \limits_{e \in \tri^{(1)}} \chi (\operatorname{lk}_e (\tri))) \le 2f_1
  \end{align*}
  with equality if and only if all of the edge links have Euler characteristic equal to $2$, that is they are $2$-spheres.
\end{remark}

\subsection{A conditional lower bound for the complexity of simply connected $4$-manifolds}
\label{sec:conditional}

Let $\tri$ be a triangulation of a simply connected $4$-manifold $\manifold$. We start with \Cref{eq:DehnSommerville1,eq:DehnSommerville2,eq:DehnSommerville3}. Eliminating $f_2$ and $f_3$ yields one equation

\[ 3f_0 - f_1 + \frac12 f_4 = 6 + 3 \beta_2(\manifold) \]

or, equivalently,
\begin{equation}
\label{eq:preprebound}
2f_0 + \frac12 f_4 = 6 + 3 \beta_2(\manifold) + (f_1 - f_0)
\end{equation}

\begin{lemma}
  \label{lem:tree}
  Let $\tri$ be a triangulation of a connected (but not necessarily closed) $4$-manifold with $f$-vector $f(\tri) = (f_0,f_1,f_2,f_3,f_4)$. Then
  \[ f_0 \leq f_1 \] 
\end{lemma}

\begin{proof}
  Let $G$ be the $1$-skeleton of $\tri$. Since $\tri$ is connected, $G$ is connected. This implies $f_0 \leq f_1 + 1$, with equality if and only if $G$ is a tree. Moreover, $G$ has been obtained by identifying vertices and edges of the graphs of pentachora under the face gluings of $\tri$. Recall that each of them is a complete graph on five vertices. If $G$ contains a loop, it cannot be a tree. If, however, $G$ does not contain any loop edges, then no two vertices of any pentachoron of $\tri$ can be identified as a result of the face gluings. But this implies that $G$ contains a complete graph on $5$ vertices and hence is not a tree either. Altogether, $G$ cannot be a tree and $f_0 \leq f_1$.   
\end{proof}

\begin{remark}
The inequality of \Cref{lem:tree} is best possible: there are one-vertex, one-edge triangulations of simply connected $4$-manifolds. See, for instance, \cite{burton13-k3-yrf}.
\end{remark}

Combining \Cref{lem:tree} with \Cref{eq:preprebound} we immediately obtain the bound
\begin{equation}
  \label{eq:prebound}
  f_4 + 4 f_0 \geq 12 + 6 \beta_2(\manifold)
\end{equation}

On the other hand, we have the following observation.

\begin{proposition}
  \label{prop:linear-bound}
  Let $\tri$ be a triangulation of a connected $4$-manifold with $f_0$ vertices and $f_4$ pentachora. Then
  \[ f_0 \leq f_4 + 4 \] 
\end{proposition}

\begin{proof}
  Choose a spanning tree $\Sigma$ of the dual graph $\Gamma(\tri)$. Build up a triangulated ball $B_\Sigma$ by gluing the pentachora of $\tri$ along the arcs of $\Sigma$. This triangulated ball contains $f_4$ pentachora and $f_4+4$ vertices. Moreover, it cannot have more vertices than $\tri$, since $\tri$ can be constructed from $B_\Sigma$ by identifying additional faces. This implies the result.
  \end{proof}
  
Combining \Cref{eq:prebound} with \Cref{prop:linear-bound} we obtain

\[ 5f_4 + 16 \geq 12 + 6 \beta_2(\manifold)\]

and hence our first lower bound

\begin{equation}
  \label{eq:bound}
  f_4 \geq \frac{6 \beta_2(\manifold) - 4}{5}
\end{equation}

However, this bound is not optimal for the simple reason that the inequality in \Cref{prop:linear-bound} cannot be tight. Instead, we re-formulate the lower bound from \Cref{eq:bound} in terms of a conditional upper bound for the number of vertices of a generalised triangulation of a simply connected $4$-manifold. 

\begin{corollary}
  \label{cor:general}
  Let $\tri$ be a (generalised) triangulation of a simply connected $4$-manifold $\manifold$ with $f$-vector $f(\tri) = (f_0,f_1,f_2,f_3,f_4)$. Moreover, let $f_0 \leq a f_4 + b$ for $a \in (0,1]$, and $b \in \RR$. Then
\begin{equation}
  \label{eq:generalbound}
  f_4 \geq \frac{1}{4a+1} \left( 6 \beta_2 (\manifold) + 12-4b \right) 
\end{equation}
\end{corollary}

\begin{proof}
  This is a straightforward case of inserting $f_0 \leq a f_4 + b$ into \Cref{eq:prebound}.
\end{proof}

\subsection{Vertex numbers of triangulated $4$-manifolds}
\label{sub:conjecture}

We have seen how, given a simply connected $4$-manifold $\manifold$, a lower bound on the number of pentachora $f_4$ in a generalised triangulation of $\manifold$ can be derived from an upper bound on its number of vertices $f_0$. We conjecture that 
\begin{equation}
  \label{conj:4d}
  f_0 \leq \frac{f_4}{2}+4
\end{equation}
We generalise this conjecture to arbitrary dimensions in \Cref{conj:vertex-bound-arbdim}. \Cref{conj:4d} is a natural generalisation of the upper bound for closed surfaces, where we have $f_0 = \frac{f_2}{2} + \chi(\tri)$, so $f_0 \leq \frac{f_2}{2} + 2$ with equality exactly for triangulations of the $2$-sphere.

If \Cref{conj:4d} holds, it is tight: it is achieved by the two-pentachoron ``pillow'' triangulation of the $4$-sphere (two pentachora glued along their boundaries with the identity map) which has $f$-vector $(5,10,10,5,2)$, giving $f_0 = \frac{f_4}{2}+4$. Repeatedly performing $0$-$2$ vertex moves introduces two pentachora and one vertex at a time. Hence, this produces arbitrarily large triangulations satisfying $f_0 = \frac{f_4}{2}+4$ (see \Cref{fig:0-2s-on-pillow} for a collection of dual graphs of these triangulations). We have the following statement.

\begin{figure}[htb]
\centering
\resizebox{0.7\linewidth}{!}{

\begin{tabular}{cccc}

\begin{tabular}{c}
\begin{tikzpicture}[every node/.style={circle, inner sep=0, outer sep=0, minimum width=0.1cm, fill=black}, every path/.style={thick}, label distance=2mm]

\node (l) at (0,0) {} ;
\node (r) [right=of l] {} ;

\draw (l) -- (r) ;
\draw (l) to[bend left=30]  (r) ;
\draw (l) to[bend right=30] (r) ;
\draw (l) to[bend left=60]  (r) ;
\draw (l) to[bend right=60] (r) ;

\end{tikzpicture}
\end{tabular}

&

\begin{tabular}{c}
\begin{tikzpicture}[every node/.style={circle, inner sep=0, outer sep=0, minimum width=0.1cm, fill=black}, every path/.style={thick}, label distance=2mm]

\node (bl) at (0,0) {} ;
\node (br) [right=of bl] {} ;
\node (tl) [above=of bl] {} ;
\node (tr) [right=of tl] {} ;

\draw (tl) to[bend left=15]  (tr) ;
\draw (tl) to[bend right=15] (tr) ;
\draw (tl) to[bend left=40]  (tr) ;
\draw (tl) to[bend right=40] (tr) ;

\draw (bl) to[bend left=15]  (br) ;
\draw (bl) to[bend right=15] (br) ;
\draw (bl) to[bend left=40]  (br) ;
\draw (bl) to[bend right=40] (br) ;

\draw (tl) -- (bl) ;
\draw (tr) -- (br) ;

\end{tikzpicture}
\end{tabular}

&

\begin{tabular}{c}
\begin{tikzpicture}[every node/.style={circle, inner sep=0, outer sep=0, minimum width=0.1cm, fill=black}, every path/.style={thick}, label distance=2mm]

\node[fill=none] (c) at (0,0) {} ;

\path (c) ++(0:1)   node (mr) {} ;
\path (c) ++(60:1)  node (tr) {} ;
\path (c) ++(120:1) node (tl) {} ;
\path (c) ++(180:1) node (ml) {} ;
\path (c) ++(240:1) node (bl) {} ;
\path (c) ++(300:1) node (br) {} ;

\draw (tr) -- (tl) ;
\draw (ml) -- (bl) ;
\draw (br) -- (mr) ;

\draw (mr) to[bend left=15]  (tr) ;
\draw (mr) to[bend right=15] (tr) ;
\draw (mr) to[bend left=40]  (tr) ;
\draw (mr) to[bend right=40] (tr) ;

\draw (tl) to[bend left=15]  (ml) ;
\draw (tl) to[bend right=15] (ml) ;
\draw (tl) to[bend left=40]  (ml) ;
\draw (tl) to[bend right=40] (ml) ;

\draw (bl) to[bend left=15]  (br) ;
\draw (bl) to[bend right=15] (br) ;
\draw (bl) to[bend left=40]  (br) ;
\draw (bl) to[bend right=40] (br) ;

\end{tikzpicture}
\end{tabular}

&

\begin{tabular}{c}
\begin{tikzpicture}[every node/.style={circle, inner sep=0, outer sep=0, minimum width=0.1cm, fill=black}, every path/.style={thick}, label distance=2mm]

\node (tl) at (0,0) {} ;
\node (tr) [right=of tl] {} ;
\node (ml) [below=of tl] {} ;
\node (mr) [right=of ml] {} ;
\node (bl) [below=of ml] {} ;
\node (br) [right=of bl] {} ;

\draw (mr) -- (tr) ;
\draw (tl) -- (ml) ;
\draw (ml) -- (bl) ;
\draw (br) -- (mr) ;

\draw (tl) to[bend left=15]  (tr) ;
\draw (tl) to[bend right=15] (tr) ;
\draw (tl) to[bend left=40]  (tr) ;
\draw (tl) to[bend right=40] (tr) ;

\draw (bl) to[bend left=15]  (br) ;
\draw (bl) to[bend right=15] (br) ;
\draw (bl) to[bend left=40]  (br) ;
\draw (bl) to[bend right=40] (br) ;

\draw (ml) --                (mr) ;
\draw (ml) to[bend left=25]  (mr) ;
\draw (ml) to[bend right=25] (mr) ;

\end{tikzpicture}
\end{tabular}

\\

\begin{tabular}{c}
\begin{tikzpicture}[every node/.style={circle, inner sep=0, outer sep=0, minimum width=0.1cm, fill=black}, every path/.style={thick}, label distance=2mm]

\node[fill=none] (c) at (0,0) {} ;

\path (c) ++(22.5:1)  node (ur) {} ;
\path (c) ++(67.5:1)  node (tr) {} ;
\path (c) ++(112.5:1) node (tl) {} ;
\path (c) ++(157.5:1) node (ul) {} ;
\path (c) ++(202.5:1) node (ll) {} ;
\path (c) ++(247.5:1) node (bl) {} ;
\path (c) ++(292.5:1) node (br) {} ;
\path (c) ++(337.5:1) node (lr) {} ;

\draw (tl) -- (tr) ;
\draw (bl) -- (br) ;
\draw (ul) -- (ll) ;
\draw (ur) -- (lr) ;

\draw (ur) to[bend left=15]  (tr) ;
\draw (ur) to[bend right=15] (tr) ;
\draw (ur) to[bend left=40]  (tr) ;
\draw (ur) to[bend right=40] (tr) ;

\draw (tl) to[bend left=15]  (ul) ;
\draw (tl) to[bend right=15] (ul) ;
\draw (tl) to[bend left=40]  (ul) ;
\draw (tl) to[bend right=40] (ul) ;

\draw (ll) to[bend left=15]  (bl) ;
\draw (ll) to[bend right=15] (bl) ;
\draw (ll) to[bend left=40]  (bl) ;
\draw (ll) to[bend right=40] (bl) ;

\draw (br) to[bend left=15]  (lr) ;
\draw (br) to[bend right=15] (lr) ;
\draw (br) to[bend left=40]  (lr) ;
\draw (br) to[bend right=40] (lr) ;

\end{tikzpicture}
\end{tabular}

&

\begin{tabular}{c}
\begin{tikzpicture}[every node/.style={circle, inner sep=0, outer sep=0, minimum width=0.1cm, fill=black}, every path/.style={thick}, label distance=2mm]

\node (tl) at (0,0) {} ;
\node (tr) [right=of tl] {} ;
\node (ul) [below=of tl] {} ;
\node (ur) [right=of ul] {} ;
\node (ll) [below=of ul] {} ;
\node (lr) [right=of ll] {} ;
\node (bl) [below=of ll] {} ;
\node (br) [right=of bl] {} ;

\draw (tl) -- (ul) ;
\draw (ul) -- (ll) ;
\draw (ll) -- (bl) ;
\draw (tr) -- (ur) ;
\draw (ur) -- (lr) ;
\draw (lr) -- (br) ;

\draw (tl) to[bend left=15]  (tr) ;
\draw (tl) to[bend right=15] (tr) ;
\draw (tl) to[bend left=40]  (tr) ;
\draw (tl) to[bend right=40] (tr) ;

\draw (bl) to[bend left=15]  (br) ;
\draw (bl) to[bend right=15] (br) ;
\draw (bl) to[bend left=40]  (br) ;
\draw (bl) to[bend right=40] (br) ;

\draw (ul) --                (ur) ;
\draw (ul) to[bend left=25]  (ur) ;
\draw (ul) to[bend right=25] (ur) ;

\draw (ll) --                (lr) ;
\draw (ll) to[bend left=25]  (lr) ;
\draw (ll) to[bend right=25] (lr) ;

\end{tikzpicture}
\end{tabular}

&

\begin{tabular}{c}
\begin{tikzpicture}[every node/.style={circle, inner sep=0, outer sep=0, minimum width=0.1cm, fill=black}, every path/.style={thick}, label distance=2mm]

\node[fill=none] (c) at (0,0) {} ;

\path (c) ++(0:1)   node (ur) {} ;
\path (c) ++(60:1)  node (tr) {} ;
\path (c) ++(120:1) node (tl) {} ;
\path (c) ++(180:1) node (ul) {} ;
\path (c) ++(240:1) node (ll) {} ;
\path (c) ++(300:1) node (lr) {} ;

\node (bl) [below=of ll] {} ;
\node (br) [below=of lr] {} ;

\draw (tl) -- (tr) ;
\draw (ul) -- (ll) ;
\draw (ur) -- (lr) ;
\draw (ll) -- (bl) ;
\draw (lr) -- (br) ;

\draw (ur) to[bend left=15]  (tr) ;
\draw (ur) to[bend right=15] (tr) ;
\draw (ur) to[bend left=40]  (tr) ;
\draw (ur) to[bend right=40] (tr) ;

\draw (tl) to[bend left=15]  (ul) ;
\draw (tl) to[bend right=15] (ul) ;
\draw (tl) to[bend left=40]  (ul) ;
\draw (tl) to[bend right=40] (ul) ;

\draw (bl) to[bend left=15]  (br) ;
\draw (bl) to[bend right=15] (br) ;
\draw (bl) to[bend left=40]  (br) ;
\draw (bl) to[bend right=40] (br) ;

\draw (ll) --                (lr) ;
\draw (ll) to[bend left=25]  (lr) ;
\draw (ll) to[bend right=25] (lr) ;

\end{tikzpicture}
\end{tabular}

&

\begin{tabular}{c}
\begin{tikzpicture}[every node/.style={circle, inner sep=0, outer sep=0, minimum width=0.1cm, fill=black}, every path/.style={thick}, label distance=2mm]

\node[fill=none] (c) at (0,0) {} ;

\node (cl) [left=0.5 of c] {} ;
\node (cr) [right=0.5 of c] {} ;

\path (c) ++(0:1.5)   node (mr) {} ;
\path (c) ++(60:1.5)  node (tr) {} ;
\path (c) ++(120:1.5) node (tl) {} ;
\path (c) ++(180:1.5) node (ml) {} ;
\path (c) ++(240:1.5) node (bl) {} ;
\path (c) ++(300:1.5) node (br) {} ;

\draw (cl) -- (tr) ;
\draw (cr) -- (mr) ;
\draw (cl) -- (ml) ;
\draw (cr) -- (tl) ;
\draw (cl) -- (bl) ;
\draw (cr) -- (br) ;

\draw (mr) to[bend left=15]  (tr) ;
\draw (mr) to[bend right=15] (tr) ;
\draw (mr) to[bend left=40]  (tr) ;
\draw (mr) to[bend right=40] (tr) ;

\draw (tl) to[bend left=15]  (ml) ;
\draw (tl) to[bend right=15] (ml) ;
\draw (tl) to[bend left=40]  (ml) ;
\draw (tl) to[bend right=40] (ml) ;

\draw (bl) to[bend left=15]  (br) ;
\draw (bl) to[bend right=15] (br) ;
\draw (bl) to[bend left=40]  (br) ;
\draw (bl) to[bend right=40] (br) ;

\draw (cl) to[bend left=15]  (cr) ;
\draw (cl) to[bend right=15] (cr) ;

\end{tikzpicture}
\end{tabular}

\end{tabular}}
\caption{Dual graphs of triangulations obtained by repeated $0$-$2$ vertex moves on the pillow triangulation of $\mathbb{S}^4$ (top left) with up to 8 pentachora. These all achieve $f_0 = \frac{f_4}{2} + 4$, and by \Cref{thm:vertex-bound-gem} they are the only closed, balanced, 4-dimensional triangulations which attain this, for $f_4 \le 8$.}
\label{fig:0-2s-on-pillow}
\end{figure}

\begin{corollary}
  \label{cor:actualBound}
  Suppose \Cref{conj:4d} holds, and let $\tri$ be a generalised triangulation of a simply connected closed $4$-manifold $\manifold$ with $f$-vector $f(\tri) = (f_0,f_1,f_2,f_3,f_4)$. Then
  \[ f_4 \geq 2 \beta_2(\manifold) \]
\end{corollary}

\begin{proof}
  Introducing the upper bound $f_0 \leq \frac{f_4}{2} + 4$ into \Cref{eq:generalbound} yields
  \[ f_4 \geq 2 \beta_2(\manifold) - \frac{4}{3} \]
  Since triangulations of closed $4$-manifolds must have an even integer number of pentachora, the statement follows.
\end{proof}

\begin{remark}
  In \cite{SpreerTobin2025-SmallTriangulations} the authors present triangulations $(\tri_k)_{k \in \mathbb{N}}$ of simply connected $4$-manifolds $(\manifold_k)_{k \in \mathbb{N}}$ with $\beta_2(\manifold_k) = k$, $f_4 = 2 k + 2$ and $f_0 = \frac{f_d}{2}+2$. These almost attain both the upper bound on $f_0$ from \Cref{conj:vertex-bound-arbdim} and the lower bound on $f_4$ from \Cref{cor:actualBound} is almost attained, implying that triangulations of simply connected $4$-manifolds with small numbers of facets must have large numbers of vertices -- at least in some cases.
  
  This may seem counter intuitive at first, but can be explained as follows: All handle decompositions of $\manifold_k$ must contain at least $k+2$ handles, with at least $k$ of them being $2$-handles. These $2$-handles are, by definition, $4$-balls that attach to the rest of the manifold along a solid torus. Such a $2$-handle may be represented by simplicial cones over small triangulations of the $3$-sphere, with a possible attaching region being the neighbourhood of one of its boundary edges. This is, in fact, the case for the triangulations presented in \cite{SpreerTobin2025-SmallTriangulations}. The smallest such triangulations of potential $2$-handles are presented in \Cref{sec:pseudomanifolds} as $\textrm{DSB}_1$ and $\textrm{DSB}_2$ (with dual graphs a single node with two loop edges attached), but slightly larger examples exist in many forms. Being cones, all of these triangulations have an interior vertex which is not contained in any other pentachoron, forcing a triangulation made up of multiple such gadgets to have a high number of vertices. Incidentally, such features in a triangulation typically result in a large {\em branching number} of its dual graph, as defined in \Cref{sub:dual-graph-conditions}. This, in turn implies that the bound from \Cref{thm:branch-bound} does {\em not} imply \Cref{conj:vertex-bound-even} in the $4$-dimensional case. 
\end{remark}

We conclude this section with a number of questions.

\begin{question}
  \label{q:dehn-sommerville-vs-conjecture}
  Does \Cref{conj:4d} hold for all $4$-dimensional triangulations satisfying \Cref{eq:DehnSommerville1,eq:DehnSommerville2,eq:DehnSommerville3}?
\end{question}

\begin{question}
  \label{q:classification-equality}
  Can we classify triangulations realising equality in \Cref{eq:bound}?
\end{question}

We have already seen one family of triangulations realising this equality, from 0-2 vertex moves on the pillow 4-sphere, and to the authors' best knowledge these are the only such triangulations of 4-manifolds known. In \Cref{sub:balanced} we show that these solve \Cref{q:classification-equality} for the case of balanced triangulations. To generalise the classification of cases of equality, we introduce the following quantity, measuring the deviation of a triangulation from the hypothesised bound.

\begin{equation}
  \label{eq:delta}
  \delta_{\tri} := f_0 - \frac{f_4}{2}
\end{equation}

In this context, \Cref{q:classification-equality} asks for the classification of triangulations $\tri$ satisfying $\delta_{\tri} = 4$. 

\begin{question}
  \label{q:classification-general}
  Can we classify triangulations realising $\delta_{\tri} \in \{ 5,6,7\}$?
\end{question}

\section{Sharp upper bounds on vertex numbers in arbitrary odd dimensions}
\label{sec:higherDims}

We now begin our investigation on the maximum possible number of vertices, relative to the number of facets, of generalised triangulations of manifolds in arbitrary dimensions. In odd dimensions, we will spend this section answering this precisely in all but finitely many cases per dimension. In even dimensions, we generalise \Cref{eq:bound} and the $2$-dimensional bound $f_0 \le \frac{f_2}{2}+2$ to conjecture that the bound $f_0 \le \frac{f_d}{2}+d$ holds in arbitrary even dimensions, a conjecture we will see some evidence for in \Cref{sec:sufficient-conditions}. We summarise these results and conjectures with the following.

\begin{conjecture}
  \label{conj:vertex-bound-arbdim}
  Let $\tri$ be a triangulation of a closed and connected $d$-dimensional manifold, $d > 0$, with $f_0$ vertices and $f_d$ facets. Then
  \begin{enumerate}[label=(\alph*)]
    \item \label{item:main-odd} 
      \begin{equation}
        \label{eq:main-odd}
        f_0 \leq 
        \begin{cases}
         \frac{f_d}{2} + d \textrm{ if } d \textrm{ even, or if } d \textrm{ odd, } f_d \textrm{ even and } f_d < d \\
         f_d + \frac{d-1}{2} \textrm{ otherwise}
        \end{cases}
      \end{equation}
    \item \label{item:main-sharpness} For every feasible value of $f_d$ in every dimension, there exists an $f_d$-facet triangulation of the $d$-sphere attaining equality in the applicable bound above.
  \end{enumerate}
\end{conjecture}

\begin{theorem}
  \label{thm:vertex-bound-arbdim}
  Both parts of \Cref{conj:vertex-bound-arbdim} hold for (i) dimension two, (ii) arbitrary odd dimensions when $f_d \ge d$, and (iii) arbitrary odd dimensions when $f_d < d$ and $f_d$ is even, or $f_d=1$.
\end{theorem}

\begin{figure}
\centering
\begin{tikzpicture}
\begin{axis}[axis lines = left, xlabel = {$f_d$}, ylabel = {$f_0$}, ylabel style = {rotate=-90}, xmin = 0, ymin = 0, legend pos = south east, legend style = {nodes={scale=0.7, transform shape}}, legend cell align = left, grid=both, minor x tick num=1, minor y tick num=4]

\addplot[thick, color=red, mark=^, mark options={scale=2,ultra thick}]
    coordinates {(7,10)(8,11)(9,12)(10,13)(11,14)(12,15)};

\addplot[thick, color=red, mark=^, mark options={scale=2,ultra thick}, forget plot]
    coordinates {(1,4)};

\addplot[thick, color=blue, mark=^, mark options={scale=2,ultra thick}]
    coordinates {
        (2,8)  
        (3,8)
        (4,9)  
        (5,9)
        (6,10) 
        (7,10)
    };

\addplot[thick, color=orange, mark=v, mark options={scale=2,ultra thick}]
    coordinates {(1,4)(2,5)(3,6)(4,7)(5,8)(6,9)(7,10)(8,11)(9,12)(10,13)(11,14)(12,15)};

\addplot[thick, color=green, mark=v, mark options={scale=2,ultra thick}]
    coordinates {(2,8)(4,9)(6,10)(8,11)(10,12)(12,13)};

\addplot[thick, mark=none, color=black, dashed] coordinates {(3,6)(3,8)};
\addplot[thick, mark=none, color=black, dashed] coordinates {(5,8)(5,9)};
    
\legend{bound: $f_0 \le f_d + \frac{d-1}{2}$, bound: $f_0 \le \lfloor\frac{f_d}{2}\rfloor + d$, construction: $f_0 = f_d + \frac{d-1}{2}$, construction: $f_0 = \frac{f_d}{2} + d$}

\end{axis}
\end{tikzpicture}
\caption{Summary of the results of \Cref{sec:higherDims}, plotted in the case of $d=7$. Dashed lines show range of possible sharp bounds in the cases where they are not known. Note the cross-over between the two constructions (\Cref{prop:construction-odd-large,prop:construction-even}) at $f_d=d+1$, leading to the distinction between odd and even $f_d$ when $f_d$ is small (and the construction of \Cref{prop:construction-even}, which only exists for even $f_d$, dominates).}
\label{fig:bound-plot}
\end{figure}

The case of triangulations of surfaces $S$ -- case {\it (i)} of \Cref{thm:vertex-bound-arbdim} -- was discussed in \Cref{sub:conjecture}, and equality is attained for every triangulation of the $2$-sphere (and no triangulations of other surfaces). Case {\it (ii)} and {\it (iii)} of \Cref{thm:vertex-bound-arbdim} are the subject of \Cref{prop:construction-odd-large,prop:construction-even,prop:odd}. 

The case where $f_d<d$ and $d$ is odd, along with all cases of $d$ even, $d \ge 4$, must be left open at this point. 

\medskip

We begin with the explicit construction of the $d$-sphere triangulations which attain equality in \Cref{eq:main-odd}. The first such construction, \Cref{prop:construction-odd-large}, is based on an $f_3$-tetrahedron, $(f_3+1)$-vertex $3$-sphere construction, originally due to Burton (cf. family $\mathcal A_n$ in \cite{Burton13ComplexityBounds}).

\begin{proposition}
  \label{prop:construction-odd-large}
  Let $d$ and $f_d$ be positive integers, $d$ odd. There exists a triangulation $\mathbb{S}^d_{f_d}$ of the $d$-sphere with $f_d$ facets and $f_0 = f_d + \frac{d-1}{2}$ vertices.
\end{proposition}
\begin{proof}
  For $d=1$, the construction is simply an $f_d$-gon, and trivially $f_0 = f_1 = f_1 + \frac{d-1}{2}$.

  Let $d\geq 3$ be odd. The core building block of our triangulation $\mathbb{S}^d_{f_d}$ is a single facet with $\frac{d-1}{2}$ pairs of its ridges glued together (or ``snapped''), which we call a \textit{$\frac{d-1}{2}$-snapped $d$-ball} $SB^d_{(d-1)/2}$. Given a single facet with the standard labelling of its vertices and ridges, we can glue two ridges by folding them over their common $(d-2)$-face. If the ridges are ridge $0$ and $1$ respectively, the common $(d-2)$-face has vertices $(2,3,\dots,d)$, and this gluing is realised by the permutation $(1,2,3,\dots,d) \mapsto (0,2,3,\dots,d)$. As a result, $0$ and $1$ are identified, but no other vertex identifications are introduced.
  
  The result is a $1$-facet, $d$-vertex triangulation of the $d$-ball which we call a {\em single-snapped $d$-ball} $SB^d_1$. Its boundary is a triangulation of the $(d-1)$-sphere with $d-1$ facets, which individually are themselves each single-snapped $(d-1)$-balls, and also glued to one another pairwise along distinct $(d-2)$-faces, which for $d>3$ are again single-snapped $(d-2)$-balls. For an example of the former, ridge $(0,1,3,4,\dots,d)$ of $SB^d_1$ (which gives rise to one boundary facet) is glued to itself by $(1,3,4,\dots,d) \mapsto (0,3,4,\dots,d)$, i.e. folding over the $(d-3)$-face $(3,4,\dots,d)$. And for an example of the latter, ridges $(0,1,3,4,\dots,d)$ and $(0,1,2,4,\dots,d)$ are glued together in the boundary because they share the $(d-2)$-face $(0,1,4,\dots,d)$, and this is folded along the $(d-3)$-face $(4,5,\dots,d)$ by $(1,4,5,\dots,d) \mapsto (0,4,5,\dots,d)$.
  
  For $d=3$ there are only two ridges remaining unglued, and we have our building block $SB^3_1$. For $d>3$ odd, we can iterate this procedure by gluing ridges $2$ and $3$ by $(0,1,2,4,\dots,d) \mapsto (0,1,3,4,\dots,d)$, folding over the already folded $(d-2)$-face $(0,1,4,\dots,d)$. The result is still a triangulation of the $d$-ball, which we now call a {\em double-snapped ball} $SB^d_2$, and has $d-1$ vertices and boundary a triangulated $(d-1)$-sphere with $d-3$ facets. Again, we can argue that these facets intersect pairwise in distinct $(d-2)$-faces, which if $d>5$ are now double-snapped $(d-2)$-balls individually. Repeating this process until $\frac{d-1}{2}$ folds have been made (identifying vertices $4$ and $5$, $6$ and $7$, and so on up to $d-3$ and $d-2$) results in a $1$-facet triangulation of a $d$-ball $SB^d_{(d-1)/2}$. This has two boundary facets corresponding to ridge $d-1$ and $d$ of the original facet, both $\frac{d-1}{2}$-snapped $(d-1)$-balls which are glued along the common $(d-2)$-face $(0,1,\dots,d-2)$. This $(d-2)$ face has itself been snapped $\frac{d-1}{2}$ times, and so must be a $1$-facet triangulation of the $(d-2)$-sphere with $\frac{d-1}{2}$ vertices. By construction, the boundary of $SB^d_{(d-1)/2}$ is a double cone over this $(d-2)$-sphere triangulation, with cone vertices $d-1$ and $d$. The dual graph of $SB^d_{(d-1)/2}$ is simply $\frac{d-1}{2}$ loops attached to a single node.
  
  Given $f_d$ copies of $SB^d_{(d-1)/2}$, we can glue them together in a circular fashion, gluing face $(0,1,\ldots,d-1)$ of one to $(0,1,\ldots,d-2,d)$ of the next, covering a cone apex of one with a cone apex from the next and identifying the remaining vertices according to their original labels, to produce a closed triangulation $\mathbb{S}^d_{f_d}$. $\mathbb{S}^d_{f_d}$ has $f_d$ vertices from identifying the $2f_d$ cone apices in pairs, and $\frac{d-1}{2}$ additional vertices, giving $f_0 = f_d + \frac{d-1}{2}$ as required. The dual graph of this triangulation is an $f_d$-gon with $\frac{d-1}{2}$ loops at each node, see \Cref{fig:family}.

  Finally, we claim that $\mathbb{S}^d_{f_d}$ triangulates the standard PL $d$-sphere. PL-topologically, as a $\frac{d-1}{2}$-snapped $d$-ball is glued to an existing complex, a PL $d$-ball is glued to another PL $d$-ball by a PL-homeomorphism identifying the upper hemisphere of the boundary of the first PL $d$-ball (a PL $(d-1)$-ball) to the lower hemisphere of the boundary of the second PL $d$-ball (another PL $(d-1)$-ball) to form another PL $d$-ball. At the final step we identify the upper and lower hemispheres of the boundary of a PL $d$-ball (in a way that is the identity on the $(d-2)$-dimensional equator) so that the resulting space is indeed the PL standard $d$-sphere.
\end{proof} 

\begin{figure}
\centering
\resizebox{0.3\linewidth}{!}{
\begin{tikzpicture}[every node/.style={circle, inner sep=0, outer sep=0, minimum width=0.15cm, fill=black}, every path/.style={thick}, label distance=2mm]

\node[fill=none] (c) {} ;

\path (c) ++(18:1)   node (1) {} ;
\path (c) ++(90:1)  node (2) {} ;
\path (c) ++(162:1) node (3) {} ;
\path (c) ++(234:1) node (4) {} ;
\path (c) ++(306:1) node (5) {} ;

\draw (1) -- (2) ;
\draw (2) -- (3) ;
\draw (3) -- (4) ;
\draw (4) -- (5) ;
\draw (5) -- (1) ;

\draw (1) to[out=48,in=-12,loop] (1) ;
\draw[scale=1.8] (1) to[out=58,in=-22,loop] (1) ;

\draw (2) to[out=120,in=60,loop] (2) ;
\draw[scale=1.8] (2) to[out=130,in=50,loop] (2) ;

\draw (3) to[out=192,in=132,loop] (3) ;
\draw[scale=1.8] (3) to[out=202,in=122,loop] (3) ;

\draw (4) to[out=264,in=204,loop] (4) ;
\draw[scale=1.8] (4) to[out=274,in=194,loop] (4) ;

\draw (5) to[out=336,in=276,loop] (5) ;
\draw[scale=1.8] (5) to[out=346,in=256,loop] (5) ;

\end{tikzpicture}
}
\resizebox{0.3\linewidth}{!}{
\begin{tikzpicture}[every node/.style={circle, inner sep=0, outer sep=0, minimum width=0.15cm, fill=black}, every path/.style={thick}, label distance=2mm]

\node[fill=none] (c) {} ;

\path (c) ++(0:1)   node (1) {} ;
\path (c) ++(60:1)  node (2) {} ;
\path (c) ++(120:1) node (3) {} ;
\path (c) ++(180:1) node (4) {} ;
\path (c) ++(240:1) node (5) {} ;
\path (c) ++(300:1) node (6) {} ;

\draw (1) -- (2) ;
\draw (2) -- (3) ;
\draw (3) -- (4) ;
\draw (4) -- (5) ;
\draw (5) -- (6) ;
\draw (6) -- (1) ;

\draw (1) to[out=30,in=-30,loop] (1) ;
\draw[scale=1.8] (1) to[out=40,in=-40,loop] (1) ;

\draw (2) to[out=90,in=30,loop] (2) ;
\draw[scale=1.8] (2) to[out=100,in=20,loop] (2) ;

\draw (3) to[out=150,in=90,loop] (3) ;
\draw[scale=1.8] (3) to[out=160,in=80,loop] (3) ;

\draw (4) to[out=210,in=150,loop] (4) ;
\draw[scale=1.8] (4) to[out=220,in=140,loop] (4) ;

\draw (5) to[out=270,in=210,loop] (5) ;
\draw[scale=1.8] (5) to[out=280,in=200,loop] (5) ;

\draw (6) to[out=330,in=270,loop] (6) ;
\draw[scale=1.8] (6) to[out=340,in=260,loop] (6) ;

\end{tikzpicture}
}
\resizebox{0.3\linewidth}{!}{
\begin{tikzpicture}[every node/.style={circle, inner sep=0, outer sep=0, minimum width=0.15cm, fill=black}, every path/.style={thick}, label distance=2mm]

\node[fill=none] (c) {} ;

\path (c) ++({90-360/7}:1)   node (1) {} ;
\path (c) ++(90:1)  node (2) {} ;
\path (c) ++({90+360/7}:1) node (3) {} ;
\path (c) ++({90+2*360/7}:1) node (4) {} ;
\path (c) ++({90+3*360/7}:1) node (5) {} ;
\path (c) ++({90+4*360/7}:1) node (6) {} ;
\path (c) ++({90+5*360/7}:1) node (7) {} ;

\draw (1) -- (2) ;
\draw (2) -- (3) ;
\draw (3) -- (4) ;
\draw (4) -- (5) ;
\draw (5) -- (6) ;
\draw (6) -- (7) ;
\draw (7) -- (1) ;

\draw (1) to[out={90-360/7+30},in={90-360/7-30},loop] (1) ;
\draw[scale=1.8] (1) to[out={90-360/7+40},in={90-360/7-40},loop] (1) ;

\draw (2) to[out=120,in=60,loop] (2) ;
\draw[scale=1.8] (2) to[out=130,in=50,loop] (2) ;

\draw (3) to[out={90+360/7+30},in={90+360/7-30},loop] (3) ;
\draw[scale=1.8] (3) to[out={90+360/7+40},in={90+360/7-40},loop] (3) ;

\draw (4) to[out={90+2*360/7+30},in={90+2*360/7-30},loop] (4) ;
\draw[scale=1.8] (4) to[out={90+2*360/7+40},in={90+2*360/7-40},loop] (4) ;

\draw (5) to[out={90+3*360/7+30},in={90+3*360/7-30},loop] (5) ;
\draw[scale=1.8] (5) to[out={90+3*360/7+40},in={90+3*360/7-40},loop] (5) ;

\draw (6) to[out={90+4*360/7+30},in={90+4*360/7-30},loop] (6) ;
\draw[scale=1.8] (6) to[out={90+4*360/7+40},in={90+4*360/7-40},loop] (6) ;

\draw (7) to[out={90+5*360/7+30},in={90+5*360/7-30},loop] (7) ;
\draw[scale=1.8] (7) to[out={90+5*360/7+40},in={90+5*360/7-40},loop] (7) ;

\end{tikzpicture}
}
\caption{Dual graphs of the triangulations (left to right) $\mathbb{S}^5_5$, $\mathbb{S}^5_6$ and $\mathbb{S}^5_7$ as constructed in \Cref{prop:construction-odd-large}.\label{fig:family}}  
\label{fig:ring-spheres}
\end{figure}

\begin{proposition}
  \label{prop:construction-even}
  Let $d$ and $f_d$ be positive integers with $f_d$ even. Then there exists a triangulation $\mathbb{S}^d_{f_d}$ of the $d$-sphere with $f_d$ facets and $f_0 = \frac{f_d}{2} + d$ vertices.
\end{proposition}
\begin{proof}
  Consider the standard triangulation of a $d$-sphere with two facets glued along their boundaries according to the identity map. This triangulation has $d+1$ vertices and $2$ facets and the bound is satisfied. The statement now follows by iteratively applying $0$-$2$ vertex moves. Each such move adds two facets and one vertex to the triangulation, preserving the identity. This proves the statement.
\end{proof}

\Cref{prop:construction-odd-large,prop:construction-even} complete the proof of part \labelcref{item:main-sharpness} of \Cref{conj:vertex-bound-arbdim}, proving that \Cref{eq:main-odd} is sharp if it is true. We now move on to proofs of the bounds themselves, starting with some more general lemmas.

\begin{lemma}
  \label{lem:first-simplex-loops}
  Let $\tri$ be a $d$-dimensional triangulation consisting of a single facet with $l$ pairs of its ridges identified. Then $\tri$ has at most $d+1-l$ vertices.
\end{lemma}
\begin{proof}
  Consider a ridge which is used in one of the identifications. The facet has one vertex, say $v^*$, which is not contained in this ridge. However, the other ridge which this ridge is glued to must contain $v^*$, and so after gluing, $v^*$ must be identified with at least one other vertex. This applies to all $2l$ ridges that appear in a gluing. In particular, it applies twice for every pair of ridges forming a gluing (with the respective vertices $v^*$ possibly becoming identified to each other). Thus, at least $2l$ of the original $d+1$ vertices appear in vertices of degree at least $2$ in $\tri$. If we label these vertices by $v_1,\dots,v_{2l}$ and the others by $v_{2l+1},\dots,v_{d+1}$, we have:
  \begin{align*}
    \# \{\textrm{vertices of } \tri\} &= \sum_{i=1}^{d+1} \frac{1}{\textrm{degree of $v_i$ in $\tri$}} \\
                                      &= \sum_{i=1}^{2l} \frac{1}{\textrm{degree of $v_i$ in $\tri$}} + \sum_{i=2l+1}^{d+1} \frac{1}{\textrm{degree of $v_i$ in $\tri$}} \\
                                      &\le \sum_{i=1}^{2l} \frac{1}{2} + \sum_{i=2l+1}^{d+1} 1 \\
                                      &= d+1-l
  \end{align*}
\end{proof}

\begin{corollary}
  \label{cor:fdone}
  Part \labelcref{item:main-odd} of \Cref{conj:vertex-bound-arbdim} holds in the case of $f_d=1$.
\end{corollary}
\begin{proof}
  To be closed, the triangulation must consist of a single facet with $\frac{d+1}{2}$ pairs of ridges identified, and the result follows immediately from \Cref{lem:first-simplex-loops}.
\end{proof}

\begin{lemma}
  \label{lem:general-bound-loops}
  Let $\tri$ be a closed and connected $d$-dimensional triangulation with $f_0$ vertices and $f_d$ facets, and $l$ the maximum number of loops at a node of $\Gamma(\tri)$. Then
  \begin{equation*}
    f_0 \le \frac{d+1}{2(d-l)}f_d + (d-l) - \frac{1}{d-l}
  \end{equation*}
\end{lemma}
\begin{proof}
    Let $s_1$ be a node of $\Gamma(\tri)$ with $l$ loops, and label the remaining nodes by $s_2,s_3,\dots,s_{f_d}$, such that each $s_k$ has at least one arc connecting it to some $s_j$ with $j < k$ (for example, ordering them by a breadth-first search through a spanning tree based at $s_1$).

    Let $\tri_k$ be the triangulation obtained from $\tri$ by removing the facets corresponding to $s_{k+1},\dots,s_{f_d}$, all of their gluings, and all self-gluings in $\tri$ \textit{except} the $l$ self-gluings corresponding to the $l$ loops at $s_1$. We can construct $\tri$ as a sequence $\tri_1,\tri_2,\dots,\tri_{f_d},\tri$ where at each step we add on the facet corresponding to $s_k$ and all gluings joining it to existing facets, and then in the final step add in any remaining self-gluings (which we note cannot increase the number of vertices). As we do so, we keep track of two quantities, the number of vertices $f_0^k$ of $\tri_k$ and the size of the ``cut set'':
    \begin{align*}
        C_k := \#\{e \in E(\Gamma(\tri)) \ | \ e = \langle v_i,v_j \rangle \textrm{ with } i \le k < j \}
    \end{align*}
    That is, if we draw $\Gamma(\tri)$ with $s_1,s_2,\dots,s_{f_4}$ arranged in order along a horizontal line, then $C_k$ is the number of arcs intersected by a vertical line between $s_k$ and $s_{k+1}$. For $k=1$, we have $C_1 = d+1-2l$, and by \Cref{lem:first-simplex-loops}, $f_0^1 \le d+1-l$.

    Now for each $k$, consider the number of arcs going ``forwards" and ``backwards" from $s_k$:
    \begin{align*}
        F_k &:= \#\{e \in E(\Gamma(\tri)) \ | \ e = \langle s_k,s_j \rangle \textrm{ with } j > k \} \\
        B_k &:= \#\{e \in E(\Gamma(\tri)) \ | \ e = \langle s_j,s_k \rangle \textrm{ with } j < k \}
    \end{align*}
    For $k \ge 2$, we observe that $C_k = C_{k-1} + F_k - B_k$, and since each $s_k$ may have at most $l$ loops in $\Gamma(\tri)$, $d+1-2l \le B_k+F_k \le d+1$. By the construction of our labelling, $B_k \ge 1$ for all $k \ge 2$, and so we can consider two cases:

    \smallskip
    \noindent
    \textbf{Case 1: $B_k = 1$.}

    $C_k \ge C_{k-1}+d-1-2l$, with equality when $s_k$ has $l$ loops, and we claim $f_0^k = f_0^{k-1} + 1$. Indeed, $\tri_k$ is obtained from $\tri_{k-1}$ by adding one facet (with $d+1$ vertices) and one gluing between it and an existing facet, so $d$ of the new vertices are identified with existing vertices by the gluing, and the one remaining vertex remains isolated and is added to the count of $f_0^k$.

    \smallskip
    \noindent
    \textbf{Case 2: $B_k \ge 2$.}

    $C_k \ge C_{k-1}-d-1$, with equality when $B_k=d+1$ and $F_k=0$, and $f_0^k \le f_0^{k-1}$, since all $d+1$ vertices of the new facet are involved in at least one of the two gluings to existing facets, and hence are identified with some existing vertex.  

    \medskip
    \noindent
    Letting $X$ and $Y$ be the number of times cases 1 and 2 occur respectively:

    \begin{align*}
        0 = C_{f_4} &\ge C_1 + (d-1-2l)X - (d+1)Y \\
                    &= d+1-2l + (d-1-2l)X - (d+1)(f_d-X-1) \\
                    &= 2(d-l)X -(d+1)f_d + 2(d-l) + 2 \\
        \Rightarrow X &\le \frac{(d+1)}{2(d-l)}f_d - \frac{1}{d-l} - 1
    \end{align*}
  On the other hand:
  \begin{align*}
    f_0 &\le f_0^1 + X \\
        &\le d+1-l + \frac{(d+1)}{2(d-l)}f_d - \frac{1}{d-l} - 1 \\
        &= \frac{(d+1)}{2(d-l)}f_d + (d-l) - \frac{1}{d-l}
  \end{align*}
\end{proof}

\begin{proposition}
    \label{prop:small-n}
    Let $\tri$ be a closed, connected $d$-dimensional triangulation with $f_0$ vertices and $f_d$ facets. If $2 \le f_d \le d$ then
    \begin{equation*}
        f_0 \leq \left\lfloor \frac{f_d}{2} \right\rfloor + d
    \end{equation*}
\end{proposition}
\begin{proof}
    Let $l$ be the maximum number of loops of any node of $\Gamma(\tri)$. First, suppose $d$ is even. Then $l \le \frac{d}{2}$, and using \Cref{lem:general-bound-loops}:
    \begin{align*}
        f_0-\frac{f_d}{2}-d &\le \frac{d+1}{2(d-l)}f_d - \frac{d-l}{2(d-l)}f_d -l - \frac{1}{d-l} \\
                            &=   \frac{(l+1)f_d-2l(d-l)-2}{2(d-l)} \\
                            &\le \frac{(l+1)d-2l(d-l)-2}{2(d-l)} \\
                            &=   \frac{d-dl+2l^2-2}{2(d-l)} \\
                            &\le \frac{d-dl+2l\cdot\frac{d}{2}-2}{2(d-\frac{d}{2})} \\
                            &=   \frac{d-2}{d} \\
                            &<   1
    \end{align*}
    And since the left-hand side is an integer, we have $f_0-\frac{f_d}{2}-d \le 0$, and hence the result.

    Now suppose $d$ is odd. We now have $l \le \frac{d-1}{2}$, since for a connected triangulation $l=\frac{d+1}{2}$ is possible only when $f_d=1$. Again using \Cref{lem:general-bound-loops}:
    \begin{align*}
        f_0-\frac{f_d-1}{2}-d &\le \frac{d+1}{2(d-l)}f_d - \frac{d-l}{2(d-l)}(f_d-1) -l - \frac{1}{d-l} \\
                            &=   \frac{(l+1)f_d+(1-2l)(d-l)-2}{2(d-l)} \\
                            &\le \frac{(l+1)d+(1-2l)(d-l)-2}{2(d-l)} \\
                            &=   \frac{2d-dl+2l^2-l-2}{2(d-l)}
    \end{align*}
    Now suppose, for a contradiction, that the right-hand side above were greater than or equal to one. Then:
    \begin{align}
        2d-dl+2l^2-l-2 &\ge 2d-2l \nonumber \\
         2l^2+(1-d)l-2 &\ge 0 \label{eq:small-n-proof-contradiction}
    \end{align}
    Which, as a concave-up quadratic in $l$, occurs when $l \le l_-$ or $l \ge l_+$, where $l_\pm:=(d-1 \pm \sqrt{(d-1)^2+16})/4$ are its roots. However, we observe that:
    \begin{align*}
        l_- &= \frac{d-1-\sqrt{(d-1)^2+16}}{4} < \frac{d-1-\sqrt{(d-1)^2}}{4} = 0 \\
        l_+ &= \frac{d-1+\sqrt{(d-1)^2+16}}{4} > \frac{d-1+\sqrt{(d-1)^2}}{4} = \frac{d-1}{2}
    \end{align*}
    And hence $l_- < 0 \le l \le \frac{d-1}{2} < l_+$, so \Cref{eq:small-n-proof-contradiction} never occurs and we have our contradiction. So $f_0-\frac{f_d-1}{2}-d < 1$, and since the left-hand side is an integer we have $f_0-\frac{f_d-1}{2}-d \le 0$ and hence the result.
\end{proof}

\begin{proposition}
  \label{prop:odd}
  Let $\tri$ be a closed, connected $d$-dimensional triangulation, $d>0$ odd, with $f_0$ vertices and $f_d$ facets. If $f_d \ge d$ or $f_d=1$ then
  \begin{equation*}
    f_0 \leq f_d + \frac{d-1}{2}
  \end{equation*}
\end{proposition}
\begin{proof}
  If $f_d=1$, the result follows from \Cref{cor:fdone}.

  Now let $f_d > 1$. Let $l$ be the maximum number of loops on any node of $\Gamma(\tri)$. As observed in the proof of \Cref{prop:small-n}, $l \le \frac{d-1}{2}$. From \Cref{lem:general-bound-loops} we have:
  \begin{align*}
    f_0-f_d-\frac{d-1}{2} &\le \frac{d+1-2(d-l)}{2(d-l)}f_d + \frac{d}{2}-l - \frac{1}{d-l} + \frac{1}{2} \\
                          &\le \frac{(2l-d+1)f_d + d(d-l)-2l(d-l)-2+(d-l)}{2(d-l)} \\
                          &=   \frac{(2l-d+1)f_d + d^2 + 2l^2 - 3dl + d - l - 2}{2(d-l)}
  \end{align*}
  Now let $c:=f_d-d \ge 0$. Then:
  \begin{align*}
    f_0-f_d-\frac{d-1}{2} &=   \frac{(2l-d+1)(d+c) + d^2 + 2l^2 - 3dl + d - l - 2}{2(d-l)} \\
                          &=   \frac{2l^2 + 2cl -dl - l -dc+c + 2d - 2}{2(d-l)} \\
                          &\le \frac{l(d-1) + c(d-1) -dl - l -dc+c + 2d - 2}{2(d-l)} \\
                          &= \frac{2(d -l - 1)}{2(d-l)} \\
                          &<1
  \end{align*}
  And since $f_0-f_d-\frac{d-1}{2} \in \ZZ$ this implies $f_0-f_d-\frac{d-1}{2} \le 0$ and hence the result.
\end{proof}

The cases which remain to complete a proof of \Cref{conj:vertex-bound-arbdim} are as follows. First, the case $d\geq 4$ and even of \labelcref{item:main-odd} remains unproven -- progress on this part is discussed in detail in \Cref{sec:sufficient-conditions}. And second, the final remaining case for $d$ odd of part \labelcref{item:main-odd}: that \Cref{eq:main-odd} holds when $f_d$ is odd and $f_d < d$. \Cref{prop:small-n} proves that $f_0 \le \lfloor \frac{f_d}{2} \rfloor + d = \frac{f_d-1}{2} + d$, and \Cref{prop:construction-odd-large} shows that $f_0 = f_d + \frac{d-1}{2}$ is attainable. However, when $f_d < d$ we have $f_d + \frac{d-1}{2} < \frac{f_d-1}{2} + d$, leaving a gap between our bound and construction. We conjecture that the true maximum is in fact the smaller $f_0 = f_d + \frac{d-1}{2}$.

\begin{remark}
  In all of the odd-dimensional cases in which \Cref{conj:vertex-bound-arbdim} has been proven, the requirement that the underlying space is a manifold has no impact on the bound. That is, the bounds are proven to hold for any closed, connected triangulation, but are achievable by a sphere. So the maximum number of vertices possible in a $d$-sphere triangulation, a $d$-manifold triangulation and an arbitrary $d$-dimensional triangulation all coincide. This very much appears not to be the case in dimension $4$, and the aim of \Cref{sec:pseudomanifolds} is to examine this, both in dimension $4$ and in arbitrary even dimensions.
\end{remark}

\section{Sufficient conditions in even dimensions}
\label{sec:sufficient-conditions}

We now turn our attention to case of even dimensions, which turns out to be very different to -- and much more difficult than -- the odd-dimensional case. What we are aiming to prove is the following sub-case of \Cref{conj:vertex-bound-arbdim}, which we re-state as its own conjecture for clarity.

\begin{conjecture}
  \label{conj:vertex-bound-even}
  Let $d$ be an even integer, and let $\tri$ be a (generalised) triangulation of a closed and connected $d$-manifold $\manifold$ with $f$-vector $(f_0,f_1,\ldots ,f_d)$. Then
    \[ f_0 \le \frac{f_d}{2}+d \]
\end{conjecture}

We know from \Cref{prop:construction-even} that if \Cref{conj:vertex-bound-even} is true, this bound is the best possible in all even dimensions, and for all possible values of $f_d$.

For triangulations that are simplicial complexes, the Lower Bound Theorem \cite{Barnette73ProofLBCConvPoly,Kalai87RigidityLBT} implies $f_0  \leq \frac{f_d}{d} + d$. This means that \Cref{conj:vertex-bound-even} holds for combinatorial manifolds in all even dimensions $d$. In the remainder of this section we further extend the class of triangulations for which we can prove this conjecture.

\subsection{Balanced triangulations}
\label{sub:balanced}

We show that \Cref{conj:vertex-bound-even} holds for the class of balanced triangulations studied in {\em gem} or {\em crystallisation theory}. These objects give an alternative, graph-theoretic way of representing (pseudo)manifolds combinatorially, and have a rich theory and history in their own right \cite{bandieri01-crystallization,FerriGagliardi1982,Ferri1986}, but for our purposes it is most convenient to think of them as a special case of (generalised) triangulations.

\begin{definition}
  A $d$-dimensional triangulation $\tri$ is called {\em balanced} if every gluing uses the identity permutation. That is, in the notation of \Cref{sec:TRIG} every gluing is of the form $(\Delta_1,\Delta_2,i,\id)$. If $|\tri|$ is a PL-manifold, $\tri$ is called a {\em gem} or {\em graph-encoded manifold}, and a gem with $f_0=d+1$ is called a {\em crystallisation}. 
\end{definition}

This can equivalently be thought of as imposing a vertex colouring on the triangulation -- consider the vertex labels $\{0,1,\dots,d\}$ as colours, and only allow vertices to be identified if they are the same colour, resulting in a triangulation with $d+1$ vertex colours where every facet has exactly one vertex of each colour. Such a colouring induces an arc colouring on the dual graph, where each arc is coloured with the one vertex colour {\em not} involved in the corresponding gluing. This arc-coloured dual graph, which we call $G(\tri)$, uniquely determines the triangulation, and it is this object that the terms ``gem'' and ``crystallisation'' typically refer to in gem theory, hence the name {\em graph encoded manifold}.

\begin{theorem}
  \label{thm:vertex-bound-gem}
  Let $\tri$ be a closed, balanced, $d$-dimensional triangulation with $f$-vector $(f_0,f_1,\ldots ,f_d)$ and $d \ge 2$. Then
    \[ f_0 \le \frac{f_d}{2}+d \]
  Moreover, equality is achieved if and only if :
  \begin{enumerate}[label=(\alph*)]
    \item \label{item:gem-equality-cases} $\tri$ is a triangulation of $\mathbb{S}^d$ obtained from the standard pillow triangulation by repeated 0-2 vertex moves, if $d \ge 3$.
    \item $\tri$ is any balanced triangulation of $\mathbb{S}^d$, if $d=2$.
  \end{enumerate}
\end{theorem}

Note that the triangulations in part \labelcref{item:gem-equality-cases} are precisely the ones we constructed in \Cref{prop:construction-even}, see also \Cref{fig:0-2s-on-pillow}. The idea of the proof is similar to Lins and Mandel's proof that every (3-dimensional) manifold admits a crystallisation \cite{LinsMandel1985}, a fact originally proved by Pezzana \cite{Pezzana1974}. It requires one elementary move specific to gems, the {\em 1-dipole cancellation}.

\begin{definition}
  Let $\tri$ be a $d$-dimensional balanced triangulation, and $C \subseteq \{0,1,\dots,d\}$ a subset of colours. The {\em $C$-residues} of $G(\tri)$ are the connected components obtained after removing all arcs with colours in $C$.
\end{definition}

The $(\{0,1,\dots,d\} \setminus \{i\})$-residues of $G(\tri)$ (which we denote the {\em $\hat{\imath}$-residues}) are in one-to-one correspondence with the $i$-coloured vertices of $\tri$ \cite{Ferri1986}, with the number of nodes in the connected component corresponding to the degree of the vertex. The $\{i,j\}$-residues ($i \ne j$) are cycles of even length without repetition, alternating between colours $i$ and $j$, which we call {\em bicoloured cycles}.

\begin{definition}
  \label{def:dipole}
  Let $v,w$ be nodes of $G(\tri)$ joined by an arc $e$ of colour $i \in \{0,1,\dots,d\}$, such that $v$ and $w$ are contained in distinct $\hat{\imath}$-residues. The {\em $1$-dipole cancellation} along $e$ removes $v$, $w$ and all their incident arcs, and, for each $c \in \{0,1,\dots,d\} \setminus \{i\}$, adds a $c$-coloured arc between the other endpoints of the $c$-coloured arcs of $v$ and $w$. The inverse operation is called a {\em $1$-dipole creation}.
  Note that, necessarily, the $d$ arcs for a dipole creation must all be contained in the same residue.
  See \Cref{fig:1dipole} for an illustration.
\end{definition}

\begin{figure}[htb]
  \centering
\resizebox{\linewidth}{!}{
\begin{tikzpicture}[every node/.style={circle, inner sep=0, outer sep=0, minimum width=0.1cm, fill=black}, every path/.style={thick}, label distance=2mm, on grid]



  \node[fill=none] (c3) at (0,0) {} ;

  \node (ti3) [above=0.5 of c3] {} ;
  \node (bi3) [below=0.5 of c3] {} ;

  \node (tc3) [above=0.5 of ti3] {} ;
  \node (tl3) [above left=0.5 of ti3] {} ;
  \node (tr3) [above right=0.5 of ti3] {} ;

  \node (bc3) [below=0.5 of bi3] {} ;
  \node (bl3) [below left=0.5 of bi3] {} ;
  \node (br3) [below right=0.5 of bi3] {} ;

  \node (tcl3) [above left=0.5 and 0.2 of tc3] {} ;
  \node (tcr3) [above right=0.5 and 0.2 of tc3] {} ;
  \node (tcc3) [above=0.5 of tc3] {} ;

  \node (tll3) [above left=0.5 and 0.4 of tl3] {} ;
  \node (tlr3) [above=0.5 and 0.2 of tl3] {} ;
  \node (tlc3) [above left=0.5 and 0.2 of tl3] {} ;

  \node (trl3) [above=0.5 of tr3] {} ;
  \node (trr3) [above right=0.5 and 0.4 of tr3] {} ;
  \node (trc3) [above right=0.5 and 0.2 of tr3] {} ;

  \node (bcl3) [below left=0.5 and 0.2 of bc3] {} ;
  \node (bcr3) [below right=0.5 and 0.2 of bc3] {} ;
  \node (bcc3) [below=0.5 of bc3] {} ;

  \node (bll3) [below left=0.5 and 0.4 of bl3] {} ;
  \node (blr3) [below=0.5 and 0.2 of bl3] {} ;
  \node (blc3) [below left=0.5 and 0.2 of bl3] {} ;

  \node (brl3) [below=0.5 of br3] {} ;
  \node (brr3) [below right=0.5 and 0.4 of br3] {} ;
  \node (brc3) [below right=0.5 and 0.2 of br3] {} ;

  \node[fill=none] (lspace3) [left=of c3] {} ;
  \node[fill=none] (rspace3) [left=of c3] {} ;

  \draw[\ca] (ti3) -- (bi3) ;

  \draw[\cc] (ti3) -- (tl3) ;
  \draw[\cb] (ti3) -- (tr3) ;
  \draw[\cd] (ti3) -- (tc3) ;

  \draw[\cc] (bi3) -- (bl3) ;
  \draw[\cb] (bi3) -- (br3) ;
  \draw[\cd] (bi3) -- (bc3) ;

  \draw[\ca] (tl3) -- (tll3) ;
  \draw[\cd] (tl3) -- (tlc3) ;
  \draw[\cb] (tl3) -- (tlr3) ;

  \draw[\cc] (tc3) -- (tcl3) ;
  \draw[\ca] (tc3) -- (tcc3) ;
  \draw[\cb] (tc3) -- (tcr3) ;

  \draw[\cc] (tr3) -- (trl3) ;
  \draw[\cd] (tr3) -- (trc3) ;
  \draw[\ca] (tr3) -- (trr3) ;

  \draw[\ca] (bl3) -- (bll3) ;
  \draw[\cd] (bl3) -- (blc3) ;
  \draw[\cb] (bl3) -- (blr3) ;

  \draw[\cc] (bc3) -- (bcl3) ;
  \draw[\ca] (bc3) -- (bcc3) ;
  \draw[\cb] (bc3) -- (bcr3) ;

  \draw[\cc] (br3) -- (brl3) ;
  \draw[\cd] (br3) -- (brc3) ;
  \draw[\ca] (br3) -- (brr3) ;


  \node[fill=none,outer sep=0.5] (altop) at (1.25,0.05) {} ;
  \node[fill=none,outer sep=0.5] (artop) at (2.25,0.05) {} ;

  \node[fill=none,outer sep=0.5] (albot) at (1.25,-0.05) {} ;
  \node[fill=none,outer sep=0.5] (arbot) at (2.25,-0.05) {} ;

  \draw[->,black] (altop) -- (artop) {} ;
  \draw[->,black] (arbot) -- (albot) {} ;

  \node[fill=none] at (1.75,0.3) {\small{cancellation}} ;
  \node[fill=none] at (1.75,-0.3) {\small{creation}} ;


  \node[fill=none] (c4) at (3.5,0) {} ;

  \node[fill=none] (ti4) [above=0.5 of c4] {} ;
  \node[fill=none] (bi4) [below=0.5 of c4] {} ;

  \node (tc4) [above=0.5 of ti4] {} ;
  \node (tl4) [above left=0.5 of ti4] {} ;
  \node (tr4) [above right=0.5 of ti4] {} ;

  \node (bc4) [below=0.5 of bi4] {} ;
  \node (bl4) [below left=0.5 of bi4] {} ;
  \node (br4) [below right=0.5 of bi4] {} ;

  \node (tcl4) [above left=0.5 and 0.2 of tc4] {} ;
  \node (tcr4) [above right=0.5 and 0.2 of tc4] {} ;
  \node (tcc4) [above=0.5 of tc4] {} ;

  \node (tll4) [above left=0.5 and 0.4 of tl4] {} ;
  \node (tlr4) [above=0.5 and 0.2 of tl4] {} ;
  \node (tlc4) [above left=0.5 and 0.2 of tl4] {} ;

  \node (trl4) [above=0.5 of tr4] {} ;
  \node (trr4) [above right=0.5 and 0.4 of tr4] {} ;
  \node (trc4) [above right=0.5 and 0.2 of tr4] {} ;

  \node (bcl4) [below left=0.5 and 0.2 of bc4] {} ;
  \node (bcr4) [below right=0.5 and 0.2 of bc4] {} ;
  \node (bcc4) [below=0.5 of bc4] {} ;

  \node (bll4) [below left=0.5 and 0.4 of bl4] {} ;
  \node (blr4) [below=0.5 and 0.2 of bl4] {} ;
  \node (blc4) [below left=0.5 and 0.2 of bl4] {} ;

  \node (brl4) [below=0.5 of br4] {} ;
  \node (brr4) [below right=0.5 and 0.4 of br4] {} ;
  \node (brc4) [below right=0.5 and 0.2 of br4] {} ;

  \node[fill=none] (lspace4) [left=of c4] {} ;
  \node[fill=none] (rspace4) [left=of c4] {} ;

  \draw[\cc] (bl4) -- (tl4) ;
  \draw[\cb] (br4) -- (tr4) ;
  \draw[\cd] (bc4) -- (tc4) ;

  \draw[\ca] (tl4) -- (tll4) ;
  \draw[\cd] (tl4) -- (tlc4) ;
  \draw[\cb] (tl4) -- (tlr4) ;

  \draw[\cc] (tc4) -- (tcl4) ;
  \draw[\ca] (tc4) -- (tcc4) ;
  \draw[\cb] (tc4) -- (tcr4) ;

  \draw[\cc] (tr4) -- (trl4) ;
  \draw[\cd] (tr4) -- (trc4) ;
  \draw[\ca] (tr4) -- (trr4) ;

  \draw[\ca] (bl4) -- (bll4) ;
  \draw[\cd] (bl4) -- (blc4) ;
  \draw[\cb] (bl4) -- (blr4) ;

  \draw[\cc] (bc4) -- (bcl4) ;
  \draw[\ca] (bc4) -- (bcc4) ;
  \draw[\cb] (bc4) -- (bcr4) ;

  \draw[\cc] (br4) -- (brl4) ;
  \draw[\cd] (br4) -- (brc4) ;
  \draw[\ca] (br4) -- (brr4) ;



  \node[fill=none] (c1) at (7,0) {} ;
  \node[\cb] (l1) [left=of c1] {} ;
  \node[\cc] (r1) [right=of c1] {} ;

  \node[\ca] (t1) [above=0.7 of c1] {} ;
  \node[\ca] (b1) [below=of c1] {} ;

  \path (l1) arc (180:0:1 and 0.6) node[\cd,pos=-0.55,auto,anchor=center] (sm1) {} ;
  \fill[\ca,opacity=0.5] (l1.center) -- (r1.center) -- (sm1.center) -- cycle;

  \draw (l1) -- (r1) ;
  \draw (l1) -- (sm1) ;
  \draw (sm1) -- (r1) ;

  \draw[dash pattern=on 2pt off 1pt] (l1) -- (t1) ;
  \draw (t1) -- (sm1) ;
  \draw[dash pattern=on 2pt off 1pt] (r1) -- (t1) ;

  \draw[dash pattern=on 2pt off 1pt] (l1) -- (b1) ;
  \draw (sm1) -- (b1) ;
  \draw[dash pattern=on 2pt off 1pt] (r1) -- (b1) ;

  \node[\cc] (tl1) [above left=1 and 0.2 of l1] {} ;
  \node[\cb] (tr1) [above right=1 and 0.2 of r1] {} ;
  \node[\cd] (tb1) [above=1.1 of c1] {} ;

  \node[\cc] (bl1) [below left=1.3 and 0.2 of l1] {} ;
  \node[\cb] (br1) [below right=1.3 and 0.2 of r1] {} ;
  \node[\cd] (bb1) [below=1.4 of c1] {} ;

  \node[fill=none] (hltb1) at (intersection of l1.center--tb1.center and t1.center--tl1.center) {} ;
  \node[fill=none] (hrtb1) at (intersection of r1.center--tb1.center and t1.center--tr1.center) {} ;

  \node[fill=none] (hlbb1) at (intersection of l1.center--bb1.center and b1.center--bl1.center) {} ;
  \node[fill=none] (hrbb1) at (intersection of r1.center--bb1.center and b1.center--br1.center) {} ;

  \draw (l1) -- (tl1);
  \draw (t1) -- (tl1);
  \draw (sm1) -- (tl1);

  \draw (r1) -- (tr1);
  \draw (t1) -- (tr1);
  \draw (sm1) -- (tr1);

  \draw[dash pattern=on 2pt off 1pt] (l1) -- (hltb1.center);
    \draw (hltb1.center) -- (tb1);
  \draw[dash pattern=on 2pt off 1pt] (r1) -- (hrtb1.center);
    \draw (hrtb1.center) -- (tb1);
  \draw[dash pattern=on 2pt off 1pt] (t1) -- (tb1);

  \draw (l1) -- (bl1);
  \draw (b1) -- (bl1);
  \draw (sm1) -- (bl1);

  \draw (r1) -- (br1);
  \draw (b1) -- (br1);
  \draw (sm1) -- (br1);

  \draw[dash pattern=on 2pt off 1pt] (l1) -- (hlbb1.center);
    \draw (hlbb1.center) -- (bb1);
  \draw[dash pattern=on 2pt off 1pt] (r1) -- (hrbb1.center);
    \draw (hrbb1.center) -- (bb1);
  \draw[dash pattern=on 2pt off 1pt] (b1) -- (bb1);


  \node[fill=none,outer sep=0.5] (altop) at (8.75,0.05) {} ;
  \node[fill=none,outer sep=0.5] (artop) at (9.75,0.05) {} ;

  \node[fill=none,outer sep=0.5] (albot) at (8.75,-0.05) {} ;
  \node[fill=none,outer sep=0.5] (arbot) at (9.75,-0.05) {} ;

  \draw[->,black] (altop) -- (artop) {} ;
  \draw[->,black] (arbot) -- (albot) {} ;

  \node[fill=none] at (9.25,0.3) {\small{cancellation}} ;
  \node[fill=none] at (9.25,-0.3) {\small{creation}} ;


  \node[fill=none] (c2) at (11.5,0) {} ;
  \node[\cb] (l2) [left=of c2] {} ;
  \node[\cc] (r2) [right=of c2] {} ;

  \path (l2) arc (180:0:1 and 0.6) node[\cd,pos=-0.55,auto,anchor=center] (sm2) {} ;
  \node[\ca] (x2) at (barycentric cs:l2=1,r2=1,sm2=1) {} ;

  \fill[\cd,opacity=0.5] (l2.center) -- (r2.center) -- (x2.center) -- cycle;
  \fill[\cc,opacity=0.5] (l2.center) -- (sm2.center) -- (x2.center) -- cycle;
  \fill[\cb,opacity=0.5] (sm2.center) -- (r2.center) -- (x2.center) -- cycle;

  \draw[dash pattern=on 2pt off 1pt] (l2) -- (r2) ;
  \draw (l2) -- (sm2) ;
  \draw (sm2) -- (r2) ;

  \draw[dash pattern=on 2pt off 1pt] (l2) -- (x2) ;
  \draw (sm2) -- (x2) ;
  \draw[dash pattern=on 2pt off 1pt] (r2) -- (x2) ;

  \node[\cc] (tl2) [above left=1 and 0.2 of l2] {} ;
  \node[\cb] (tr2) [above right=1 and 0.2 of r2] {} ;
  \node[\cd] (tb2) [above=1.1 of c2] {} ;

  \node[\cc] (bl2) [below left=1.3 and 0.2 of l2] {} ;
  \node[\cb] (br2) [below right=1.3 and 0.2 of r2] {} ;
  \node[\cd] (bb2) [below=1.4 of c2] {} ;

  \node[fill=none] (hrtb2) at (intersection of x2.center--tr2.center and r2.center--tb2.center) {} ;
  \node[fill=none] (hltb2) at (intersection of x2.center--tl2.center and l2.center--tb2.center) {} ;

  \node[fill=none] (hrbb2) at (intersection of x2.center--br2.center and r2.center--bb2.center) {} ;
  \node[fill=none] (hlbb2) at (intersection of x2.center--bl2.center and l2.center--bb2.center) {} ;

  \node[fill=none] (hxbl2) at (intersection of x2.center--bl2.center and sm2.center--l2.center) {} ;
  \node[fill=none] (hxbr2) at (intersection of x2.center--br2.center and sm2.center--r2.center) {} ;

  \draw (l2) -- (tl2);
  \draw (x2) -- (tl2);
  \draw (sm2) -- (tl2);

  \draw (r2) -- (tr2);
  \draw (x2) -- (tr2);
  \draw (sm2) -- (tr2);

  \draw[dash pattern=on 2pt off 1pt] (l2) -- (hltb2.center);
    \draw (hltb2.center) -- (tb2);
  \draw[dash pattern=on 2pt off 1pt] (r2) -- (hrtb2.center);
    \draw (hrtb2.center) -- (tb2);
  \draw (x2) -- (tb2);

  \draw (l2) -- (bl2);
  \draw[dash pattern=on 2pt off 1pt] (x2) -- (hxbl2.center);
    \draw (hxbl2.center) -- (bl2);
  \draw (sm2) -- (bl2);

  \draw (r2) -- (br2);
  \draw[dash pattern=on 2pt off 1pt] (x2) -- (hxbr2.center);
    \draw (hxbr2.center) -- (br2);
  \draw (sm2) -- (br2);

  \draw[dash pattern=on 2pt off 1pt] (l2) -- (hlbb2.center);
    \draw (hlbb2.center) -- (bb2);
  \draw[dash pattern=on 2pt off 1pt] (r2) -- (hrbb2.center);
    \draw (hrbb2.center) -- (bb2);
  \draw (x2) -- (bb2);
\end{tikzpicture}
}
\caption{Effect of a 1-dipole move on a 3-dimensional balanced triangulation, from the perspective of coloured graphs (left) and triangulations (right).\label{fig:1dipole}}
\end{figure}

$1$-dipole cancellations and creations do not change the PL type of $|\tri|$ \cite{FerriGagliardi1982}. To see more explicitly how a 1-dipole creation is performed, we first make the following observation:

\begin{lemma}
  Let $G$ be a connected $h$-regular, $h$-arc-coloured graph. To separate $G$ into more than one connected component by removing arcs, one must either:
  \begin{enumerate}[label=(\alph*)]
    \item remove at least two arcs of the same colour, or
    \item remove at least one arc of every colour
  \end{enumerate} 
  Moreover, if $G$ is separated into multiple components by removing exactly one arc of each colour, then there are exactly two components, and each removed arc has one endpoint in each component.
\end{lemma}
\begin{proof}
  For (a) and (b), suppose we remove arcs $e_0,e_1,\dots,e_{h-2}$, of distinct colours $i_0,i_1,\dots,i_{h-2}$ respectively, to obtain a graph $G'$, and let $c$ be the one colour not removed. It is sufficient to prove that $G'$ is connected, so suppose by contradiction it is disconnected. Let $G_1$ be one connected component of $G'$, and $G_2$ the subgraph consisting of all other connected components. Then there is an arc joining $G_1$ to $G_2$ in $G$, but not in $G'$. So one of the arcs $e_k$ must have endpoints $v_1 \in G_1$ and $v_2 \in G_2$, and there are no $c$-coloured arcs or other $i_k$-coloured arcs with an endpoint in both $G_1$ and $G_2$. Now consider the bicoloured cycle of colours $i_k,c$ in $G$ containing $v_1$. This cycle contains $v_2$, and so there are two distinct paths from $v_1$ to $v_2$ consisting only of arcs of colours $i_k$ and $c$, one which is just the arc $e_k$ and one which does not include $e_k$. But this is a contradiction, since $v_1 \in G_1$, $v_2 \in G_2$, and $e_k$ is only one arc of colour $i_k$ or $c$ between $G_1$ and $G_2$.

  Finally, suppose that we remove one arc $e_0,e_1,\dots,e_{h-1}$ of every colour $i_0,i_1,\dots,i_{h-1}$, and proceed as above, supposing $G'$ is disconnected and defining subgraphs $G_1,G_2$. If any arc $e_j$ does not have endpoints in both $G_1$ and $G_2$, we can repeat the same argument replacing $c$ with $i_j$ to obtain a contradiction. Hence the only way for $G'$ to be disconnected is if every removed arc has one endpoint in $G_1$ and one in $G_2$. And since, by the above, removing only $h-1$ of these $h$ arcs always results in a connected $G'$, we see that $G'$ is obtained from a connected graph by removing just one arc, and so cannot have more than two connected components.
\end{proof}

With this is mind, we can think of a 1-dipole creation as follows. For a colour $i$ and an $\hat{\imath}$-residue $R$ of $G(\tri)$, find a set of $d$ arcs in $R$, one of each colour, such that removing them disconnects $R$, necessarily into two connected components $R_1$ and $R_2$. Remove these $d$ arcs, add two nodes $v_1,v_2$ joined by an $i$-coloured arc, and for each removed arc add two new arcs of the same colour, joining $v_1$ to its endpoint in $R_1$ and $v_2$ to its endpoint in $R_2$.

We now begin the proof of \Cref{thm:vertex-bound-gem}. The key point of the argument is the proof of \Cref{lem:vertex-bound-gem-dipole}, which proves the bound itself and characterises equality in terms of 1-dipole moves. The rest of the work to be done is simply reducing this classification (which a priori is broader than the original statement) to a stronger one in terms of 0-2 vertex moves for $d \ge 3$.

\begin{lemma}
  Let $\tri$ be a closed, balanced, $d$-dimensional triangulation with $f$-vector $(f_0,f_1,\ldots ,f_d)$ and $d \ge 2$. Then any 1-dipole cancellation on $\tri$ reduces $f_0$ by one and $f_d$ by two.
  \label{lem:dipole-fvector}
\end{lemma}
\begin{proof}
  Let the cancellation be along arc $e$ of colour $i$ with endpoints $v$ and $w$. By definition, the move removes two nodes (facets) and merges two distinct $\hat{\imath}$-residues ($i$-coloured vertices) into one. We need only check that it does not affect $f_0$ in any other way, i.e. it does not separate one residue into two, merge two residues of other colours, or destroy a residue entirely.

  First, we show it does not destroy a residue $R$ entirely. This can only occur if the only nodes of $R$ are the two removed, i.e. $R$ consists of $v$ and $w$ with $d$ parallel arcs between them. Since by assumption $d \ge 2$, there must be at least one arc of a colour $c \ne i$ joining $v$ and $w$, which implies $v$ and $w$ are in the same $\hat{\imath}$-residue, a contradiction.

  Now fix a colour $c$ and let $v_j$ ($w_j$) be the other endpoint of the $j$-coloured arc at $v$ ($w$) for each $j \in \{0,1,\dots,d\} \setminus \{i,c\}$. We claim all $v_j$ are in the same $\hat{c}$-residue $R_c$, both before and after the move. Before the move this is clear -- if $d=2$ it is trivial as there is only one $v_j$, and if $d \ge 2$ there is a path $v_{j_1} \rightarrow v \rightarrow v_{j_2}$ following arcs of colours $j_1,j_2$ for each $j_1,j_2 \ne c$. Moreover, this forms part of a bicoloured cycle, so either $v_{j_1}=v_{j_2}$ (and the result is trivial) or there is another distinct path of colours $j_1,j_2$ joining $v_{j_1}$ and $v_{j_2}$. This path is necessarily contained in $R_c$, and it cannot involve any of the $w_j$ since this implies $v_{j_1}$ and $w_j$ are in the same $\hat{\imath}$-residue. Hence this path remains as-is after the cancellation, and $v_{j_1},v_{j_2}$ are still in the same $\hat{c}$-residue. The same argument holds to show all $w_j$ are in the same $\hat{c}$-residue $S_c$.

  If $c=i$, $R_c$ and $S_c$ are distinct before the cancellation, but the same after. If $c \ne i$, then they are the same before (since they are joined by arc $e$) and also the same after (since any pair $v_j,w_j$ is connected by a $j$-coloured arc). So indeed the move reduces the number of $\hat{\imath}$-residues by exactly one, and preserves the number of residues for all other colours.
\end{proof}

\begin{remark}
  For $d=1$, the only closed, balanced triangulations are cycles of even length with alternating $0$- and $1$-coloured vertices, and the corresponding graphs are cycles of the same length with alternating $0$- and $1$-coloured edges. In this case clearly $f_0=f_d$, and any 1-dipole cancellation simply reduces the length of the cycle by 2. A 1-dipole cancellation along (for example) a 0-coloured arc both merges a pair of $\hat{0}$-residues and removes a $\hat{1}$-residue (the 0-coloured arc) entirely, a situation only possible in dimension 1 because a 1-dipole is also a set of $d$ parallel arcs.
\end{remark}

\begin{lemma}
  Let $\tri$ be a closed, balanced, $d$-dimensional triangulation with $f$-vector $(f_0,f_1,\ldots ,f_d)$ and $d \ge 2$. Then
    \[ f_0 \le \frac{f_d}{2}+d \]
  Moreover, equality is achieved if and only if $\tri$ is a triangulation of $\mathbb{S}^d$ obtained from the standard pillow triangulation by repeated 1-dipole creations.
  \label{lem:vertex-bound-gem-dipole}
\end{lemma}
\begin{proof}
  Any closed $d$-dimensional balanced triangulation must have at least $d+1$ vertices (since it has $d+1$ distinct vertex colours) and at least 2 facets (since a closed triangulation with only one facet requires a non-identity gluing). In the case $f_0=d+1$, we have $f_0=\frac{2}{2}+d \le \frac{f_d}{2}+d$, with equality when $f_d=2$.

  If $f_0 > d+1$, then it must have at least two vertices of the same colour, say colour $i$. That is, $G(\tri)$ has at least two $\hat{\imath}$-residues. If we take one such residue $R$, there must be some node $v_1 \in R$ such that the $i$-coloured arc $e$ from $v$ has an endpoint $v_2$ in a different $\hat{\imath}$-residue, since $G(\tri)$ is a connected graph. Hence, we can perform a $1$-dipole cancellation along $e$. By \Cref{lem:dipole-fvector}, this reduces $f_0$ by one and $f_d$ by two.

  We can repeat this process $f_0-(d+1)$ times, until we reach a balanced triangulation $\tri'$ with $f_0'=d+1$ vertices and $f_d'=f_d-2(f_0-(d+1))$ facets. As above, we have $f_0' \le \frac{f_d'}{2}+d$, and hence:
  \begin{align*}
    f_0 = f_0' + (f_0-(d+1)) = f_0' + \frac{f_d-f_d'}{2} \le \frac{f_d'}{2}+d+\frac{f_d-f_d'}{2} = \frac{f_d}{2}+d
  \end{align*}
  Finally, we have equality if and only if $f_d'=2$. There is only one possible closed, balanced, $d$-dimensional triangulation with only two facets (and $(d+1)$ vertices), that is, two facets glued by the identity along all of their faces, the standard pillow triangulation of $\mathbb{S}^d$. So the cases of equality are precisely those triangulations which can be reduced to the pillow triangulation by only 1-dipole cancellations, or equivalently obtained from it by only 1-dipole creations.
\end{proof}

\begin{lemma}
  For $d\ge 3$, every $d$-dimensional (balanced) triangulation obtained from the standard pillow triangulation of $\mathbb{S}^d$ by a series of 1-dipole creations can also be obtained from it using only 0-2 vertex moves.
  \label{lem:dipole-vs-02}
\end{lemma}
\begin{proof}
  We begin with two important observations:

  \begin{enumerate}[label=(\alph*)]
    \item \label{item:twodipoles} Any triangulation obtained from the pillow by a nonzero number of 0-2 vertex moves contains at least two sets of $d$ parallel arcs in its dual graph.
  
    This can be easily seen by induction on the number of facets: up to isomorphism there is only one such triangulation with 4 facets by symmetry of the pillow itself (see \Cref{fig:0-2s-on-pillow}, top row second from the left for the case $d=4$), and it contains two sets of $d$ parallel arcs. For the induction step, note that performing a 0-2 vertex move always creates a set of $d$ parallel arcs, and can destroy at most one existing set (see \Cref{fig:0-2-move} for the effect of a 0-2 vertex move on the dual graph). 
    \item Any set of $d$ parallel arcs in a gem, which we call a {\em $d$-dipole}, is a valid site for a 2-0 vertex move.
  
    A $d$-dipole in a gem represents two facets glued by the identity along all but one ridge, which is precisely the condition for a 2-0 vertex move. \label{item:dipoleobs}
  \end{enumerate}
  
  \smallskip

  We prove the statement by induction over the number $k$ of $1$-dipole moves performed on the $2$-facet pillow. The case $k=0$ is trivial. For $k \geq 1$, assume that all triangulations $\tri$ obtained from the $2$-facet pillow by $k$ creations of $1$-dipoles, can also be obtained from it by $k$ 0-2 vertex moves. 

  Let $\tri'$ be obtained from $\tri$ by a 1-dipole creation. We must show that $\tri'$ is obtainable by 0-2 vertex moves from the $2$-facet pillow. For this, it is sufficient to show that $\tri'$ contains at least one $d$-dipole. Since if it does, we can perform a 2-0 vertex move (by observation \labelcref{item:dipoleobs}) to obtain a triangulation $\tri''$ with the same number of facets as $\tri$. All moves used to obtain $\tri''$ from the pillow preserve the relation $f_0 = \frac{f_d}{2}+d$, so by \Cref{lem:vertex-bound-gem-dipole}, $\tri''$ is obtainable from the pillow by 1-dipole creations, and hence also by 0-2 vertex moves by the induction hypothesis.

  We know $\tri$ has two distinct $d$-dipoles by observation \labelcref{item:twodipoles}. We simply check all cases for the location of the 1-dipole creation relative to these $d$-dipoles, and ensure that in all cases $\tri'$ still has a $d$-dipole. When performing the move, let $i$ be the chosen colour for the $\hat{\imath}$-residue $R$, and $S$ the set of $d$ chosen arcs which separate $R$. For one of the $d$-dipoles of $\tri$, let $D$ be the set of $d$ arcs appearing in the dipole, $c$ the one colour not appearing in it, $e_1,e_2$ the two $c$-coloured arcs adjacent to it, and $v_1,v_2$ the endpoints of the dipole shared with $e_1$ and $e_2$ respectively.

  \smallskip
  
  \noindent
  \textbf{Case 1:} $S$ does not contain $e_1,e_2$ or any element of $D$.
  The move does not affect any node or arc of the dipole, so it remains as is in $\tri'$.  

  \smallskip

  \noindent
  \textbf{Case 2:} $i=c$.
  $S$ must contain at least one element $e \in D$, or else we are in Case 1. But the $\hat{c}$-residue $R$ which $e$ lies in is just the dipole itself, so the only way to choose $d$ arcs from $R$ is to choose $S=D$. This not only preserves the original dipole but actually creates an additional one, see \Cref{fig:1-dipole-proof-case}.

  \smallskip

  \noindent
  \textbf{Case 3A:} $i\ne c$, and one of $e_1,e_2$ is in $S$. Without loss of generality, say $e_1 \in S$.

  Suppose no element of $D$ is in $S$. Then all elements of $D$ remain in $\tri'$, with the only change that $e_1$ is replaced by an arc to one end of the new 1-dipole, so the $d$-dipole remains intact.

  Hence, assume that at least one element of $D$ is also in $S$. Then $v_1$ and $v_2$ must end up in different $\hat{\imath}$-residues of $\tri'$. In particular they cannot have any arcs of colours besides $i$ joining them in $\tri'$, so all other elements of $D$ except the $i$-coloured arc must also be in $S$. Hence, $S$ consists of all arcs incident to $v_1$ except the $i$-coloured one, so when $S$ is removed and $R$ separates into two connected components, one component is the isolated vertex $v_1$. Hence $d$ parallel arcs are added between $v_1$ and one end of the dipole, and $\tri'$ still has a $d$-dipole.

  \smallskip

  \noindent
  \textbf{Case 3B:} $i\ne c$, and the $c$-coloured arc of $S$ is neither $e_1$ nor $e_2$. Call it $e^\ast$.

  If $S$ contains no element of $D$ we are again in Case 1, so at least one element of $D$ is in $S$. By the same reasoning as in the previous case, $S$ must then contain all elements of $D$ except the $i$-coloured one.

  Now we make use of the second $d$-dipole in $\tri$. If $e^\ast$ is neither contained in nor adjacent to this dipole, then we are in the same situation as Case 1 but for the second dipole instead. If $e^\ast$ is adjacent to the dipole, then we are in Case 3A but for the second dipole. 

  Finally, if $e^\ast$ is contained in the second dipole, then we split into two more sub-cases. If there is no $i$-coloured arc in the second dipole, then we are in Case 2 for the second dipole. If there is, then we are in Case 3 for the second dipole, and by the same reasoning as before $S$ must contain all arcs of the second dipole of colours besides $i$, or none of them. But this is a contradiction, since we know $S$ contains an element of $D$ and an arc $e^\ast$ of the second dipole, and it cannot contain $d-1$ arcs of both dipoles (since $2(d-1)>d$ for $d\ge 3$).
\end{proof}

\begin{figure}
\centering
\begin{tikzpicture}[every node/.style={circle, inner sep=0, outer sep=0, minimum width=0.1cm, fill=black}, every path/.style={thick}, label distance=2mm, centremark/.style={decoration={markings, mark=at position 0.5 with {\arrow{#1}}}, postaction={decorate}}]

\node (bl1) at (0,0) {} ;
\node (tl1) [above=of bl1] {} ;
\node (tr1) [right=of tl1] {} ;
\node (br1) [below=of tr1] {} ;

\draw[\ca] (bl1) -- (tl1) ;
\draw[\ca] (br1) -- (tr1) ;

\draw[\cb,centremark={|}] (tl1) to[bend left=50] (tr1) ;
\draw[\cc,centremark={|}] (tl1) -- (tr1) ;
\draw[\cd,centremark={|}] (tl1) to[bend right=50] (tr1) ;

\node[fill=none,outer sep=0.5] (altop) at (2,0.7) {} ;
\node[fill=none,outer sep=0.5] (artop) at (3,0.7) {} ;

\draw[->,black] (altop) -- (artop) {} ;

\node[fill=none] at (2.5,1) {\small{1-dipole}} ;
\node[fill=none] at (2.5,0.4) {\small{creation}} ;

\node (bl2) at (4,0) {} ;
\node (tl2) [above=of bl2] {} ;
\node (tli2) [right=of tl2] {} ;
\node (tri2) [right=of tli2] {} ;
\node (tr2) [right=of tri2] {} ;
\node (br2) [below=of tr2] {} ;

\draw[\ca] (bl2) -- (tl2) ;
\draw[\ca] (br2) -- (tr2) ;

\draw[\cb] (tl2) to[bend left=50] (tli2) ;
\draw[\cc] (tl2) -- (tli2) ;
\draw[\cd] (tl2) to[bend right=50] (tli2) ;

\draw[\ca,centremark={|}] (tli2) -- (tri2) ;

\draw[\cb] (tri2) to[bend left=50] (tr2) ;
\draw[\cc] (tri2) -- (tr2) ;
\draw[\cd] (tri2) to[bend right=50] (tr2) ;

\end{tikzpicture}
\caption{Illustration of Case 2 of the proof of \Cref{lem:dipole-vs-02}, showing the result of a 1-dipole creation which uses all arcs of a $d$-dipole (here $d=3$). Arcs removed to separate a residue (left) and the arc between the newly created nodes (right) are marked with a bar.}
\label{fig:1-dipole-proof-case}
\end{figure}

\begin{figure}
\centering
\begin{tikzpicture}[every node/.style={circle, inner sep=0, outer sep=0, minimum width=0.1cm, fill=black}, every path/.style={thick}, label distance=2mm, centremark/.style={decoration={markings, mark=at position 0.5 with {\arrow{#1}}}, postaction={decorate}}]

\node (bl1) at (0,0) {} ;
\node (bil1) [above=of bl1] {} ;
\node (til1) [above=of bil1] {} ;
\node[fill=none] (tl1) [above=of til1] {} ;

\node (br1) [right=of bl1] {} ;
\node (bir1) [above=of br1] {} ;
\node (tir1) [above=of bir1] {} ;
\node[fill=none] (tr1) [above=of tir1] {} ;

\draw[\cc] (til1) -- (bil1) ;
\draw[\cb] (bil1) -- (bl1) ;
\draw[\cc] (tir1) -- (bir1) ;
\draw[\cb] (bir1) -- (br1) ;

\draw[\ca] (til1) to[bend right=30] (tir1) ;
\draw[\cb,centremark={|}] (til1) to[bend left=30] (tir1) ;

\draw[\ca] (bil1) -- (bir1) ;

\draw[\cc,centremark={|}] (bl1) to[bend right=30] (br1) ;
\draw[\ca] (bl1) to[bend left=30] (br1) ;

\node[fill=none,outer sep=0.5] (altop) at (1.75,1.5) {} ;
\node[fill=none,outer sep=0.5] (artop) at (2.75,1.5) {} ;

\draw[->,black] (altop) -- (artop) {} ;

\node[fill=none] at (2.25,1.8) {\small{1-dipole}} ;
\node[fill=none] at (2.25,1.2) {\small{creation}} ;

\node (bl2) at (4,0) {} ;
\node (bil2) [above=of bl2] {} ;
\node (til2) [above=of bil2] {} ;
\node (tl2) [above=of til2] {} ;

\node (br2) [right=of bl2] {} ;
\node (bir2) [above=of br2] {} ;
\node (tir2) [above=of bir2] {} ;
\node (tr2) [above=of tir2] {} ;

\draw[\cb] (tl2) -- (til2) ;
\draw[\cc] (til2) -- (bil2) ;
\draw[\cb] (bil2) -- (bl2) ;
\draw[\cb] (tr2) -- (tir2) ;
\draw[\cc] (tir2) -- (bir2) ;
\draw[\cb] (bir2) -- (br2) ;

\draw[\cc] (tl2) to[bend right=30] (bl2) ;
\draw[\cc] (tr2) to[bend left =30] (br2) ;

\draw[\ca,centremark={|}] (tl2) -- (tr2) ;
\draw[\ca] (til2) -- (tir2) ;
\draw[\ca] (bil2) -- (bir2) ;
\draw[\ca] (bl2) -- (br2) ;

\end{tikzpicture}
\caption{A 1-dipole creation which takes a balanced triangulation obtainable from the pillow by 0-2 vertex moves (left) to one which is not (right), by reducing the number of $d$-dipoles by 2. This is only possible for $d=2$. Arcs removed to separate a residue (left) and the arc between the newly created nodes (right) are marked with a bar.}
\label{fig:1-dipole-kills-two-d-dipoles}
\end{figure}

To complete the proof of \Cref{thm:vertex-bound-gem}, all that remains is to consider the equality cases when $d=2$. But for $d=2$ we already know that the triangulations of closed 2-manifolds satisfying $f_0 = \frac{f_2}{2}+2$ are precisely all triangulations of $\mathbb{S}^2$ (see \Cref{sub:conjecture}). Any 2-dimensional closed triangulation is automatically a manifold, so the equality cases are all balanced triangulations of $\mathbb{S}^2$, completing the proof of \Cref{thm:vertex-bound-gem} and incidentally also proving (by \Cref{lem:vertex-bound-gem-dipole}) that all such triangulations can be obtained by 1-dipole creations on the pillow triangulation. This is {\em not} true of 0-2 vertex moves, precisely because the final sub-case in the proof of \Cref{lem:dipole-vs-02} fails for $d=2$, allowing a single 1-dipole creation to reduce the number of 2-dipoles by two (see \Cref{fig:1-dipole-kills-two-d-dipoles}).

\begin{remark}
  \Cref{thm:vertex-bound-gem}, together with \Cref{cor:actualBound}, immediately gives a lower bound of $2 \beta_2 (\manifold)$ on the number of pentachora in a balanced triangulation of a simply connected closed $4$-manifold $\manifold$. However, in the balanced setting there already exists a stronger and tight lower bound of $6\beta_2(\manifold)+2$ \cite{Basak14Crystallizations,CasaliCristoforiGagliardi20}.
\end{remark}

\subsection{Sufficient conditions on the dual graph}
\label{sub:dual-graph-conditions}

Let $\tri$ be a triangulation of $\manifold$. In this section, we show that \Cref{conj:vertex-bound-even} holds under certain sufficient conditions on the dual graph $\Gamma(\tri)$. The proof is purely combinatorial and does not use the requirement for $\manifold$ to be a manifold.

First, recall the decomposition of a graph into 2-connected components. Since we work with both simple and non-simple graphs, it is more convenient for our purposes to use the two following separate notions of this decomposition, using the terminology of \textit{cut} and \textit{separating vertices} as in \cite{BondyMurty2008}.

\begin{definition}
  Let $\Gamma$ be a connected graph.
  \begin{enumerate}[label=(\arabic*)]
     \item A node $v \in V(\Gamma)$ is a {\em cut node} if removing it from $\Gamma$ results in a disconnected graph. A graph is {\em 2-connected} if it has no cut nodes, and a {\em 2-connected component} of $\Gamma$ is a maximal 2-connected subgraph. The {\em block-cut tree} of $\Gamma$ is the graph with a node for each cut node and each 2-connected component of $\Gamma$, and an arc from each cut node to each component containing it.
     \item A node $v \in V(\Gamma)$ is a {\em separating node} if there are subgraphs $\Gamma_1$ and $\Gamma_2$, such that every arc of $\Gamma$ is contained in exactly one of these graphs and $v$ is their only common node. A graph is {\em non-separable} if it has no separating nodes, and a {\em non-separable component} of $\Gamma$ is a maximal non-separable subgraph. The {\em block-separating tree} of $\Gamma$ is the graph with a node for each separating node and each non-separable component of $\Gamma$, and an arc from each separating node to each component containing it. See \Cref{fig:blocksep} for an illustration.
  \end{enumerate}
\end{definition}

The first definition coincides with the more typical definitions of ``cut vertices'' and ``blocks'' used for loopless graphs (for example in \cite{Harary2018}). A cut node is also a separating node, and a node is a separating node if and only if $(a)$ it is a cut node or $(b)$ it has a loop and at least one other non-loop arc incident to it (or both). Hence for loopless graphs these two notions of the decomposition coincide. For general graphs, there is exactly one non-separable component for each 2-connected component and each loop. Moreover, any two non-separable components have at most one node in common, so the block-cut and block-separating trees are genuinely trees \cite{BondyMurty2008}.

\begin{remark}
  \label{rmk:2-connected-comp-has-2-nodes}
  In a connected graph with at least two nodes, any 2-connected component must also have at least two nodes. A one-node subgraph is always contained in a subgraph consisting of two nodes joined by a single arc, which is 2-connected, and hence a one-node subgraph cannot be a (maximal) 2-connected component.
\end{remark}

The main combinatorial property we are interested in is a quantity we call the {\em branching number}.

\begin{definition}
  Let $\Gamma$ be a connected graph, and $L$ be the number of leaves of its block-separating tree. Then the {\em branching number} $\branch(\Gamma)$ is:
  \[
    \branch(\Gamma) := \begin{cases}
      1 \textrm{ if } \Gamma \textrm{ is non-separable and } |V(\Gamma)|=1 \\
      2 \textrm{ if } \Gamma \textrm{ is non-separable and } |V(\Gamma)| \ge 2  \\
      L \textrm{ otherwise}
    \end{cases}
  \]
\end{definition}

\begin{figure}
\centering
\resizebox{0.8\linewidth}{!}{

\begin{tabular}{ccc}

\begin{tikzpicture}[every node/.style={circle, inner sep=0, outer sep=0, minimum width=0.15cm, fill=black}, every path/.style={thick}, label distance=2mm]

\node[fill=none] (hexC) at (0,0) {} ;
\path (hexC) ++(90:1) node[red] (hex1) {} ;
\path (hexC) ++(150:1) node (hex2) {} ;
\path (hexC) ++(210:1) node (hex3) {} ;
\path (hexC) ++(270:1) node (hex4) {} ;
\path (hexC) ++(330:1) node[red] (hex5) {} ;
\path (hexC) ++(30:1) node  (hex6) {} ;

\path (hex1) ++(45:1.5) node[fill=none] (eyeC) {} ;
\path (eyeC) ++(45:0.75) node  (eyein1) {} ;
\path (eyeC) ++(135:0.75) node (eyein2) {} ;
\path (eyeC) ++(225:0.75) node (eyein3) {} ;
\path (eyeC) ++(315:0.75) node (eyein4) {} ;
\path (eyeC) ++(45:1.5) node[red]  (eyeout1) {} ;
\path (eyeC) ++(135:1.5) node (eyeout2) {} ;
\path (eyeC) ++(315:1.5) node[red] (eyeout4) {} ;

\path (eyeout1) ++(45:0.5) node[fill=none] (triuuC) {} ;
\path (triuuC) ++(345:0.5) node (triuu1) {} ;
\path (triuuC) ++(105:0.5) node (triuu2) {} ;

\path (hex1) ++(135:0.5) node[fill=none] (bridgeC) {} ;
\path (bridgeC) ++(135:0.5) node[red] (bridge1) {} ;

\path (bridge1) ++(135:0.5) node[fill=none] (triluC) {} ;
\path (triluC) ++(75:0.5) node (trilu1) {} ;
\path (triluC) ++(195:0.5) node (trilu2) {} ;

\path (bridge1) ++(225:0.5) node[fill=none] (trillC) {} ;
\path (trillC) ++(165:0.5) node (trill1) {} ;
\path (trillC) ++(285:0.5) node (trill2) {} ;

\path (hex5) ++(0:0.7) node[fill=none] (looplrC) {} ;
\path (eyeout4) ++(0:0.7) node[fill=none] (loopurC) {} ;
\path (eyeout4) ++(-90:0.7) node[fill=none] (loopubC) {} ;

\draw (hex1) -- (hex2) ;
\draw (hex1) -- (hex3) ;
\draw (hex1) -- (hex5) ;
\draw (hex1) -- (hex6) ;
\draw (hex4) -- (hex2) ;
\draw (hex4) to[bend left=15] (hex3) ;
\draw (hex4) to[bend right=15] (hex3) ;
\draw (hex4) -- (hex5) ;
\draw (hex4) to[bend left=15] (hex6) ;
\draw (hex4) to[bend right=15] (hex6) ;
\draw (hex5) to[scale=2,out=30,in=-30,loop] (hex5) ;

\draw (hex1) -- (eyeout2) ;
\draw (hex1) -- (eyeout4) ;
\draw (hex1) -- (eyein3) ;
\draw (eyein1) -- (eyein2) ;
\draw (eyein2) -- (eyein3) ;
\draw (eyein3) -- (eyein4) ;
\draw (eyein4) -- (eyein1) ;
\draw (eyeout1) to[bend left=15] (eyeout2) ;
\draw (eyeout1) to[bend right=15] (eyeout2) ;
\draw (eyeout1) -- (eyein1) ;
\draw (eyeout1) -- (eyeout4) ;
\draw (eyeout4) to[scale=2,out=-30,in=30,loop] (eyeout4) ;
\draw (eyeout4) to[scale=2,out=300,in=230,loop] (eyeout4) ;

\draw (eyeout1) -- (triuu1) ;
\draw (eyeout1) -- (triuu2) ;
\draw (triuu1) -- (triuu2) ;

\draw (hex1) to[bend left=15] (bridge1) ;
\draw (hex1) to[bend right=15] (bridge1) ;

\draw (bridge1) -- (trilu1) ;
\draw (bridge1) -- (trilu2) ;
\draw (trilu1) -- (trilu2) ;

\draw (bridge1) -- (trill1) ;
\draw (bridge1) -- (trill2) ;
\draw (trill1) -- (trill2) ;
\draw (trill1) to[bend left=30] (trill2) ;
\draw (trill1) to[bend right=30] (trill2) ;

\end{tikzpicture}

&

\begin{tikzpicture}[every node/.style={circle, inner sep=0, outer sep=0, minimum width=0.15cm, fill=black}, every path/.style={thick}, label distance=2mm]

\node (hexC) at (0,0) {} ;
\path (hexC) ++(90:1) node[red] (hex1) {} ;
\path (hexC) ++(150:1) node[fill=none] (hex2) {} ;
\path (hexC) ++(210:1) node[fill=none] (hex3) {} ;
\path (hexC) ++(270:1) node[fill=none] (hex4) {} ;
\path (hexC) ++(330:1) node[red] (hex5) {} ;
\path (hexC) ++(30:1) node[fill=none]  (hex6) {} ;

\path (hex1) ++(45:1.5) node (eyeC) {} ;
\path (eyeC) ++(45:0.75) node[fill=none]  (eyein1) {} ;
\path (eyeC) ++(135:0.75) node[fill=none] (eyein2) {} ;
\path (eyeC) ++(225:0.75) node[fill=none] (eyein3) {} ;
\path (eyeC) ++(315:0.75) node[fill=none] (eyein4) {} ;
\path (eyeC) ++(45:1.5) node[red]  (eyeout1) {} ;
\path (eyeC) ++(135:1.5) node[fill=none] (eyeout2) {} ;
\path (eyeC) ++(315:1.5) node[red] (eyeout4) {} ;

\path (eyeout1) ++(45:0.5) node (triuuC) {} ;
\path (triuuC) ++(345:0.5) node[fill=none] (triuu1) {} ;
\path (triuuC) ++(105:0.5) node[fill=none] (triuu2) {} ;

\path (hex1) ++(135:0.5) node (bridgeC) {} ;
\path (bridgeC) ++(135:0.5) node[red] (bridge1) {} ;

\path (bridge1) ++(135:0.5) node (triluC) {} ;
\path (triluC) ++(75:0.5) node[fill=none] (trilu1) {} ;
\path (triluC) ++(195:0.5) node[fill=none] (trilu2) {} ;

\path (bridge1) ++(225:0.5) node (trillC) {} ;
\path (trillC) ++(165:0.5) node[fill=none] (trill1) {} ;
\path (trillC) ++(285:0.5) node[fill=none] (trill2) {} ;

\path (hex5) ++(0:0.7) node (looplrC) {} ;
\path (eyeout4) ++(0:0.7) node (loopurC) {} ;
\path (eyeout4) ++(-90:0.7) node (loopubC) {} ;

\draw (hexC) -- (hex1) ;
\draw (hexC) -- (hex5) ;

\draw (looplrC) -- (hex5) ;

\draw (eyeC) -- (hex1) ;
\draw (eyeC) -- (eyeout1) ;
\draw (eyeC) -- (eyeout4) ;

\draw (triuuC) -- (eyeout1) ;

\draw (loopurC) -- (eyeout4) ;

\draw (loopubC) -- (eyeout4) ;

\draw (bridgeC) -- (hex1) ;
\draw (bridgeC) -- (bridge1) ;

\draw (triluC) -- (bridge1) ;
\draw (trillC) -- (bridge1) ;

\end{tikzpicture}

&

\

\end{tabular}
}
\caption{An example of a graph (left) and its block-separating tree (right). Separating nodes are highlighted in red in both. The black nodes of the block-separating tree correspond to non-separable components of the original graph.\label{fig:blocksep}}
\end{figure}

For the purposes of \Cref{conj:vertex-bound-even}, we are always able to assume without loss of generality that we are working in the loopless case, where we need only consider cut nodes and the block-cut tree, due to the following result.

\begin{lemma}
  \label{lem:loop-removal}
  Let $\tri$ be a closed and connected $d$-dimensional triangulation, $d$ even. Then there exists a triangulation $\tri'$ of the same space with no loops in $\Gamma(\tri')$ such that 
  \begin{enumerate}[label=(\alph*)]
    \item the block-separating tree of $\Gamma(\tri)$ is isomorphic to the block-cut tree of $\Gamma(\tri')$, and
    \item $\delta_{\tri'}=\delta_{\tri}$, where $\delta_\tri := f_0 - \frac{f_d}{2}$ (as in \Cref{eq:delta} for $d=4$).
  \end{enumerate}
  In particular, $\branch(\Gamma(\tri))=\branch(\Gamma(\tri'))$, and \Cref{conj:vertex-bound-even} holds for $\tri$ if and only if it holds for $\tri'$.
\end{lemma}

\begin{proof}
  If $\Gamma(\tri)$ has no loops then there is nothing to prove, so we assume $\Gamma$ has at least one loop. We also know that $\Gamma(\tri)$ must contain at least two nodes since a closed even dimensional triangulation has an even number of pentachora. So every node which has a loop is a separating node, and moreover $\Gamma(\tri)$ is not non-separable (hence $\branch(\Gamma(\tri))$ is simply the number of leaves of the block-separating tree).

  A $0$-$2$ vertex move adds two facets and one vertex to a triangulation, so in particular it does not change $\delta_{\tri}$. Moreover, a $0$-$2$ move applied to a ridge corresponding to a loop arc in $\Gamma(\tri)$ turns this loop into a 3-node ``ringpull'' consisting of a single arc, followed by an arc of multiplicity $d$, followed by another single arc to the starting node (see \Cref{fig:loop-removal}). In particular, it removes the loop from the dual graph. The ringpull is a non-separable component and its addition turns $v$ into a cut node (if it was not already). Hence $v$ remains a separating node and the block-separating tree is unchanged, with the node of the tree corresponding to the loop replaced with one corresponding to the ringpull, both of which have $v$ as their only neighbour.

  Applying $0$-$2$ vertex moves to all loop edges of $\Gamma(\tri)$ hence produces a triangulation $\tri'$ whose dual graph has no loops, an unchanged block-separating tree (which is identical to the block-cut tree since $\Gamma(\tri)$ is loopless), and $\delta_{\tri'}=\delta_{\tri}$, as required.
\end{proof}

\begin{figure}
\centering

\begin{tikzpicture}[every node/.style={circle, inner sep=0, outer sep=0, minimum width=0.1cm, fill=black}, every path/.style={thick}, label distance=2mm]

\node (c1) at (0,0) {} ;
\path (c1) ++(234:1) node (bl1) {} ;
\path (c1) ++(270:1) node (bm1) {} ;
\path (c1) ++(306:1) node (br1) {} ;

\draw[scale=3] (c1) to[out=54,in=126,loop] (c1) ;
\draw (c1) -- (bl1) ;
\draw (c1) -- (bm1) ;
\draw (c1) -- (br1) ;

\node[fill=none,outer sep=0.5] (al) at (2,0) {} ;
\node[fill=none,outer sep=0.5] (ar) at (3,0) {} ;

\draw[->] (al.90) -- (ar.90) {} ;
\draw[->] (ar.270) -- (al.270) {} ;

\node[fill=none] at (2.5,0.4) {0-2 (vertex)} ;
\node[fill=none] at (2.5,-0.4) {2-0 (vertex)} ;

\node (c2) at (5,0) {} ;
\path (c2) ++(234:1) node (bl2) {} ;
\path (c2) ++(270:1) node (bm2) {} ;
\path (c2) ++(306:1) node (br2) {} ;

\path (c2) ++(126:1) node (tl2) {} ;
\path (c2) ++(54:1) node (tr2) {} ;

\draw (c2) -- (bl2) ;
\draw (c2) -- (bm2) ;
\draw (c2) -- (br2) ;
\draw (c2) -- (tl2) ;
\draw (c2) -- (tr2) ;

\draw (tl2) to[bend left=15] (tr2) ;
\draw (tl2) to[bend right=15] (tr2) ;
\draw (tl2) to[bend left=40] (tr2) ;
\draw (tl2) to[bend right=40] (tr2) ;

\end{tikzpicture}

\caption{Replacing a loop with a ``ringpull" via a $0$-$2$ vertex move, in dimension 4.}
\label{fig:loop-removal}
\end{figure}

The basic idea of our conditional proof of \Cref{conj:vertex-bound-even} is similar to what we have already seen in the proof of \Cref{prop:linear-bound}: build up the triangulation by adding on facets one-by-one, keeping track of how many vertices are added in each step. However, we now want to consider all of the arcs of the dual graph, not just those forming a spanning tree, and ask if there is a more ``optimal'' order with which to add the pentachora to guarantee a small number of added vertices. The requirement on our ordering is that it has a small number of what we call \textit{critical points} in a way reminiscent of (discrete) Morse theory, and our next goal is to show that the minimum number of such points is characterised by the branching number.

\begin{definition}
  Let $\Gamma$ be a connected graph, and $S=(v_1,v_2,\dots,v_{\ell})$ a sequence of distinct nodes of $\Gamma$. For two nodes $v,w$ in $S$, we write $v < w$ ($v > w$) if $v$ appears before (after) $w$ in the sequence.

  We call a node $v$ of $S$ a \textit{source} (resp. \textit{sink}) if for every neighbour $w$ of $v$ which is in $S$, $w > v$ (resp. $w < v$), and a \textit{critical point} if it is either a source or a sink. We set:
  \begin{align*}
    \crit(S) &:= \#\{v \in V(\Gamma) \ | \ v \textrm{ is a critical point of } S\} \\
    \crit(\Gamma) &:= \min_{S, \ |S|=|V(\Gamma)|} \crit(S)
  \end{align*}
\end{definition}

Note that, as long as $S$ and $\Gamma$ contain at least two nodes, we always have $\crit(S),\crit(\Gamma) \ge 2$, since the first node in the order is always a source, and the last is always a sink.

\begin{lemma}
  \label{lem:2-connected-implies-2-crit-points}
  If $\Gamma$ is 2-connected and has at least two nodes, then $\crit(\Gamma)=2$. Moreover, for any two distinct nodes $v_1,v_2$ of $\Gamma$, there exists a sequence starting at $v_1$ and ending at $v_2$ which realises $\crit(\Gamma)$.
\end{lemma}
\begin{proof}
  We aim to construct a sequence $S$ containing all nodes of $\Gamma$ with $\crit(S)=2$, starting with $v_1$ and ending with $v_2$. Initialise $S$ as the sequence of nodes of $\Gamma$ along a path between $v_1$ and $v_2$ (which must exist since $\Gamma$ is connected). By construction, every node of $S$, except $v_1$ and $v_2$, has a neighbour both before and after it in the sequence, so $\crit(S)=2$. Let $R$ be the induced subgraph of $\Gamma$ consisting of all remaining nodes not in $S$.

  Since $\Gamma$ is connected, there must exist some node $v$ of $R$ which (as a node of $\Gamma$) has a neighbour $w$ in $S$. Let $R_v$ be the connected component of $v$ in $R$. We claim that at least one node $v'$ of $R_v$ must have a neighbour $w'$ in $S$ which is distinct from $w$. For if no such node exists, then every arc between $R_v$ and $S$ has $w$ as an endpoint, and hence $w$ is a cut node of $\Gamma$, which contradicts the 2-connectedness of $\Gamma$.

  Take a path from $v$ to $v'$ in $R_v$, and remove it from $R$. If $w'$ appears after $w$ in $S$, then insert this path into $S$ immediately after $w$. Similarly, if $w'$ appears before $w$, then insert the path in reverse order immediately after $w'$. Then every newly inserted node has a neighbour both before and after it in $S$, and we still have $\crit(S)=2$. Moreover, the path is always inserted immediately after a node which was not the last node in the sequence, and no nodes are inserted before $v_1$ or after $v_2$. Repeating this process until there are no remaining nodes in $R$, we obtain our desired sequence and the statement follows.
  \end{proof}

\begin{lemma}
  \label{lem:every-leaf-contains-a-crit}
  Let $\Gamma$ be a connected graph, and $\Sigma$ a 2-connected component of $\Gamma$ containing only one cut node $v$. Then in any node sequence $S$ containing all nodes of $\Gamma$ exactly once, at least one node of $V(\Sigma) \setminus \{v\}$ is a critical point.
\end{lemma}
\begin{proof}
  Since $\Sigma$ contains no other cut nodes, removing $v$ must disconnect $\Sigma$ from the rest of $\Gamma$. In particular, every node of $V(\Sigma) \setminus \{v\}$ has all of its neighbours in $\Sigma$. We know $\Gamma$ must have more than one node since it contains $\Sigma$, and by \Cref{rmk:2-connected-comp-has-2-nodes} $\Sigma$ has at least two nodes.

  Let $w$ and $w'$ be the first and last nodes of $V(\Sigma) \setminus \{v\}$, respectively, in the order they appear in $S$. Naturally, since $w \leq w'$ in $S$, either $v<w'$ or $v>w$ must hold. If $v < w'$, then $w'$ must be a sink, since all of its possible neighbours appear before it in the ordering. Similarly, if $v > w$ then $w$ must be a source. Altogether, for every possible position of $v$ in the order, at least one of $w$ and $w'$ must be critical.
\end{proof}

\begin{lemma}
  \label{lem:branch-equals-crit}
  Let $\Gamma$ be a connected graph without loops. Then
  \begin{enumerate}[label=(\alph*)]
    \item \label{item:branch-equals-crit} $\branch(\Gamma)=\crit(\Gamma)$; and
    \item \label{item:one-source} there exists a sequence $S$ containing all nodes of $\Gamma$ such that $\crit(S)=\crit(\Gamma)$ and $S$ has only one source.
  \end{enumerate}
\end{lemma}
\begin{proof}
  First, when $\Gamma$ consists of only one node, there is only one possible sequence of nodes of length one (with one source) and we trivially have $\branch(\Gamma)=\crit(\Gamma)=1$. If $\Gamma$ has at least two nodes and is 2-connected, then by definition $\branch(\Gamma)=2$, and the result follows by \Cref{lem:2-connected-implies-2-crit-points}. Hence, we may assume $\Gamma$ has at least two nodes and is not 2-connected. In particular, it has at least one cut node, and every 2-connected component contains a cut node and at least two nodes overall.

  Let $S$ be a node sequence containing all of the nodes of $\Gamma$. By \Cref{lem:every-leaf-contains-a-crit}, each 2-connected component of $\Gamma$ which is a leaf of the block-cut tree contains a node which is a critical point and is not its (single) cut node. This node is not contained in any other 2-connected component. Hence $\crit(\Gamma) \ge \branch(\Gamma)$.

  We now construct a sequence $S$ containing all nodes of $\Gamma$ with $\branch(\Gamma)$ critical points and only one source, showing $\crit(\Gamma) \le \branch(\Gamma)$. This will prove both \labelcref{item:branch-equals-crit} and \labelcref{item:one-source}.

  Choose a 2-connected component $\Sigma_1$ which is a leaf of the block-cut tree, and order the remaining 2-connected components $\Sigma_2,\Sigma_3,\dots,\Sigma_k$ in the order they are visited in a breadth-first search through the block-cut tree rooted at $\Sigma_1$. Observe that every $\Sigma_i$, $i>1$, has exactly one cut node which is shared with one or more components which appear before it in this order, and every $\Sigma_i$ which is not a leaf also shares a cut node (distinct from the first) with one which appears after. 
  
  By \Cref{lem:2-connected-implies-2-crit-points}, for each $i$ we have a sequence $S_i$ containing all nodes of $\Sigma_i$ which have two critical points. Since (again, by \Cref{lem:2-connected-implies-2-crit-points}) we are free to choose the starting and ending nodes, we enforce that

  \begin{itemize}
    \item in $\Sigma_1$ the last node is the (unique) cut node contained in $\Sigma_1$;
    \item in every other leaf of the block-cut tree the first node is the (unique) cut node contained in it; and
    \item in every $\Sigma_i$ which is not a leaf, the first node is the (unique) cut node appearing in a component before it in the order, and the last node is a cut node appearing in a component after it in the order.
  \end{itemize}

  Initialise sequence $S$ as $S_1$, and insert the sequences $S_2,\dots,S_k$ into $S$ one-by-one as follows. At each step, the first node $v$ of the next $S_i$ is already in the sequence (and we note it is not the first node of $S$, since $S_1$ contains at least two nodes and only the last is a cut node). We insert the remaining nodes of $S_i$ in order into $S$ immediately after $v$.

  The non-critical nodes of each $S_i$ remain non-critical in $S$, and the only possible critical nodes are the first and last node of each $S_i$. The first node of every $S_i$ except $S_1$ is non-critical in $S$, as it is also contained in another $S_j$ where it is not the first node. The last node of every non-leaf $S_i$ is also non-critical, since it is the first node of some other $S_j$:  it is contained in some $\Sigma_j$ which is added after $\Sigma_i$, and it must be the first node of $S_j$ because it appears in a component before $\Sigma_j$. The remaining critical points are the first node of $\Sigma_1$ (which is a source) and the last node of each leaf component (which are all sinks), and hence $\crit(S)$ equals the number of leaves of the block-cut tree, that is $\crit(S)=\branch(\Gamma)$ as required.
\end{proof}

Intuitively, this result and its proof imply that any graph can be arranged as an ``antenna'' as in \Cref{fig:antenna}, where along every horizontal and vertical ``branch'' we obtain a sequence with only two critical points, and the antenna has exactly $\branch(\Gamma)$ ``spikes''.

The sequence with $\branch(\Gamma)$ critical points as constructed in the proof can be thought of as replacing this antenna with a tree tracing over it (where horizontal and vertical branches are genuine paths), and then taking a breadth-first search based at one of the leaves. The case $\crit(\Gamma)=2$ is exactly when this antenna is just a single horizontal line or, equivalently, consists of 2-connected components which are joined one-by-one in a sequence. Every time our graph ``branches off'' from this line or an existing branch (at a cut node joining more than two 2-connected components) we require one more critical point.

In fact, a more transparent equivalent definition of $\branch(\Gamma)$ for loopless graphs is the number of leaves in this traced-over tree. In most cases this coincides with the number of leaves of the block-cut tree, but when there is only a single 2-connected component with at least two nodes we obtain two leaves even though the block-cut tree is trivial, explaining why we treat the cases when $\Gamma$ is non-separable the way we do.

\begin{figure}
\centering
\resizebox{\linewidth}{!}{

\begin{tikzpicture}[every node/.style={circle, inner sep=0, outer sep=0, minimum width=0.15cm, fill=black}, every path/.style={thick}, label distance=2mm]

\node (1-1) [label=below:$v_1$] at (0,0) {} ;
\node (1-2) [right=of 1-1,label=below:$v_2$] {} ;
\node (1-3) [right=of 1-2,label=below:$v_3$] {} ;
\node (1-4) [right=of 1-3,label=below:$v_4$] {} ;
\node (2-1) [right=of 1-4,label=below:$v_5$] {} ;
\node (2-2) [right=of 2-1,label=below:$v_6$] {} ;
\node (2-3) [right=of 2-2,label=below:$v_7$] {} ;
\node (2-4) [right=of 2-3,label=below:$v_8$] {} ;
\node (2-5) [right=of 2-4,label=below:$v_9$] {} ;
\node (2-6) [right=of 2-5,label=below:$v_{10}$] {} ;
\node (2-7) [right=of 2-6,label=below:$v_{11}$] {} ;
\node (5-1) [right=of 2-7,label=below:$v_{15}$] {} ;
\node (5-2) [right=of 5-1,label=below:$v_{16}$] {} ;

\node (3-1) [above=of 2-1,label=right:$v_{12}$] {} ;
\node (6-1) [above=of 3-1,label=right:$v_{17}$] {} ;
\node (6-2) [above=of 6-1,label=right:$v_{18}$] {} ;
\node (6-3) [above=of 6-2,label=right:$v_{19}$] {} ;
\node (6-4) [above=of 6-3,label=right:$v_{20}$] {} ;
\node (6-5) [above=of 6-4,label=right:$v_{21}$] {} ;
\node (9-1) [above=of 6-5,label=right:$v_{26}$] {} ;

\node (4-1) [below=of 2-5,label=right:$v_{13}$] {} ;
\node (4-2) [below=of 4-1,label=right:$v_{14}$] {} ;

\node (7-1)  [right=of 6-2,label=below:$v_{22}$] {} ;
\node (7-2)  [right=of 7-1,label=below:$v_{23}$] {} ;
\node (7-3)  [right=of 7-2,label=below:$v_{24}$] {} ;
\node (11-1) [right=of 7-3,label=below:$v_{29}$] {} ;
\node (11-2) [right=of 11-1,label=below:$v_{30}$] {} ;

\node (8-1)  [left=of 6-3,label=below:$v_{25}$] {} ;
\node (12-1) [left=of 8-1,label=below:$v_{31}$] {} ;
\node (12-2) [left=of 12-1,label=below:$v_{32}$] {} ;

\node (10-1) [above=of 7-2,label=right:$v_{27}$] {} ;
\node (10-2) [above=of 10-1,label=right:$v_{28}$] {} ;

\draw[red] (1-1) to[bend left=12] (1-2) ;
\draw[red] (1-1) to[bend right=12] (1-2) ;
\draw[red] (1-1) to[bend left=20] (1-3) ;
\draw[red] (1-1) to[bend left=30] (1-3) ;
\draw[red] (1-1) to[bend right=20] (1-4) ;

\draw[red] (1-2) to[bend left=12] (1-3) ;
\draw[red] (1-2) to[bend right=12] (1-3) ;
\draw[red] (1-2) to[bend right=15] (1-4) ;

\draw[red] (1-3) -- (1-4) ;

\draw[blue] (1-4) -- (2-1) ;
\draw[blue] (1-4) to[bend left=20] (2-7) ;

\draw[blue] (2-1) to[bend left=12] (2-2) ;
\draw[blue] (2-1) to[bend right=12] (2-2) ;
\draw[blue] (2-1) to[bend left=15] (2-3) ;
\draw[red] (2-1) -- (3-1) ;

\draw[teal] (2-2) to[scale=3,out=-60,in=-120,loop] (2-2) ;
\draw[blue] (2-2) to[bend left=15] (2-7) ;

\draw[blue] (2-3) -- (2-4) ;
\draw[blue] (2-3) to[bend left=15] (2-4) ;
\draw[blue] (2-3) to[bend right=15] (2-4) ;
\draw[blue] (2-3) to[bend left=20] (2-5) ;

\draw[blue] (2-4) to[bend right=20] (2-6) ;
\draw[blue] (2-4) to[bend right=30] (2-6) ;

\draw[blue] (2-5) to[bend left=12] (2-6) ;
\draw[blue] (2-5) to[bend right=12] (2-6) ;
\draw[red] (2-5) -- (4-1) ;
\draw[red] (2-5) to[bend right=40] (4-2) ;

\draw[blue] (2-6) -- (2-7) ;

\draw[red] (2-7) -- (5-1) ;
\draw[red] (2-7) to[bend right=40] (5-2) ;

\draw[blue] (3-1) -- (6-1) ;
\draw[blue] (3-1) to[bend left=15] (6-2) ;
\draw[blue] (3-1) to[bend right=15] (6-3) ;
\draw[blue] (3-1) to[bend left=25] (6-5) ;

\draw[red] (4-1) to[bend left=10] (4-2) ;
\draw[red] (4-1) to[bend right=10] (4-2) ;
\draw[red] (4-1) to[bend left=30] (4-2) ;
\draw[red] (4-1) to[bend right=30] (4-2) ;

\draw[red] (5-1) to[bend left=10] (5-2) ;
\draw[red] (5-1) to[bend right=10] (5-2) ;
\draw[red] (5-1) to[bend left=30] (5-2) ;
\draw[red] (5-1) to[bend right=30] (5-2) ;

\draw[teal] (6-1) to[scale=3,out=150,in=210,loop] (6-1) ;
\draw[blue] (6-1) to[bend left=20] (6-4) ;
\draw[blue] (6-1) to[bend left=30] (6-4) ;

\draw[blue] (6-2) to[bend right=20] (6-4) ;
\draw[red] (6-2) to[bend left=15] (7-1) ;
\draw[red] (6-2) to[bend right=15] (7-1) ;
\draw[red] (6-2) to[bend left=20] (7-2) ;

\draw[teal] (6-3) to[scale=3,out=30,in=-30,loop] (6-3) ;
\draw[blue] (6-3) -- (6-4) ;
\draw[red] (6-3) -- (8-1) ;

\draw[blue] (6-4) -- (6-5) ;

\draw[teal] (6-5) to[scale=3,out=30,in=-30,loop] (6-5) ;
\draw[red] (6-5) -- (9-1) ;

\draw[teal] (7-1) to[scale=3,out=-60,in=-120,loop] (7-1) ;
\draw[red] (7-1) to[bend right=20] (7-3) ;

\draw[red] (7-2) to[bend left=12] (7-3) ;
\draw[red] (7-2) to[bend right=12] (7-3) ;
\draw[blue] (7-2) -- (10-1) ;
\draw[blue] (7-2) to[bend right=40] (10-2) ;

\draw[blue] (7-3) -- (11-1) ;
\draw[blue] (7-3) to[bend right=40] (11-2) ;

\draw[blue] (8-1) to[bend left=12] (12-1) ;
\draw[blue] (8-1) to[bend right=12] (12-1) ;
\draw[blue] (8-1) to[bend right=20] (12-2) ;
\draw[blue] (8-1) to[bend right=30] (12-2) ;

\draw[teal] (9-1) to[scale=3,out=60,in=120,loop] (9-1) ;
\draw[teal] (9-1) to[scale=3,out=150,in=210,loop] (9-1) ;

\draw[blue] (10-1) to[bend left=10] (10-2) ;
\draw[blue] (10-1) to[bend right=10] (10-2) ;
\draw[blue] (10-1) to[bend left=30] (10-2) ;
\draw[blue] (10-1) to[bend right=30] (10-2) ;

\draw[blue] (11-1) to[bend left=10] (11-2) ;
\draw[blue] (11-1) to[bend right=10] (11-2) ;
\draw[blue] (11-1) to[bend left=30] (11-2) ;
\draw[blue] (11-1) to[bend right=30] (11-2) ;

\draw[blue] (12-1) -- (12-2) ;
\draw[blue] (12-1) to[bend left=15] (12-2) ;
\draw[blue] (12-1) to[bend right=15] (12-2) ;

\end{tikzpicture}

}
\caption{Example of a 5-regular graph $\Gamma$ arranged as an ``antenna'', with nodes ordered as in the proof of \Cref{lem:branch-equals-crit}. Arcs belonging to each non-separable component are highlighted alternately in red, blue or green. Observe that each is arranged horizontally or vertically according to a sequence with two critical points. Since $\Gamma$ has loops, our sequence with $\branch(\Gamma)$ critical points first requires expanding them as in \Cref{fig:loop-removal}. The antenna has $13=\branch(\Gamma)$ ``spikes'' (noting loops count as spikes, unless they extend an existing spike like the loop above $v_{26}$), and the order of the nodes has seven critical points. After expanding the loops, six more critical points are needed.}  
\label{fig:antenna}
\end{figure}

\begin{remark}
  We can deduce from the results in this section that the converse of \Cref{lem:2-connected-implies-2-crit-points} is also true. That is, the existence of such a sequence between any two nodes is an equivalent characterisation of 2-connectedness. \Cref{lem:branch-equals-crit} implies that such a sequence between {\em some} two nodes exists if and only if the branch-cut tree either has exactly two leaves, or is 2-connected and has at least two nodes. If the branch-cut tree has two leaves, then it has two distinct 2-connected components which have only one cut vertex, and so (by \Cref{lem:every-leaf-contains-a-crit}) one node of each must always be critical. In particular we cannot have both endpoints of a $\crit(S)=2$ sequence in the same component. Hence the only case where we have free choice of the endpoints is when the graph is 2-connected.
\end{remark}

We are now ready to prove the main result of this section:

\begin{theorem} 
  \label{thm:branch-bound}
  Let $\tri$ be a closed and connected $d$-dimensional triangulation, $d$ even, with $f_0$ vertices, $f_d$ facets, and $\delta_\tri := f_0 - \frac{f_d}{2}$. Then
  \begin{align}
    \delta_{\tri} \le d + \left\lfloor\frac{\branch(\Gamma(\tri))-2}{d-1}\right\rfloor
  \end{align}
\end{theorem}
\begin{proof}
  We proceed in a similar way to the proof of \Cref{lem:general-bound-loops}, making use of the extra freedom when choosing the node ordering and \Cref{lem:branch-equals-crit} to refine the procedure.

  Let the number of vertices of $\tri$ be $f_0$ and the number of facets be $f_d$. By \Cref{lem:loop-removal} we may assume without loss of generality that $\Gamma(\tri)$ has no loops. Let $S$ be a sequence containing all nodes of $\Gamma(\tri)$ with $\branch(\Gamma(\tri))$ critical points, such that only the first node is a source. This exists by \Cref{lem:branch-equals-crit}. Label the nodes of $\Gamma(\tri)$ by $v_1,v_2,\dots,v_{f_d}$ in the order they appear in $S$.
  
  Let $\tri_k$ be the triangulation obtained from $\tri$ by removing the facets corresponding to $v_{k+1},v_{k+2},\dots,v_{f_d}$ and all of their gluings. We build up $\tri$ step-by-step as a sequence $\tri_1,\tri_2,\dots,\tri_{f_4}=\tri$, adding one facet each time. As in the proof of \Cref{lem:general-bound-loops}, we keep track of the number $f_0^k$ of vertices of $\tri_k$, and the size $C_k$ of the cut set
  \begin{align*}
    C_k := \#\{e \in E(\Gamma(\tri)) \ | \ e = \langle v_i,v_j \rangle \textrm{ with } i \le k < j \}
  \end{align*}
  We also recall the number of arcs going ``forward'' and ``backward'':
  \begin{align*}
    F_k &:= \#\{e \in E(\Gamma(\tri)) \ | \ e = \langle v_k,v_i \rangle \textrm{ with } i > k \} \\
    B_k &:= \#\{e \in E(\Gamma(\tri)) \ | \ e = \langle v_i,v_k \rangle \textrm{ with } i < k \}
  \end{align*} 

  We have $f_0^1 = d+1$ and $C_1=d+1$, and for each $k \ge 2$ (noting that $\Gamma$ has no loops)
  \begin{align*}
    C_k = C_{k-1} + F_k - B_k = C_{k-1} + (d+1-B_k) - B_k = C_{k-1} + d+1 - 2B_k
  \end{align*}

  Since none of $v_2,v_3,\dots,v_{f_4}$ are sources, we must have $B_k \ge 1$ for all $2 \le k \le f_4$. We have three cases.

  \smallskip
  
  \noindent
  \textbf{Case 1: $B_k = 1$.}
  In this case, $C_k=C_{k-1}+d-1$. $\tri_k$ is obtained from $\tri_{k-1}$ by adding one facet and one gluing between it and an existing facet. As a result of the gluing, $d$ of the $d+1$ new vertices are identified with existing vertices, and the remaining vertex is added to give $f_0^k = f_0^{k-1}+1$.

  \smallskip
  
  \noindent
  \textbf{Case 2: $2 \le B_k \le d$.}
  We have $C_k \ge C_{k-1}-(d-1)$, with equality when $B_k=d$. The new facet in $\tri_k$ has at least two gluings to existing facets from $\tri_{k-1}$. Each of the $d+1$ new vertices is involved in at least one of these gluings, and so is identified with an existing vertex. Thus, $f_0^k \le f_0^{k-1}$.

  \smallskip
  
  \noindent
  \textbf{Case 3: $B_k = d+1$.}
  This is precisely the case where $v_k$ is critical. We have $C_k = C_{k-1}-(d+1)$ and, as in the previous case, $f_0^k \le f_0^{k-1}$.

  \medskip

  Now let $X$, $Y$ and $Z$ be the number of nodes where case 1, 2 and 3 occurs, respectively. Then
  \begin{align*}
    0 = C_{f_d} &\ge  C_1    + (d-1)X  - (d-1)Y           - (d+1)Z \\
                &=    (d+1)  + (d-1)X  - (d-1)(f_d-1-X-Z) - (d+1)Z \\
                &=    2d     + 2(d-1)X - (d-1)f_d         - 2Z \\
                &=    2d     + 2(d-1)X - (d-1)f_d         - 2(\branch(\Gamma(\tri))-1) \\
                &=    2(d+1) + 2(d-1)X - (d-1)f_d         - 2\branch(\Gamma(\tri)) \\
    \Rightarrow X &\le \frac{f_d}{2} + \frac{\branch(\Gamma(\tri))}{d-1} - \frac{d+1}{d-1} \\
                  &= \frac{f_d}{2} + \frac{\branch(\Gamma(\tri))-2}{d-1} - 1
  \end{align*}

  On the other hand:
  \begin{align*}
    f_0 &\le f_0^{1} + X = d+1 + X \le \frac{f_d}{2} + \frac{\branch(\Gamma(\tri))-2}{d-1} + d
  \end{align*}

  And since $f_0$ is an integer and $f_d$ is even:
  \begin{align*}
                     f_0 &\le \left\lfloor \frac{f_d}{2} + \frac{\branch(\Gamma(\tri))-2}{d-1} + d \right\rfloor \\
                         &= \frac{f_d}{2} + \left\lfloor\frac{\branch(\Gamma(\tri))-2}{d-1}\right\rfloor + d \\
    \Rightarrow \delta_{\tri} &= f_0 - \frac{f_d}{2} \le d + \left\lfloor\frac{\branch(\Gamma(\tri))-2}{d-1}\right\rfloor
  \end{align*}
\end{proof}

\begin{corollary} 
  \label{cor:low-branch-implies-result}
  Let $\tri$ be a closed and connected $d$-dimensional triangulation, $d$ even. If $\branch(\Gamma(\tri)) \le d$, then the bound from \Cref{conj:vertex-bound-even} holds. In particular, when $d=4$, \Cref{conj:vertex-bound-even} holds if $\branch(\Gamma(\tri)) \le 4$.
\end{corollary}

Finally, we turn our attention to some simpler and more familiar (but weaker) sufficient conditions which follow from \Cref{thm:branch-bound}.

\begin{corollary}
  Let $\tri$ be a closed and connected $d$-dimensional triangulation, $d$ even. If $\Gamma(\tri)$ is non-separable, then the bound from \Cref{conj:vertex-bound-arbdim} holds.
\end{corollary}
\begin{proof}
  We have that $\Gamma(\tri)$ has at least two nodes, is 2-connected and has no loops. By \Cref{lem:2-connected-implies-2-crit-points,lem:branch-equals-crit}, $\branch(\Gamma(\tri))=2$ and the result follows immediately by \Cref{thm:branch-bound}.
\end{proof}

\begin{corollary}
  Let $\tri$ be a closed and connected $d$-dimensional triangulation, $d$ even. If $\Gamma(\tri)$ has no loops and there exists a spanning tree $\Sigma$ of $\Gamma(\tri)$ with $L$ leaves, then:
  \begin{align*}
    \delta_{\tri} \le d + \left\lfloor \frac{L-2}{d-1} \right\rfloor
  \end{align*}
  In particular, if $L \le d$ then the bound from \Cref{conj:vertex-bound-even} holds.
\end{corollary}

\begin{proof}
  Choose a leaf $\ell$ of $\Sigma$ and let $S$ be the sequence of nodes of $\Gamma(\tri)$ given by the order they are visited by a breath-first search through $\Sigma$ based at $\ell$. In this order, $\ell$ is a source, every other leaf is a sink, and every non-leaf node must have a neighbour which appears before and after it in the order and so is not critical. Hence $\crit(S)=L$. By \Cref{lem:branch-equals-crit} $\branch(\Gamma(\tri)) \le L$, and the result follows from \Cref{thm:branch-bound}.
\end{proof}

\section{Necessity of the Manifold Condition}
\label{sec:pseudomanifolds}

In \Cref{sub:dual-graph-conditions} we present a sufficient condition on the dual graph of a triangulation for \Cref{conj:vertex-bound-even} to hold. This result does not require the underlying space of the triangulation to be a manifold, but merely for the triangulation to be closed and connected. 

In this section, we show that \Cref{conj:vertex-bound-even} (and hence also \Cref{conj:vertex-bound-arbdim}) do not hold for all such triangulations. More specifically, we present examples of $4$-dimensional triangulations $\tri$ with underlying space not a manifold, such that $\delta_\tri = f_0 - \frac{f_d}{2} > 4$. The content of this section demonstrates the significance of the manifold condition in \Cref{conj:vertex-bound-even}: To prove these statements, the combinatorial methods of \Cref{sub:dual-graph-conditions} are not enough and, at some point, we must appeal to the topology of the underlying space. Along the way, we also pose a number of open questions (\Cref{q:exceed-3/4,q:exceed-9/16,q:smaller-than-14,q:1-nonsingular-counterexample}) which may be more straightforward to answer than proving \Cref{conj:vertex-bound-even}.

Naturally, the aim here is to find counterexamples to the bounds in \Cref{conj:vertex-bound-even} that are as ``manifold-like'' as possible. To make this notion more precise, we have the following definition.

\begin{definition}
  \label{def:pseudo}
  Let $\tri$ be a $d$-dimensional closed, connected triangulation. If $\tri$ does not have any faces (of any dimension) which are identified with themselves under a non-identity relation (for example, an edge identified with itself in reverse) then we call $\tri$ a {\em $d$-pseudomanifold
  (triangulation)}.

  An $s$-face with link in $\tri$ a standard PL $(d-1-s)$-sphere is called {\em nonsingular}, all other faces are called {\em singular}. A $d$-pseudomanifold is \textit{$s$-nonsingular}, or $N_s$, if all of its $s$-faces are nonsingular.
\end{definition}

In particular, an $N_0$ pseudomanifold is a manifold, and every $d$-pseudomanifold is $N_{d-1}$ because every ridge is contained in exactly one gluing, and so has link $\mathbb{S}^0$. In fact, this extends further to $N_{d-2}$.

\begin{proposition}
  The link of a $k$-face in a $d$-pseudomanifold triangulation, $k \leq d-2$, is connected.
  \label{prop:link-is-connected}
\end{proposition}
\begin{proof}
  Let $\Sigma$ be an $i$-face of a $d$-dimensional triangulation $\tri$, and $\delta,\delta'$ two facets of the link of $\Sigma$. We show that there is a sequence of gluings in the link of $\Sigma$ (a path in its dual graph) joining $\delta$ and $\delta'$.

  Our facets $\delta,\delta'$ correspond to $(d-i-1)$-faces $\tau,\tau'$ of facets $\Delta,\Delta'$ in $\tri$ respectively, where the faces $\sigma$ and $\sigma'$ opposite $\tau$ and $\tau'$ respectively are both in the class corresponding to $\Sigma$, and hence are identified in $\tri$. It follows that there must be a sequence of face gluings realising the identification of $\sigma$ with $\sigma'$. That is, there is a sequence of facets $\Delta=\Delta_1,\Delta_2,\dots,\Delta_k=\Delta'$, faces $\sigma=\sigma_1,\sigma_2,\dots,\sigma_k=\sigma'$, and gluings $g_1,g_2,\dots,g_{k-1}$, such that $\sigma_i$ is a face of $\Delta_i$ and $g_i$ is a gluing from $\Delta_i$ to $\Delta_{i+1}$ which identifies $\sigma_i$ with $\sigma_{i+1}$. Let $\tau_i$ be the face opposite $\sigma_i$ in $\Delta_i$, giving a sequence $\tau=\tau_1,\tau_2,\dots,\tau_k=\tau'$. Each $\tau_i$ corresponds to a facet in the link of $\Sigma$, and $g_i$ involves all vertices of $\sigma_i$ and hence involves all but one of the vertices of $\tau_i$ (i.e. a $(d-i-2)$-face of $\tau_i$). So, by construction of the link, there is a gluing in the link of $\Sigma$ between $\tau_i$ and $\tau_{i+1}$ for each $i=1,2,\dots,k-1$. Hence there is indeed a sequence of gluings joining $\tau$ and $\tau'$ in the link of $\Sigma$.
\end{proof}

\begin{corollary}
  Every $d$-pseudomanifold triangulation is $N_{d-2}$.
  \label{cor:always-d-2}
\end{corollary}
\begin{proof}
  Following \Cref{prop:link-is-connected}, the link of a $(d-2)$-face in a $d$-pseudomanifold triangulation is a connected, 1-dimensional, closed triangulation, and hence a closed cycle which is a triangulation of (the PL standard) $\mathcal{S}^1$.
\end{proof}

Finally, observe that if the link of every $s$-face is a $(d-1-s)$-sphere, then the link of every $(s+1)$-face is a $(d-2-s)$-sphere. Thus, if a pseudomanifold is $N_s$ it is also $N_{s'}$ for any $s' \ge s$. This directly leads to the following hierarchy.

\begin{align}\small
  \{\textrm{\,manifolds\,}\} = N_0 \subseteq N_1 \subseteq \dots \subseteq N_{d-2} = \{\textrm{\,pseudomanifolds\,}\} \subseteq \{\textrm{\,triangulations\,}\}
\end{align}

We want to construct counterexamples to the bounds in \Cref{conj:vertex-bound-even} that are $s$-nonsingular pseudomanifolds with $s$ as small as possible. The key points that we (constructively) prove in this section are as follows:

\begin{theorem}
\label{thm:counterexamples}
  For any $d$ even and $f_d \equiv 2 \textrm{ mod } d$, there exists a $d$-dimensional triangulation with $f_d$ facets such that
  \begin{align*}
    f_0 = \frac{d-1}{d}f_d+\frac{d+2}{d}
  \end{align*}
  In particular, there exists a 2-nonsingular 4-pseudomanifold with 14 pentachora such that
  \begin{align*}
    f_0 > \frac{f_4}{2}+4
  \end{align*}
\end{theorem}

Many of the examples featured in this section are described without precise specification of the gluings involved. A complete description, in the form of code used to generate these examples in {\it Regina} \cite{regina}, can be found in \Cref{app:code}. The gluings are chosen to ensure the triangulations are pseudomanifolds, have a small number of non-spherical edge links, and low genus in those links.

By \Cref{cor:low-branch-implies-result}, any counterexample to the bound in \Cref{conj:vertex-bound-even} must have $\branch(\Gamma(\tri)) \ge d+1$. Hence a natural starting point is to consider triangulations $\tri$ with very large values for $\branch(\Gamma(\tri))$. This can most effectively be achieved in triangulations with many nodes of the dual graph being in the maximum possible number of $d/2$ loops. This way, a single node increases $\branch(\Gamma(\tri))$ by $d/2$. 

In dimension 4, such nodes with two loops attached correspond to $1$-pentachoron triangulations of the $4$-ball, sometimes called ``double snapped balls'', which are obtained by gluing two pairs of boundary tetrahedra of a single pentachoron. The resulting complex has a one internal vertex, and a one-tetrahedron triangulation of the $3$-sphere in its boundary. There are two such triangulations, one ``folded'', with the $2$-vertex $1$-tetrahedron triangulation of $\mathbb{S}^3$ as its boundary, and one ``twisted'', with the $1$-vertex $1$-tetrahedron triangulation of $\mathbb{S}^3$ as its boundary. We denote these triangulations by $\textrm{DSB}_1$ and $\textrm{DSB}_2$ respectively (see \Cref{tab:double-loop-gluing}). It follows that  each copy of a double snapped ball in a triangulation adds one (internal) vertex with only one pentachoron, and naturally maximising the number of such copies pushes $f_0$ past $f_0/f_d \approx \frac{1}{2}$ and towards the trivial bound of $f_0/f_d \approx 1$ as in \Cref{prop:linear-bound}.

$\textrm{DSB}_1$ naturally generalises to an even-dimensional ``$\frac{d}{2}$-snapped ball'', a triangulation of the $d$-ball with an internal vertex, obtained by identifying ridge $i$ with ridge $d-i+1$ by the permutation which simply transposes $i$ and $d-i+1$, for each $i=1,2,\dots,\frac{d}{2}$. For the remainder of this section we only work with this version of a $\frac{d}{2}$-snapped ball. Note that its boundary, i.e., its one unglued ridge, is the triangulation $\mathbb{S}^{d-1}_1$ from the proof of \Cref{prop:construction-odd-large} with its vertices identified in pairs.

\begin{table}
\centering
\begin{tabular}{c}
\begin{tabular}{|c|r|r|r|r|r|}
\hline
Pentachoron & Tet 0123                        & Tet 0124           & Tet 0134           & Tet 0234           & Tet 1234                                       \\ \hline \hline
0           & -                               & 0 (3124)           & 0 (0234)           & 0 (0134)           & 0 (1204)                                       \\ \hline
\end{tabular}
\\
\\
\begin{tabular}{|c|r|r|r|r|r|}
\hline
Pentachoron & Tet 0123                        & Tet 0124           & Tet 0134           & Tet 0234           & Tet 1234                                       \\ \hline \hline
0           & -                               & 0 (3124)           & 0 (2304)           & 0 (3014)           & 0 (1204)                                       \\ \hline
\end{tabular}
\end{tabular}
\caption{Gluing tables for $\textrm{DSB}_1$ (upper) and $\textrm{DSB}_2$ (lower). The entry ``$p$ ($abcd$)'' indicates that this pentachoron is glued to pentachoron $p$ along a tetrahedron such that its (ordered) vertices are paired with vertices $a,b,c,d$ of $p$. $\textrm{DSB}_1$ has $f$-vector $(3,5,5,3,1)$ and boundary a 2-vertex 1-tetrahedron 3-sphere, $\textrm{DSB}_2$ has $f$-vector $(2,3,4,3,1)$ and boundary a 1-vertex 1-tetrahedron 3-sphere. Notation and labels may differ from the ones presented in \cite{Burke23}.}
\label{tab:double-loop-gluing}
\end{table}

The largest number of copies of $\textrm{DSB}_1$ that can be glued to a single pentachoron is of course five, resulting in a closed triangulation with $6$ pentachora where $\Gamma(\tri)$ has $10$ loops and branching number $8$. With an appropriate choice of gluings, this yields a triangulation of a 2-nonsingular pseudomanifold with $f$-vector $(6,13,18,15,6)$, which has exactly one singular edge (with link a torus) and one singular vertex. We call this the {\em dumpling} $\mathcal{D}^4_0$. Note that the five nonsingular vertices come from the internal vertices of the five copies of $\textrm{DSB}_1$. All other vertices of the triangulation are identified into the unique singular vertex.

Building larger triangulations, we are limited to four copies of $\textrm{DSB}_1$ glued to a common pentachoron. We can construct such a triangulation inductively from $\mathcal{D}^4_0$. Consider the ``unit'' obtained by deleting one copy of $\textrm{DSB}_1$ from $\mathcal{D}^4_0$, leaving four copies of $\textrm{DSB}_1$ around a pentachoron with one unglued ridge. This triangulation has five vertices, four internal vertices from the $\textrm{DSB}_1$ and one boundary vertex. Again, all five vertices of the central pentachoron are identified into one vertex. We can obtain a new {\em dumpling} $\mathcal{D}^4_1$ by replacing one copy of $\textrm{DSB}_1$ in $\mathcal{D}^4_0$ with this unit, increasing the number of pentachora by four, and the number of double snapped balls by three. Iterating this procedure we obtain dumplings $\mathcal{D}^4_k$ with $6+4k$ pentachora and $5+3k$ copies of $\textrm{DSB}_1$. 

The analogous construction in arbitrary even dimension $d$ using $\frac{d}{2}$-snapped balls, results in a dumpling $\mathcal{D}^d_k$ with $d+2+dk$ facets and $d+1+(d-1)k$ $\frac{d}{2}$-snapped balls. As described, $\mathcal{D}^d_k$ is not unique: even on the level of dual graphs we obtain one for every possible $(d+1)$-regular tree with $d+2+dk$ nodes. However, we can easily standardise the choice of $\frac{d}{2}$-snapped ball to replace, as well as the choice of gluing to obtain a canonical $\mathcal{D}^d_k$. See \Cref{fig:dumpling} for the $1$-skeleton of a $2$-dimensional dumpling, and \Cref{fig:dualdumpling} for examples of dual graphs of dumplings in dimension $4$.

\begin{figure}
\centering
\resizebox{0.4\linewidth}{!}{
\begin{tikzpicture}[every node/.style={circle, inner sep=0, outer sep=0, minimum width=0.15cm, fill=black}, every path/.style={thick}, label distance=2mm]

\node (0) {} ;

\draw (0) arc (90:450:3 and 2) {} ;

\node[fill=none] (1) [below left= 2.5 and 2.3 of 0] {} ;
\node[fill=none] (2) [below right=2.5 and 2.3 of 0] {} ;

\node[fill=none] (1x) [below left =2.5 and 0.5 of 0] {} ;
\node[fill=none] (2x) [below right=2.5 and 0.5 of 0] {} ;

\draw (0) to[out=180,in=140] (1.center) ;
\draw (1.center) to[out=-40,in=-130] (1x.center) ;
\draw (1x.center) to[out=50,in=-90] (0) ;

\draw (0) to[out=0,in=40] (2.center) ;
\draw (2.center) to[out=-140,in=-50] (2x.center) ;
\draw (2x.center) to[out=130,in=-90] (0) ;

\node[fill=none] (1a) [below left=1 and 1.6 of 0] {} ;
\node[fill=none] (1b) [below left=1.6 and 0.6 of 0] {} ;

\draw (0) to[out=180,in=160]  (1a.center) ;
\draw (0) to[out=-135,in=-20] (1a.center) ;

\draw (0) to[out=-90,in=-20]   (1b.center) ;
\draw (0) to[out=-135,in=160] (1b.center) ;

\node[fill=none] (2a) [below right=1 and 1.6 of 0] {} ;
\node[fill=none] (2b) [below right=1.6 and 0.6 of 0] {} ;

\draw (0) to[out=0,in=20]  (2a.center) ;
\draw (0) to[out=-45,in=-160] (2a.center) ;

\draw (0) to[out=-90,in=-160]    (2b.center) ;
\draw (0) to[out=-45,in=20] (2b.center) ;

\node (1av) [below left=0.5 and 1 of 0] {} ;
\node (1bv) [below left=0.9 and 0.3 of 0] {} ;

\draw (0) -- (1av) ;
\draw (0) -- (1bv) ;

\node (2av) [below right=0.5 and 1 of 0] {} ;
\node (2bv) [below right=0.9 and 0.3 of 0] {} ;

\draw (0) -- (2av) ;
\draw (0) -- (2bv) ;

\end{tikzpicture}
}
\caption{A dumpling (specifically $\mathcal{D}^2_3$)\label{fig:dumpling}}
\end{figure}

\begin{lemma}
  Let $\Delta$ be a facet of a $d$-dimensional triangulation $\tri$, $d$ even. If $\Delta$ is glued to $d$ distinct facets, each of which, as a result of the other gluings in $\tri$, has all $d$ vertices in the gluing already identified 
  \begin{enumerate}[label=(\alph*)]
    \item into a single vertex, or
    \item in pairs,
  \end{enumerate}
  then all $d+1$ vertices of $\Delta$ are identified in $\tri$.
  \label{lem:single-vertex-induction-step}
\end{lemma}
\begin{proof}
  Case (a) is strictly stronger than case (b), hence assume all of the neighbouring facets satisfy case (b).

  Let $V$ be the set of vertices of $\Delta$, and assume not all elements of $V$ are identified in $\tri$. Then we can partition $V$ into two nonempty disjoint sets $V = P_0 \sqcup P_1$ such that no vertex in $P_0$ is identified with a vertex of $P_1$. Since $|V|=d+1$ is odd, exactly one of these two sets must have an odd number of elements. Without loss of generality, say $|P_0|$ is even and $|P_1|$ is odd.

  For each of the $d$ gluings, there is precisely one element of $V$ not involved in the gluing, and the remaining $d$ vertices are identified together in pairs in $\tri$, since their images under the gluing are identified in pairs. Now $|P_0|\ge 2$, so $|P_1| \le d-1 < d$. Hence there must be at least one gluing for which the one unused vertex is in $P_0$. But then every element of $P_1$ must be paired with another vertex, and since $|P_1|$ is odd at least one must be paired with a vertex of $P_0$. So we have a vertex of $P_0$ and a vertex of $P_1$ which are identified in $\tri$, a contradiction.
\end{proof}

\begin{proposition}
  $\mathcal{D}^d_k$ has $d+2+(d-1)k$ vertices, of which (at most) one is singular.
\end{proposition}
\begin{proof}
  We can immediately see that $\mathcal{D}^d_k$ has $d+1+(d-1)k$ nonsingular, degree 1 vertices, one for each copy of the $\frac{d}{2}$-snapped ball. So we need only show that all other vertices are identified into a single vertex.

  Fix a ``root'' facet $\Delta_0$ which is not a $\frac{d}{2}$-snapped ball. Let $F_i$ be the set of facets which are at distance $i$ from $\Delta_0$ in the dual graph, and $m$ maximal such that $F_m$ is nonempty, so that the set of facets of $\mathcal{D}^d_k$ is $F_0 \sqcup F_1 \sqcup \dots \sqcup F_m$. Note that each element of $F_i$ is either glued to $d$ distinct elements of $F_{i+1}$ or is a $\frac{d}{2}$-snapped ball, and for $i>0$ each element of $F_i$ is glued to exactly one element of $F_{i-1}$ (which is not a $\frac{d}{2}$-snapped ball). If an element of $F_i$ is a $\frac{d}{2}$-snapped ball, then the $d$ vertices involved in its gluing to an element of $F_{i-1}$ (the boundary vertices of the unglued $\frac{d}{2}$-snapped ball) are identified together in pairs. By induction on $i$ (from the ``outside in'', decreasing from $F_m$ to $F_0$), we see that for each $i=0,1,\dots,m$, every facet of $F_i$ is either a $\frac{d}{2}$-snapped ball or has all of its vertices identified together, with the induction step given by \Cref{lem:single-vertex-induction-step} and the base clear since every element of $F_m$ is a $\frac{d}{2}$-snapped ball. Another induction (from the ``inside out'', increasing from $F_0$ to $F_m$) shows that all vertices except the degree 1 vertices of the $\frac{d}{2}$-snapped balls must be identified with the single vertex class contained in $\Delta_0$, since every facet of $F_i$ has its non-degree $1$ vertices identified with one from a facet of $F_{i-1}$.
\end{proof}

\begin{remark}
  The proof above works regardless of the choice of ``canonical'' $\mathcal{D}^d_k$, so in particular there is no ``better'' way to perform this construction and obtain a larger number of vertices than our canonical choice. The same will be true if we try to generalise $\textrm{DSB}_2$ instead for the construction, since all boundary vertices of $\textrm{DSB}_2$ are identified into one and \Cref{lem:single-vertex-induction-step} still applies. What these choices can impact, however, is how ``manifold-like'' $\mathcal{D}^d_k$ is. With the canonical choice presented here, $\mathcal{D}^4_k$ appears to always be a 2-nonsingular pseudomanifold, and to have only one singular edge link, which is an orientable surface of genus $k+1$. Experimentally, no choice has been found which makes $\mathcal{D}^4_k$ 1-nonsingular or produces a single singular edge of lower genus.
  \label{rmk:no-better-than-canonical}
\end{remark}

\begin{proof}[Proof of \Cref{thm:counterexamples}]
From our construction we can deduce that $f_0(\mathcal{D}^d_k) = d+2+(d-1)k = \frac{d-1}{d}(d+2+dk) + \frac{1}{d}(d+2) = \frac{d-1}{d}f_d(\mathcal{D}^d_k)+\frac{d+2}{d}$, and in particular $f_4(\mathcal{D}^4_2) = 14$ and $f_0(\mathcal{D}^4_2)= 12 = \frac{14}{2}+5$. It remains to show that we can construct $\mathcal{D}^4_2$ such that it is a pseudomanifold (and hence is 2-nonsingular by \Cref{cor:always-d-2}). This can be checked directly on our explicit and canonical triangulation $\mathcal{D}^4_2$.
\end{proof}

\begin{figure}
\centering
\resizebox{0.8\linewidth}{!}{
\begin{tikzpicture}[every node/.style={circle, inner sep=0, outer sep=0, minimum width=0.15cm, fill=black}, every path/.style={thick}, label distance=2mm]

\node(0) {} ;

\path (0) ++(90:1)  node(3) {} ;
\path (0) ++(18:1)  node(2) {} ;
\path (0) ++(306:1) node(5) {} ;
\path (0) ++(234:1) node(4) {} ;
\path (0) ++(162:1) node(1) {} ;

\draw (0) -- (1) ;
\draw (0) -- (2) ;
\draw (0) -- (3) ;
\draw (0) -- (4) ;
\draw (0) -- (5) ;

\path (1) ++(222:1) node(6) {} ;
\path (1) ++(182:1) node(7) {} ;
\path (1) ++(142:1) node(8) {} ;
\path (1) ++(102:1) node(9) {} ;

\draw (1) -- (6) ;
\draw (1) -- (7) ;
\draw (1) -- (8) ;
\draw (1) -- (9) ;

\path (2) ++(78:1)  node(10) {} ;
\path (2) ++(38:1)  node(11) {} ;
\path (2) ++(358:1)   node(12) {} ;
\path (2) ++(318:1) node(13) {} ;

\draw (2) -- (10) ;
\draw (2) -- (11) ;
\draw (2) -- (12) ;
\draw (2) -- (13) ;

\draw            (3) to[out=60,in=120,loop] (3) ;
\draw[scale=1.8] (3) to[out=50,in=130,loop] (3) ;

\draw            (4) to[out=204,in=264,loop] (4) ;
\draw[scale=1.8] (4) to[out=194,in=274,loop] (4) ;

\draw            (5) to[out=276,in=336,loop] (5) ;
\draw[scale=1.8] (5) to[out=266,in=346,loop] (5) ;

\draw            (6) to[out=177,in=237,loop] (6) ;
\draw[scale=1.8] (6) to[out=167,in=247,loop] (6) ;
\draw            (7) to[out=147,in=207,loop] (7) ;
\draw[scale=1.8] (7) to[out=137,in=217,loop] (7) ;
\draw            (8) to[out=117,in=177,loop] (8) ;
\draw[scale=1.8] (8) to[out=107,in=187,loop] (8) ;
\draw            (9) to[out=87,in=147,loop]  (9) ;
\draw[scale=1.8] (9) to[out=77,in=157,loop]  (9) ;

\draw            (10) to[out=33,in=93,loop]  (10) ;
\draw[scale=1.8] (10) to[out=23,in=103,loop] (10) ;
\draw            (11) to[out=3,in=63,loop]   (11) ;
\draw[scale=1.8] (11) to[out=353,in=73,loop] (11) ;
\draw            (12) to[out=333,in=33,loop] (12) ;
\draw[scale=1.8] (12) to[out=323,in=43,loop] (12) ;
\draw            (13) to[out=303,in=3,loop]  (13) ;
\draw[scale=1.8] (13) to[out=293,in=13,loop] (13) ;

\end{tikzpicture}
\begin{tikzpicture}[every node/.style={circle, inner sep=0, outer sep=0, minimum width=0.15cm, fill=black}, every path/.style={thick}, label distance=2mm]

\node(0) {} ;

\path (0) ++(90:1)  node(3) {} ;
\path (0) ++(18:1)  node(2) {} ;
\path (0) ++(306:1) node(5) {} ;
\path (0) ++(234:1) node(4) {} ;
\path (0) ++(162:1) node(1) {} ;

\draw (0) -- (1) ;
\draw (0) -- (2) ;
\draw (0) -- (3) ;
\draw (0) -- (4) ;
\draw (0) -- (5) ;

\path (1) ++(222:1) node(6) {} ;
\path (1) ++(182:1) node(7) {} ;
\path (1) ++(142:1) node(8) {} ;
\path (1) ++(102:1) node(9) {} ;

\draw (1) -- (6) ;
\draw (1) -- (7) ;
\draw (1) -- (8) ;
\draw (1) -- (9) ;

\path (2) ++(78:1)  node(10) {} ;
\path (2) ++(38:1)  node(11) {} ;
\path (2) ++(358:1)   node(12) {} ;
\path (2) ++(318:1) node(13) {} ;

\draw (2) -- (10) ;
\draw (2) -- (11) ;
\draw (2) -- (12) ;
\draw (2) -- (13) ;

\path (3) ++(50:0.5)  node (3a) {} ; 
\path (3) ++(130:0.5) node (3b) {} ;
\path (3) ++(70:0.8)  node (3c) {} ; 
\path (3) ++(110:0.8) node (3d) {} ; 

\draw (3) -- (3a) ;
\draw (3) -- (3b) ;
\draw (3) -- (3c) ;
\draw (3) -- (3d) ;

\draw (3a) to[bend left=15]  (3b) ;
\draw (3a) to[bend left=40]  (3b) ;
\draw (3a) to[bend right=15] (3b) ;
\draw (3a) to[bend right=40] (3b) ;

\draw (3c) to[bend left=15]  (3d) ;
\draw (3c) to[bend left=40]  (3d) ;
\draw (3c) to[bend right=15] (3d) ;
\draw (3c) to[bend right=40] (3d) ;

\path (4) ++(194:0.5)  node (4a) {} ; 
\path (4) ++(274:0.5) node (4b) {} ;
\path (4) ++(214:0.8)  node (4c) {} ; 
\path (4) ++(254:0.8) node (4d) {} ; 

\draw (4) -- (4a) ;
\draw (4) -- (4b) ;
\draw (4) -- (4c) ;
\draw (4) -- (4d) ;

\draw (4a) to[bend left=15]  (4b) ;
\draw (4a) to[bend left=40]  (4b) ;
\draw (4a) to[bend right=15] (4b) ;
\draw (4a) to[bend right=40] (4b) ;

\draw (4c) to[bend left=15]  (4d) ;
\draw (4c) to[bend left=40]  (4d) ;
\draw (4c) to[bend right=15] (4d) ;
\draw (4c) to[bend right=40] (4d) ;

\path (5) ++(266:0.5)  node (5a) {} ; 
\path (5) ++(346:0.5) node (5b) {} ;
\path (5) ++(286:0.8)  node (5c) {} ; 
\path (5) ++(326:0.8) node (5d) {} ; 

\draw (5) -- (5a) ;
\draw (5) -- (5b) ;
\draw (5) -- (5c) ;
\draw (5) -- (5d) ;

\draw (5a) to[bend left=15]  (5b) ;
\draw (5a) to[bend left=40]  (5b) ;
\draw (5a) to[bend right=15] (5b) ;
\draw (5a) to[bend right=40] (5b) ;

\draw (5c) to[bend left=15]  (5d) ;
\draw (5c) to[bend left=40]  (5d) ;
\draw (5c) to[bend right=15] (5d) ;
\draw (5c) to[bend right=40] (5d) ;

\path (6) ++(182:0.5)  node (6a) {} ; 
\path (6) ++(262:0.5) node (6b) {} ;
\path (6) ++(202:0.8)  node (6c) {} ; 
\path (6) ++(242:0.8) node (6d) {} ; 

\draw (6) -- (6a) ;
\draw (6) -- (6b) ;
\draw (6) -- (6c) ;
\draw (6) -- (6d) ;

\draw (6a) to[bend left=15]  (6b) ;
\draw (6a) to[bend left=40]  (6b) ;
\draw (6a) to[bend right=15] (6b) ;
\draw (6a) to[bend right=40] (6b) ;

\draw (6c) to[bend left=15]  (6d) ;
\draw (6c) to[bend left=40]  (6d) ;
\draw (6c) to[bend right=15] (6d) ;
\draw (6c) to[bend right=40] (6d) ;

\path (7) ++(142:0.5)  node (7a) {} ; 
\path (7) ++(222:0.5) node (7b) {} ;
\path (7) ++(162:0.8)  node (7c) {} ; 
\path (7) ++(202:0.8) node (7d) {} ; 

\draw (7) -- (7a) ;
\draw (7) -- (7b) ;
\draw (7) -- (7c) ;
\draw (7) -- (7d) ;

\draw (7a) to[bend left=15]  (7b) ;
\draw (7a) to[bend left=40]  (7b) ;
\draw (7a) to[bend right=15] (7b) ;
\draw (7a) to[bend right=40] (7b) ;

\draw (7c) to[bend left=15]  (7d) ;
\draw (7c) to[bend left=40]  (7d) ;
\draw (7c) to[bend right=15] (7d) ;
\draw (7c) to[bend right=40] (7d) ;

\path (8) ++(102:0.5)  node (8a) {} ; 
\path (8) ++(182:0.5) node (8b) {} ;
\path (8) ++(122:0.8)  node (8c) {} ; 
\path (8) ++(162:0.8) node (8d) {} ; 

\draw (8) -- (8a) ;
\draw (8) -- (8b) ;
\draw (8) -- (8c) ;
\draw (8) -- (8d) ;

\draw (8a) to[bend left=15]  (8b) ;
\draw (8a) to[bend left=40]  (8b) ;
\draw (8a) to[bend right=15] (8b) ;
\draw (8a) to[bend right=40] (8b) ;

\draw (8c) to[bend left=15]  (8d) ;
\draw (8c) to[bend left=40]  (8d) ;
\draw (8c) to[bend right=15] (8d) ;
\draw (8c) to[bend right=40] (8d) ;

\path (9) ++(62:0.5)  node (9a) {} ; 
\path (9) ++(142:0.5) node (9b) {} ;
\path (9) ++(82:0.8)  node (9c) {} ; 
\path (9) ++(122:0.8) node (9d) {} ; 

\draw (9) -- (9a) ;
\draw (9) -- (9b) ;
\draw (9) -- (9c) ;
\draw (9) -- (9d) ;

\draw (9a) to[bend left=15]  (9b) ;
\draw (9a) to[bend left=40]  (9b) ;
\draw (9a) to[bend right=15] (9b) ;
\draw (9a) to[bend right=40] (9b) ;

\draw (9c) to[bend left=15]  (9d) ;
\draw (9c) to[bend left=40]  (9d) ;
\draw (9c) to[bend right=15] (9d) ;
\draw (9c) to[bend right=40] (9d) ;

\path (10) ++(38:0.5)  node (10a) {} ; 
\path (10) ++(118:0.5) node (10b) {} ;
\path (10) ++(58:0.8)  node (10c) {} ; 
\path (10) ++(98:0.8) node (10d) {} ; 

\draw (10) -- (10a) ;
\draw (10) -- (10b) ;
\draw (10) -- (10c) ;
\draw (10) -- (10d) ;

\draw (10a) to[bend left=15]  (10b) ;
\draw (10a) to[bend left=40]  (10b) ;
\draw (10a) to[bend right=15] (10b) ;
\draw (10a) to[bend right=40] (10b) ;

\draw (10c) to[bend left=15]  (10d) ;
\draw (10c) to[bend left=40]  (10d) ;
\draw (10c) to[bend right=15] (10d) ;
\draw (10c) to[bend right=40] (10d) ;

\path (11) ++(358:0.5)  node (11a) {} ; 
\path (11) ++(78:0.5) node (11b) {} ;
\path (11) ++(18:0.8)  node (11c) {} ; 
\path (11) ++(58:0.8) node (11d) {} ; 

\draw (11) -- (11a) ;
\draw (11) -- (11b) ;
\draw (11) -- (11c) ;
\draw (11) -- (11d) ;

\draw (11a) to[bend left=15]  (11b) ;
\draw (11a) to[bend left=40]  (11b) ;
\draw (11a) to[bend right=15] (11b) ;
\draw (11a) to[bend right=40] (11b) ;

\draw (11c) to[bend left=15]  (11d) ;
\draw (11c) to[bend left=40]  (11d) ;
\draw (11c) to[bend right=15] (11d) ;
\draw (11c) to[bend right=40] (11d) ;

\path (12) ++(318:0.5)  node (12a) {} ; 
\path (12) ++(38:0.5) node (12b) {} ;
\path (12) ++(338:0.8)  node (12c) {} ; 
\path (12) ++(18:0.8) node (12d) {} ; 

\draw (12) -- (12a) ;
\draw (12) -- (12b) ;
\draw (12) -- (12c) ;
\draw (12) -- (12d) ;

\draw (12a) to[bend left=15]  (12b) ;
\draw (12a) to[bend left=40]  (12b) ;
\draw (12a) to[bend right=15] (12b) ;
\draw (12a) to[bend right=40] (12b) ;

\draw (12c) to[bend left=15]  (12d) ;
\draw (12c) to[bend left=40]  (12d) ;
\draw (12c) to[bend right=15] (12d) ;
\draw (12c) to[bend right=40] (12d) ;

\path (13) ++(278:0.5)  node (13a) {} ; 
\path (13) ++(358:0.5) node (13b) {} ;
\path (13) ++(298:0.8)  node (13c) {} ; 
\path (13) ++(338:0.8) node (13d) {} ; 

\draw (13) -- (13a) ;
\draw (13) -- (13b) ;
\draw (13) -- (13c) ;
\draw (13) -- (13d) ;

\draw (13a) to[bend left=15]  (13b) ;
\draw (13a) to[bend left=40]  (13b) ;
\draw (13a) to[bend right=15] (13b) ;
\draw (13a) to[bend right=40] (13b) ;

\draw (13c) to[bend left=15]  (13d) ;
\draw (13c) to[bend left=40]  (13d) ;
\draw (13c) to[bend right=15] (13d) ;
\draw (13c) to[bend right=40] (13d) ;

\end{tikzpicture}
}
\caption{Dual graphs of $\mathcal{D}^4_2$ (left) -- the smallest known triangulation in dimension 4 with $f_0 > \frac{f_4}{2}+4$ -- and $\hat{\mathcal{D}}^4_2$ (right) -- the smallest known such triangulation which also has no loops in its dual graph.\label{fig:dualdumpling}}
\end{figure}

We can also construct triangulations with $f_0 > \frac{f_d}{2}+d$ without any self-identifications of facets: we simply use 0-2 vertex moves to eliminate loops in the dual graph without affecting $\delta_\tri$. More precisely, we have the following.

\begin{corollary}
  \label{cor:counterexamples-no-loops}
  For any $d$ even and $f_d \equiv 2d+2 \textrm{ mod } d^2$ with $f_d \ge d^2+2d+2$, there exists a $d$-dimensional triangulation with $f_d$ facets and no loops in its dual graph such that
  \begin{align*}
    f_0 = \frac{(d+2)(d-1)}{2d^2}f_d + \frac{d^2+d+2}{d^2}
  \end{align*}
  In particular, there exists a 2-nonsingular 4-pseudomanifold with 58 pentachora and no loops in its dual graph such that
  \begin{align*}
    f_0 > \frac{f_4}{2}+4
  \end{align*}
\end{corollary}
\begin{proof}
  Begin with the triangulation $\mathcal{D}^d_k$, which has $f_0'$ vertices and $f_d'$ facets, $f_0'=\frac{d-1}{d}f_d'+\frac{d+2}{d}$, and $\frac{d}{2}(f_0'-1)$ loops in its dual graph. We construct a triangulation $\hat{\mathcal{D}}^d_k$ with no loops in its dual graph by performing $\frac{d}{2}(f_0'-1)$ 0-2 vertex moves. Each such move adds one vertex and two facets to $\hat{\mathcal{D}}^d_k$, and we have for the number of its vertices $f_0$ and facets $f_d$
  \begin{align*}
    f_0 &= f_0' + \frac{d}{2}\left(f_0'-1\right) = \frac{d+2}{2}f_0'-\frac{d}{2} \\
    f_d &= f_d' + d\left(f_0'-1\right) = f_d' + d\left(\frac{d-1}{d}f_d'+\frac{2}{d}\right) = df_d' + 2
  \end{align*}
  Then we have
  \begin{align*}
    f_0 &= \frac{d+2}{2}\left(\frac{d-1}{d}f_d'+\frac{d+2}{d}\right)-\frac{d}{2} \\
        &= \frac{d+2}{2}\left(\frac{d-1}{d}\left(\frac{f_d-2}{d}\right)+\frac{d+2}{d}\right)-\frac{d}{2} \\
        &= \frac{(d+2)(d-1)}{2d^2}f_d + \frac{-2(d+2)(d-1)+d(d+2)^2-d^3}{2d^2} \\
        &= \frac{(d+2)(d-1)}{2d^2}f_d + \frac{d^2+d+2}{d^2},
  \end{align*}
  as required. Finally, note that since $\mathcal{D}^d_k$ has $d+2+dk$ facets, we have $f_d=d(d+2+dk)+2=2d+2+(k+1)d^2$, and so indeed a corresponding $\mathcal{D}^d_k$ exists for all $f_d \equiv 2d+2 \textrm{ mod } d^2$ with $f_d \ge d^2+2d+2$.
\end{proof}

Focusing back on dimension 4, \Cref{thm:counterexamples} and \Cref{cor:counterexamples-no-loops} inspire the following questions.

\begin{question}
  Can we find an infinite family of 4-dimensional (pseudomanifold) triangulations $\mathcal{P}_k$ such that $\lim_{k \to \infty} \frac{f_0}{f_4} > \frac{3}{4}$?
  \label{q:exceed-3/4}
\end{question}

\begin{question}
  Can we find an infinite family of 4-dimensional (pseudomanifolds) triangulations $\hat{\mathcal{P}}_k$ without loops in $\Gamma(\hat{\mathcal{P}}_k)$ such that $\lim_{k \to \infty} \frac{f_0}{f_4} > \frac{9}{16}$?
  \label{q:exceed-9/16}
\end{question}

\begin{question}
  Is there a 4-dimensional triangulation with $f_0 > \frac{f_4}{2}+4$ that has fewer than $14$ pentachora?
  \label{q:smaller-than-14}
\end{question}

By \Cref{prop:linear-bound}, the largest possible value for $\lim_{k \to \infty} \frac{f_0}{f_4}$ is $1$. In any odd dimension, \Cref{prop:construction-odd-large} states that this maximum is indeed attainable. Moreover, \Cref{thm:counterexamples} shows that, for sufficiently large even dimension, we can get arbitrarily close to it. So \Cref{q:exceed-3/4} essentially asks how close we can push $\lim_{k \to \infty} \frac{f_0}{f_4}$ towards $1$. Note that \Cref{q:exceed-9/16} immediately implies \Cref{q:exceed-3/4}, by applying the same argument from the proof of \Cref{cor:counterexamples-no-loops} to an arbitrary triangulation (noting that $\mathcal{D}^4_k$ clearly has the maximum possible number of loops amongst all 4-dimensional triangulations $\tri$ with $f_d(\tri) \le f_d(\mathcal{D}^4_k)$). 

Focusing on the topology of our counterexamples, we can ask for a weaker version of \Cref{conj:vertex-bound-even}.

\begin{question}
  Can we find a $4$-dimensional triangulation with $f_0 > \frac{f_4}{2}+4$ which is $1$-nonsingular?
  \label{q:1-nonsingular-counterexample}
\end{question}

\begin{figure}
\centering
\resizebox{0.24\linewidth}{!}{
\begin{tikzpicture}[every node/.style={circle, inner sep=0, outer sep=0, minimum width=0.15cm, fill=black}, every path/.style={thick}, label distance=2mm]

\node(0) {} ;

\path (0) ++(90:1)  node(1) {} ;
\path (0) ++(18:1)  node(6) {} ;
\path (0) ++(306:1) node(11) {} ;
\path (0) ++(234:1) node(16) {} ;
\path (0) ++(162:1) node(21) {} ;

\draw (0) -- (1) ;
\draw (0) -- (6) ;
\draw (0) -- (11) ;
\draw (0) -- (16) ;
\draw (0) -- (21) ;

\path (1) ++(135:1) node(2) {} ;
\path (1) ++(105:1) node(3) {} ;
\path (1) ++(75:1)  node(4) {} ;
\path (1) ++(45:1)  node(5) {} ;

\draw (1) -- (2) ;
\draw (1) -- (3) ;
\draw (1) -- (4) ;
\draw (1) -- (5) ;

\path (6) ++(63:1)  node(7) {} ;
\path (6) ++(33:1)  node(8) {} ;
\path (6) ++(3:1)   node(9) {} ;
\path (6) ++(333:1) node(10) {} ;

\draw (6) -- (7) ;
\draw (6) -- (8) ;
\draw (6) -- (9) ;
\draw (6) -- (10) ;

\path (11) ++(351:1) node(12) {} ;
\path (11) ++(321:1) node(13) {} ;
\path (11) ++(291:1) node(14) {} ;
\path (11) ++(261:1) node(15) {} ;

\draw (11) -- (12) ;
\draw (11) -- (13) ;
\draw (11) -- (14) ;
\draw (11) -- (15) ;

\path (16) ++(279:1) node(17) {} ;
\path (16) ++(249:1) node(18) {} ;
\path (16) ++(219:1) node(19) {} ;
\path (16) ++(189:1) node(20) {} ;

\draw (16) -- (17) ;
\draw (16) -- (18) ;
\draw (16) -- (19) ;
\draw (16) -- (20) ;

\path (21) ++(207:1) node(22) {} ;
\path (21) ++(177:1) node(23) {} ;
\path (21) ++(147:1) node(24) {} ;
\path (21) ++(117:1) node(25) {} ;

\draw (21) -- (22) ;
\draw (21) -- (23) ;
\draw (21) -- (24) ;
\draw (21) -- (25) ;

\draw (2) to[out=105,in=165,loop] (2) ;
\draw[scale=1.8] (2) to[out=95,in=175,loop] (2) ;
\draw (3) to[out=75,in=135,loop] (3) ;
\draw[scale=1.8] (3) to[out=65,in=145,loop] (3) ;
\draw (4) to[out=45,in=105,loop] (4) ;
\draw[scale=1.8] (4) to[out=35,in=115,loop] (4) ;
\draw (5) to[out=15,in=75,loop] (5) ;
\draw[scale=1.8] (5) to[out=5,in=85,loop] (5) ;

\draw (7) to[out=33,in=93,loop] (7) ;
\draw[scale=1.8] (7) to[out=23,in=103,loop] (7) ;
\draw (8) to[out=3,in=63,loop] (8) ;
\draw[scale=1.8] (8) to[out=353,in=73,loop] (8) ;
\draw (9) to[out=333,in=33,loop] (9) ;
\draw[scale=1.8] (9) to[out=323,in=43,loop] (9) ;
\draw (10) to[out=303,in=3,loop] (10) ;
\draw[scale=1.8] (10) to[out=293,in=13,loop] (10) ;

\draw (12) to[out=321,in=21,loop] (12) ;
\draw[scale=1.8] (12) to[out=311,in=31,loop] (12) ;
\draw (13) to[out=291,in=351,loop] (13) ;
\draw[scale=1.8] (13) to[out=281,in=1,loop] (13) ;
\draw (14) to[out=261,in=321,loop] (14) ;
\draw[scale=1.8] (14) to[out=251,in=331,loop] (14) ;
\draw (15) to[out=231,in=291,loop] (15) ;
\draw[scale=1.8] (15) to[out=221,in=301,loop] (15) ;

\draw (17) to[out=249,in=309,loop] (17) ;
\draw[scale=1.8] (17) to[out=239,in=319,loop] (17) ;
\draw (18) to[out=219,in=279,loop] (18) ;
\draw[scale=1.8] (18) to[out=209,in=289,loop] (18) ;
\draw (19) to[out=189,in=249,loop] (19) ;
\draw[scale=1.8] (19) to[out=179,in=259,loop] (19) ;
\draw (20) to[out=159,in=219,loop] (20) ;
\draw[scale=1.8] (20) to[out=149,in=229,loop] (20) ;

\draw (22) to[out=177,in=237,loop] (22) ;
\draw[scale=1.8] (22) to[out=167,in=247,loop] (22) ;
\draw (23) to[out=147,in=207,loop] (23) ;
\draw[scale=1.8] (23) to[out=137,in=217,loop] (23) ;
\draw (24) to[out=117,in=177,loop] (24) ;
\draw[scale=1.8] (24) to[out=107,in=187,loop] (24) ;
\draw (25) to[out=87,in=147,loop] (25) ;
\draw[scale=1.8] (25) to[out=77,in=157,loop] (25) ;

\end{tikzpicture}
}
\resizebox{0.24\linewidth}{!}{
\begin{tikzpicture}[every node/.style={circle, inner sep=0, outer sep=0, minimum width=0.15cm, fill=black}, every path/.style={thick}, label distance=2mm]

\node(3) at (0,0) {} ;
\node(7) [right=of 3] {} ;
\node(11) [below=of 7] {} ;
\node(15) [left=of 11] {} ;

\draw (3) -- (7) ;
\draw (7) -- (11) ;
\draw (11) -- (15) ;
\draw (15) -- (3) ;

\node(0) [left=of 3] {} ;
\node(1) [above left=1 of 3] {} ;
\node(2) [above=of 3] {} ;

\draw (3) -- (0) ;
\draw (3) -- (1) ;
\draw (3) -- (2) ;

\node(4) [above=of 7] {} ;
\node(5) [above right=1 of 7] {} ;
\node(6) [right=of 7] {} ;

\draw (7) -- (4) ;
\draw (7) -- (5) ;
\draw (7) -- (6) ;

\node(8) [right=of 11] {} ;
\node(9) [below right=1 of 11] {} ;
\node(10) [below=of 11] {} ;

\draw (11) -- (8) ;
\draw (11) -- (9) ;
\draw (11) -- (10) ;

\node(12) [below=of 15] {} ;
\node(13) [below left=1 of 15] {} ;
\node(14) [left=of 15] {} ;

\draw (15) -- (12) ;
\draw (15) -- (13) ;
\draw (15) -- (14) ;

\draw (0) to[out=150,in=210,loop] (0) ;
\draw[scale=1.8] (0) to[out=140,in=220,loop] (0) ;

\draw (1) to[out=105,in=165,loop] (1) ;
\draw[scale=1.8] (1) to[out=95,in=175,loop] (1) ;

\draw (2) to[out=60,in=120,loop] (0) ;
\draw[scale=1.8] (2) to[out=50,in=130,loop] (0) ;

\draw (4) to[out=60,in=120,loop] (4) ;
\draw[scale=1.8] (4) to[out=50,in=130,loop] (4) ;

\draw (5) to[out=15,in=75,loop] (5) ;
\draw[scale=1.8] (5) to[out=5,in=85,loop] (5) ;

\draw (6) to[out=30,in=330,loop] (6) ;
\draw[scale=1.8] (6) to[out=40,in=320,loop] (6) ;

\draw (8) to[out=30,in=330,loop] (8) ;
\draw[scale=1.8] (8) to[out=40,in=320,loop] (8) ;

\draw (9) to[out=285,in=345,loop] (9) ;
\draw[scale=1.8] (9) to[out=275,in=355,loop] (9) ;

\draw (10) to[out=240,in=300,loop] (10) ;
\draw[scale=1.8] (10) to[out=230,in=310,loop] (10) ;

\draw (12) to[out=240,in=300,loop] (12) ;
\draw[scale=1.8] (12) to[out=230,in=310,loop] (12) ;

\draw (13) to[out=195,in=255,loop] (13) ;
\draw[scale=1.8] (13) to[out=185,in=265,loop] (13) ;

\draw (14) to[out=150,in=210,loop] (14) ;
\draw[scale=1.8] (14) to[out=140,in=220,loop] (14) ;

\end{tikzpicture}
}
\resizebox{0.24\linewidth}{!}{
\begin{tikzpicture}[every node/.style={circle, inner sep=0, outer sep=0, minimum width=0.15cm, fill=black}, every path/.style={thick}, label distance=2mm]

\node(3) at (0,0) {} ;
\node(7) [right=of 3] {} ;
\node(11) [below=of 7] {} ;
\node(15) [left=of 11] {} ;

\draw (3) to[bend right=20] (7) ;
\draw (7) -- (11) ;
\draw (11) to[bend right=20] (15) ;
\draw (15) -- (3) ;

\draw (3) to[bend left=20] (7) ;
\draw (11) to[bend left=20] (15) ;

\node(0) [left=of 3] {} ;
\node(2) [above=of 3] {} ;

\draw (3) -- (0) ;
\draw (3) -- (2) ;

\node(4) [above=of 7] {} ;
\node(6) [right=of 7] {} ;

\draw (7) -- (4) ;
\draw (7) -- (6) ;

\node(8) [right=of 11] {} ;
\node(10) [below=of 11] {} ;

\draw (11) -- (8) ;
\draw (11) -- (10) ;

\node(12) [below=of 15] {} ;
\node(14) [left=of 15] {} ;

\draw (15) -- (12) ;
\draw (15) -- (14) ;

\draw (0) to[out=150,in=210,loop] (0) ;
\draw[scale=1.8] (0) to[out=140,in=220,loop] (0) ;

\draw (2) to[out=60,in=120,loop] (0) ;
\draw[scale=1.8] (2) to[out=50,in=130,loop] (0) ;

\draw (4) to[out=60,in=120,loop] (4) ;
\draw[scale=1.8] (4) to[out=50,in=130,loop] (4) ;

\draw (6) to[out=30,in=330,loop] (6) ;
\draw[scale=1.8] (6) to[out=40,in=320,loop] (6) ;

\draw (8) to[out=30,in=330,loop] (8) ;
\draw[scale=1.8] (8) to[out=40,in=320,loop] (8) ;

\draw (10) to[out=240,in=300,loop] (10) ;
\draw[scale=1.8] (10) to[out=230,in=310,loop] (10) ;

\draw (12) to[out=240,in=300,loop] (12) ;
\draw[scale=1.8] (12) to[out=230,in=310,loop] (12) ;

\draw (14) to[out=150,in=210,loop] (14) ;
\draw[scale=1.8] (14) to[out=140,in=220,loop] (14) ;

\end{tikzpicture}
}
\caption{Left-to-right: dual graphs of $\mathcal{D}^4_6$, $\mathcal{Y}_4$, and $\mathcal{V}_4$ \label{fig:other}}
\end{figure}

In light of \Cref{rmk:dehn-sommerville-vs-nonsingular}, answering this question is in fact equivalent to answering \Cref{q:dehn-sommerville-vs-conjecture}. We conclude this section with experimental attempts to answer \Cref{q:1-nonsingular-counterexample} in the positive, in the form of a further two constructions in dimension 4. While the question remains open, the failure of these natural constructions does provide initial evidence towards a negative answer (and a positive answer to \Cref{q:dehn-sommerville-vs-conjecture}).

Rather than using as many double snapped balls as possible, we can try to use a still large but smaller number, such as to still exceed the bound from \Cref{conj:vertex-bound-even} but leave more freedom in how they are glued, to potentially avoid a singular edge link. So, where previously we used a unit consisting of four double-snapped balls around a central pentachoron, we can consider units with three or two copies of $\textrm{DSB}_1$ around a pentachoron.

The unit with three $\textrm{DSB}_1$'s leaves two unglued facets, letting us glue $k$ of them together in a cycle in a way reminiscent of $\mathcal{S}^d_{f_d}$ from \Cref{prop:construction-odd-large} in the odd-dimensional case, see \Cref{fig:other} in the centre. The resulting ``three-spike'' triangulation $\mathcal{Y}_k$ has $4k$ pentachora and at least $3k$ vertices, from the four pentachora and three internal vertices of each unit. The canonical $\mathcal{Y}_k$ presented here has $3k+1$ vertices, in a similar situation to what was seen in $\mathcal{D}^4_k$, and hence $f_0 = \frac{3}{4}f_4 + 1$ and the smallest entry with $\delta_{\mathcal{Y}_k} > 4$ is $\mathcal{Y}_4$ with 16 pentachora. $\mathcal{Y}_k$ appears to be a 2-nonsingular pseudomanifold for all $k$ with one singular edge, whose link is a genus $k-1$ orientable surface. No choice of gluings has been found which improves on this, or on the number of vertices.

The unit with two $\textrm{DSB}_1$'s has three unglued faces, allowing $2k$ of them to be arranged according to any 3-regular graph on $2k$ nodes. The ``two-spike'' triangulation $\mathcal{V}_{2k}$ has $6k$ pentachora and at least $4k$ vertices, see \Cref{fig:other} on the right. With considerable freedom in how to perform the gluings, this construction seems the most promising so far in terms of producing 1-nonsingular examples, but once again no such construction has been found which improves on the situation from the previous constructions. Our best found $\mathcal{V}_{2k}$ arranges the units into a cycle with one gluing to the previous gluing and two to the next. It has $4k+1$ vertices, satisfies $f_0 = \frac{2}{3}f_d+1$, and the smallest member with $\delta_{\mathcal{V}_{2k}} > 4$ is $\mathcal{V}_4$ with 24 pentachora. Again, $\mathcal{V}_{2k}$ appears to be a 2-nonsingular pseudomanifold for all $k$, with one singular edge whose link is a genus $2k-2$ orientable surface. 

See \Cref{tab:series-properties} for a summary of properties of the families of triangulations presented in this section.

\begin{table}
\centering
\begin{tabular}{|l|c|c|c|l|}
\hline
Family                  & $f_4$    & $f_0$   & $\lim_{k \to \infty} f_0/f_4$ & Singular edge links    \\ \hline
$\mathcal{D}^4_k$       & $4k+6$   & $3k+6$  & $3/4$                         & one genus $k+1$ torus  \\
$\hat{\mathcal{D}}^4_k$ & $16k+26$ & $9k+16$ & $9/16$                        & one genus $k+1$ torus  \\
$\mathcal{Y}_k$         & $4k$     & $3k+1$  & $3/4$                         & one genus $k-1$ torus  \\
$\mathcal{V}_{2k}$      & $6k$     & $4k+1$  & $2/3$                         & one genus $2k-2$ torus \\ \hline
\end{tabular}
\caption{Properties of the four featured series in dimension 4.}
\label{tab:series-properties}
\end{table}

\bibliographystyle{plain}
\bibliography{references_lucy}

\newpage

\begin{appendices}
\crefalias{section}{appendix}

\section{Code}
\label{app:code}

\subsection*{Code to Generate $\mathcal{D}^d_k$ and $\hat{\mathcal{D}}^d_k$}

\begin{lstlisting}[language=python]
from regina import *

def TriangulationD(d):
    match d:
        case 2:
            return Triangulation2
        case 4:
            return Triangulation4
        case 6:
            return Triangulation6
        case 8:
            return Triangulation8
        case 10:
            return Triangulation10
        case 12:
            return Triangulation12
        case 14:
            return Triangulation14

def PermN(n):
    match n:
        case 3:
            return Perm3
        case 5:
            return Perm5
        case 7:
            return Perm7
        case 9:
            return Perm9
        case 11:
            return Perm11
        case 13:
            return Perm13
        case 15:
            return Perm15

def dumpling(d,k):

    t=TriangulationD(d)()
    t.newSimplices(d+2)

    gluings = [PermN(d+1)(i,d) for i in range(d+1)]

    for i in range(d+1):
        t.simplex(0).join(i,t.simplex(d+1-i),gluings[i])

    newlayersize = (d+1)*d
    oldlayer = [t.simplex(i) for i in range(1,d+2)]
    newlayer = []

    parent = 0

    remaining = k

    while remaining > 0:

        if parent == len(oldlayer):
            newlayersize = newlayersize*d
            oldlayer = newlayer
            newlayer = []
            parent = 0

        base = t.size()
        t.newSimplices(d)
        for i in range(d):
            newlayer.append(t.simplex(base+i))

        for i in range(d):
            oldlayer[parent].join(i,t.simplex(base+d-1-i),gluings[i])

        parent = parent+1
        remaining = remaining-1

    for j in range(t.size()):
        s = t.simplex(j)
        if s.face(d-1,0).isBoundary():
            for i in range(int(d/2)):
                s.join(i,s,PermN(d+1)(i,d-1-i))

    return t

def zeroTwoVertex(t,index,face):

    s = t.simplex(index)
    adj = s.adjacentSimplex(face)
    gluing = s.adjacentGluing(face)

    s.unjoin(face)
    t.newSimplices(2)
    new1 = t.simplex(t.size()-2)
    new2 = t.simplex(t.size()-1)
    s.join(face,new1,PermN(t.dimension+1)()) # free face glued to first new simplex by identity
    new2.join(face,adj,gluing) # second new simplex glued to other side of original gluing by original perm

    for i in list(range(0,face))+list(range(face+1,t.dimension+1)):
        new1.join(i,new2,PermN(t.dimension+1)())

def dumpling_noloops(d,k):
    t = dumpling(d,k)
    for i in range(t.size()):
        if t.simplex(i).adjacentSimplex(0).index()==i:
            for j in range(int(d/2)):
                zeroTwoVertex(t,i,j)
    return t
\end{lstlisting}

\subsection*{Code to Generate $\mathcal{Y}_k$}

\begin{lstlisting}[language=python]
from regina import *

def three_spike(k):

    if k%2 == 0:
        perm=Perm5(2,1,4,3,0)
    else:
        perm=Perm5(4,1,2,3,0)

    t=Triangulation4()

    for i in range(0,4*k,4):

        t.newPentachora(4)

        for j in range(3):
            t.pentachoron(i+j).join(0,t.pentachoron(i+j),Perm5(2,1,0,3,4))
            t.pentachoron(i+j).join(1,t.pentachoron(i+j),Perm5(0,3,2,1,4))

        t.pentachoron(i).join(4,t.pentachoron(i+3),Perm5(0,1,2,4,3))
        t.pentachoron(i+1).join(4,t.pentachoron(i+3),Perm5(0,1,4,3,2))
        t.pentachoron(i+2).join(4,t.pentachoron(i+3),Perm5(0,3,2,4,1))

    for i in range(1,k):
        t.pentachoron(4*i-1).join(4,t.pentachoron(4*i+3),Perm5(2,1,4,3,0))
    t.pentachoron(4*k-1).join(4,t.pentachoron(3),perm)

    return t
\end{lstlisting}

\subsection*{Code to Generate $\mathcal{V}_k$}

\begin{lstlisting}[language=python]
from regina import *

def two_spike(k):

    t=Triangulation4()

    for i in range(0,6*k,3):
        
        t.newPentachora(3)
        
        t.pentachoron(i+1).join(0,t.pentachoron(i+1),Perm5(2,1,0,3,4))
        t.pentachoron(i+1).join(1,t.pentachoron(i+1),Perm5(0,3,2,1,4))
        t.pentachoron(i+2).join(0,t.pentachoron(i+2),Perm5(2,1,0,3,4))
        t.pentachoron(i+2).join(1,t.pentachoron(i+2),Perm5(0,3,2,1,4))
        
        t.pentachoron(i).join(1,t.pentachoron(i+1),Perm5(0,4,2,3,1))
        t.pentachoron(i).join(0,t.pentachoron(i+2),Perm5(4,1,2,3,0))

    for i in range(0,6*k,6):

        t.pentachoron(i).join(2,t.pentachoron((i+3)%(6*k)),Perm5())
        t.pentachoron(i).join(3,t.pentachoron((i+3)%(6*k)),Perm5())
        t.pentachoron(i).join(4,t.pentachoron((i-3)%(6*k)),Perm5())

    return t
\end{lstlisting}

\end{appendices}

\end{document}